\documentclass[a4paper,10pt]{article}
\usepackage{amsmath,amsthm,amssymb,amsfonts}
\usepackage{enumerate,amscd,amsxtra,MnSymbol}
\usepackage{mathrsfs}
\usepackage{color}
\usepackage[cmtip,all]{xy}
\usepackage{eucal}

\newtheorem{theorem}{Theorem}[section]
\newtheorem*{theorem*}{Theorem}
\newtheorem{lemma}[theorem]{Lemma}
\newtheorem{proposition}[theorem]{Proposition}
\newtheorem*{proposition*}{Proposition}
\newtheorem{corollary}[theorem]{Corollary}
\newtheorem*{corollary*}{Corollary}

\newtheorem{example}[theorem]{Example}
\newtheorem{definition}[theorem]{Definition}
\newtheorem{construction}[theorem]{Construction}
\newtheorem{definition*}{Definition}
\newtheorem{remark}[theorem]{Remark}
\newtheorem{remark*}{Remark}

\newtheorem{conjecture*}{Conjecture}

\newtheorem{notation}[theorem]{Notation}

\newcommand{\bE}{{\mathbf E}}

\newcommand{\bF}{{\mathbb F}}

\newcommand{\caR}{{\mathcal R}}
\newcommand{\caC}{{\mathcal C}}
\newcommand{\caD}{{\mathcal D}}
\newcommand{\caB}{{\mathcal B}}
\newcommand{\caA}{{\mathcal A}}

\newcommand{\caM}{{\mathcal M}}
\newcommand{\caN}{{\mathcal N}}
\newcommand{\caS}{{\mathcal S}}

\newcommand{\caH}{{\mathcal H}}

\newcommand{\caP}{{\mathcal P}}

\newcommand{\caU}{{\mathcal U}}

\newcommand{\rS}{{\mathrm S}}

\newcommand{\R}{{\mathrm R}}
\newcommand{\M}{{\mathrm M}}

\renewcommand{\L}{{\mathrm L}}

\newcommand{\K}{{\mathrm K}}

\newcommand{\integers}{{\mathbb Z}}

\newcommand{\op}{\mathrm{op}}
\newcommand{\dual}{\vee}

\newcommand{\Ho}{\mathsf{Ho}}

\newcommand{\B}{\mathrm{B}}

\newcommand{\s}{\mathsf{s}}

\newcommand{\A}{\mathbf{A}}

\newcommand{\G}{\mathrm{G}}

\newcommand{\Sp}{\mathrm{Sp}}

\newcommand{\f}{\mathsf{f}}

\newcommand{\colim}{\mathrm{colim}}

\newcommand{\Mod}{{\mathrm{Mod}}}

\newcommand{\X}{\mathrm{X}}

\renewcommand{\smash}{\wedge}

\newcommand{\h}{\mathsf{h}}

\renewcommand{\H}{\mathcal{H}}

\newcommand{\Sym}{\mathrm{Sym}}

\newcommand{\id}{\mathrm{id}}
\renewcommand{\u}{\mathrm{u}}

\newcommand{\Cat}{\mathsf{Cat}}

\newcommand{\tr}{\mathrm{tr}}

\newcommand{\Pic}{\mathrm{Pic}}
\newcommand{\pic}{\mathrm{pic}}

\renewcommand{\Pr}{\mathsf{Pr}}

\newcommand{\Fun}{\mathrm{Fun}}

\newcommand{\Grp}{\mathrm{Grp}}
\newcommand{\Rig}{\mathrm{Rig}}

\newcommand{\preadd}{\mathrm{preadd}}
\newcommand{\rig}{\mathrm{rig}}

\newcommand{\cocart}{\mathrm{cocart}}
\newcommand{\Fin}{\mathsf{Fin}}

\newcommand{\lax}{\mathrm{lax}}
\newcommand{\SegSpc}{\mathsf{SegSpc}}

\newcommand{\Perf}{\mathrm{Perf}}

\newcommand{\Tw}{\mathrm{Tw}}

\newcommand{\Span}{\mathrm{Span}}
\newcommand{\KR}{\mathrm{KR}}

\newcommand{\Cmon}{\mathrm{Cmon}}

\newcommand{\Calg}{\mathrm{Calg}}

\newcommand{\triv}{\mathrm{triv}}
\newcommand{\Y}{\mathrm{Y}}

\newcommand{\grp}{\mathrm{grp}}

\title{Infinity categories with duality and hermitian multiplicative infinite loop space machines}
\author{Hadrian Heine, Alejo Lopez-Avila, Markus Spitzweck}

\date{ }
\begin{document}

\maketitle

\begin{abstract}

We show that any preadditive $\infty$-category with duality gives rise to a direct sum hermitian $K$-theory spectrum. This assignment is lax symmetric monoidal, thereby producing $\bE_\infty$-ring spectra from preadditive symmetric monoidal $\infty$-categories with duality.
To have examples of preadditive symmetric monoidal $\infty$-categories with duality we show that any preadditive symmetric monoidal $\infty$-category, in which every object admits a dual, carries a canonical duality.
Moreover we classify and twist the dualities in various ways and apply our definitions for example to finitely generated projective modules
over $\bE_\infty$-ring spectra.
\end{abstract}

\tableofcontents

\section{Introduction}

$K$-theory has by now a longstanding history.
Vector bundles on topological spaces or schemes give rise to the $K_0$-functor
and also higher $K$-theory \cite{quillen.higher}.
Algebraic $K$-theory can be defined
for any stable $\infty$-category \cite{bgt} or Waldhausen ($\infty$-) category \cite{waldhausen}, \cite{barwick.wald}.
Direct sum $K$-theory of a symmetric monoidal $\infty$-category $\caC$ can be defined to be the spectrum $K(\caC)$ associated to the group
completion of the $\bE_\infty$-space $\caC^\simeq$, the maximal subspace of $\caC$ \cite{gepner-groth-nikolaus}. It was shown in
loc. cit. that the functor $K$ is in fact lax symmetric monoidal, using
a symmetric monoidal structure on the $\infty$-category of small symmetric monoidal $\infty$-categories resembling the one on $\bE_\infty$-spaces, which was defined in \cite{lydakis} for $\Gamma$-spaces and was studied in \cite{schwede.gamma}.

Instead of vector bundles one can consider vector bundles equipped with non-degenerate
symmetric or anti-symmetric bilinear forms giving rise to the Grothendieck-Witt group and
hermitian $K$-theory. Motivated by direct sum $K$-theory we consider in this article
a version of direct sum hermitian $K$-theory. Given a symmetric monoidal category $\caC$ with duality,
e.g. finitely generated projective $R$-modules $\caP(R)$ over a commutative ring $R$
with the direct sum as symmetric monoidal structure and the usual duality,
one can consider the maximal $\bE_\infty$-space in the symmetric monoidal category $\caH^{\mathrm{nd}}(\caC)$ of non-degenerate hermitian objects of $\caC$ and the spectrum associated to its group completion. This version
of hermitian $K$-theory coincides with the common definitions
in the case of $\caP(R)$ (see \cite{schlichting.giffen} and \cite[Theorem A.1]{schlichting.derived} if $2$ is invertible in $R$ and
\cite[Theorem A]{hesselholt-madsen} in general, compare also to \cite[1.6. Theoreme]{karoubi.th}).

Direct sum $K$-theory is also meaningful in the case of modules over $\bE_\infty$-ring spectra,
e.g. the direct sum $K$-theory of finitely generated projective $R$-modules for $R$ a connective $\bE_\infty$-ring spectrum
is the connective algebraic $K$-theory of $R$ (see e.g. \cite[Theorem 5]{lurie.notes281-19}).
We have a similar behavior in the case of hermitian $K$-theory (see Remark \ref{rey}).
In this article we exhibit perfect or finitely generated projective modules over an $\bE_\infty$-ring spectrum as an $\infty$-category
with duality together with the relevant symmetric monoidal structures, to which direct sum hermitian $K$-theory can be applied, thereby producing an $\bE_\infty$-ring spectrum.

Precisely, we define $\infty$-categories with duality as homotopy fixed points of the action of the group with two elements $C_2$ on the $\infty$-category $\Cat_\infty$ of small $\infty$-categories, where the non-trivial element
of $C_2$ sends an $\infty$-category $\caC$ to its opposite $\infty$-category $\caC^\op$ (Definition \ref{dua}). Similarly, we define symmetric monoidal $\infty$-categories with duality as homotopy $C_2$-fixed points of the induced $C_2$-action on the $\infty$-category of small symmetric monoidal $\infty$-categories, which are equivalently given by commutative monoid objects in the $\infty$-category of small $\infty$-categories with duality.
For these symmetric monoidal $\infty$-categories with duality we construct a hermitian $K$-theory that carries a multiplicative structure: we consider a closed symmetric monoidal structure on symmetric monoidal $\infty$-categories with duality and promote hermitian $K$-theory to a lax symmetric monoidal functor
from symmetric monoidal $\infty$-categories with duality to spectra
(Proposition \ref{h4rgrf}).
Moreover we refine hermitian $K$-theory to direct sum real $K$-theory $\KR$ that assigns to any symmetric monoidal $\infty$-category with duality a genuine $C_2$-spectrum whose underlying spectrum and fixed point spectrum are the following:
%enjoying the following properties: %by proving the following theorem:
\begin{theorem}(Theorem \ref{zum})
Group completion of non-degenerate hermitian and underlying objects
defines a functor
$$\KR: (\Cmon(\Cat_\infty), (-)^\op)^{hC_2} \to \Sp^{C_2}_{\geq0}$$
with
$$\Omega^\infty(\KR(\caC,D)^{C_2}) \simeq ((\caC^\simeq,D)^{hC_2})^\grp \hspace{2mm} \mathrm{and}\hspace{2mm} \Omega^\infty(\KR(\caC,D)^\u) \simeq (\caC^\simeq)^\grp, $$
where $(-)^\grp$ denotes group completion. The functor $\KR$ refines to a lax symmetric monoidal functor for the natural tensor
products on both sides.

%For every symmetric monoidal $\infty$-category $\caC$ with duality there are canonical equivalences \begin{enumerate}\item $\KR(\caC)^{C_2} \simeq \K_\h(\caC),$	\item $\KR(\caC)^\u \simeq \K(\caC). $\end{enumerate}
	
\end{theorem}
%whose fixed points are hermitian $K$-theory, and promote $\KR$ to a lax symmetric monoidal functor (Proposition \ref{hrf4t63}).
%Moreover we promote real $K$-theory $\KR$ to a lax symmetric monoidal functor (Proposition \ref{hrf4t63}).
As a consequence real $K$-theory produces a genuine $\bE_\infty$-$C_2$-ring spectrum from any commutative algebra with respect to the closed symmetric monoidal structure on symmetric monoidal $\infty$-categories with duality, which we call bimonoidal $\infty$-categories with duality. To give examples of bimonoidal $\infty$-categories with duality we show that any preadditive symmetric monoidal $\infty$-category uniquely extends to a bimonoidal $\infty$-categories with duality whose second tensor product is the sum, over which the first tensor product distributes (Theorem \ref{thhH}). So real $K$-theory gives a genuine $\bE_\infty$-$C_2$-ring spectrum for any preadditive symmetric monoidal $\infty$-category with duality. %, whose tensor product preserves sums in each component.

In section \ref{uhus} we 
%\ref{hjtgr54rz} we classify and twist dualities:
classify symmetric monoidal dualities: we equip every symmetric monoidal $\infty$-category, in which every object has a dual, with a canonical symmetric monoidal duality: % (Theorem \ref{thhh}).
\begin{theorem}\label{hgg}(Theorem \ref{ujkp})
Let $\caC$ be a symmetric monoidal $\infty$-category, in which every object $X$ admits a dual $X^\dual.$
There is a functorial symmetric monoidal duality $\caC^\op \to \caC, X \mapsto X^\dual$ on $\caC$.
Moreover for every tensor-invertible object $L$ of $\caC$ there is a duality $\caC^\op \to \caC, X \mapsto X^\dual\otimes L$.
	
\end{theorem}

%and deduce that the $C_2$-action on symmetric monoidal $\infty$-categories sending a symmetric monoidal $\infty$-category to its opposite restricts to the trivial $C_2$-action on rigid symmetric monoidal $\infty$-categories (Theorem \ref{thhh}). 
%To show this we prove that any rigid symmetric monoidal $\infty$-category $\caC$ carries a canonical symmetric monoidal duality sending an object $X \in \caC$ to its dual. 
%To organize the assignment $\caC \ni X \mapsto X^\vee$to a duality we introduce another model for dualities:
%we prove that a duality on an $\infty$-category $\caC$ is the same datum as a presheaf on $(\caC \times \caC)_{hC_2}$,whose restriction to $\caC \times \caC$ classifies a functor $\caC^\op \times \caC^\op \to \caS$ adjoint to a functor$\caC^\op \to \Fun(\caC^\op, \caS)$, where $\caS$ denotes the $\infty$-category of spaces, that factors through the Yoneda-embedding $ \caC \subset  \Fun(\caC^\op, \caS)$ and induces an equivalence $\caC^\op \simeq \caC$ (Theorem \ref{theo}). 
%More generally, given a tensor invertible object $\L $ of $\caC$ we construct a duality on $\caC$ sending an object $X \in \caC$ to $X^\vee \otimes \L$. 
To prove Theorem \ref{hgg} we identify dualities on an $\infty$-category $\caC$ with $C_2$-equivariant functors $\alpha: \widetilde{\caC^\op \times \caC^\op} \to \caS$ (Theorem \ref{main}), where the left hand $C_2$-action switches the factors and the right hand $C_2$-action is trivial, such that for every $\X \in \caC$ the functor $\alpha(\X,-):\caC^\op \to \caS$ is representable by some object $D(X)$ and the resulting functor $D: \caC^\op \to \caC$ is an equivalence. More generally, we introduce pro-dualites, correspondences of dualities, and prove an equivalence between pro-dualities on an $\infty$-category $\caC$ and $C_2$-equivariant functors $ \widetilde{\caC^\op \times \caC^\op} \to \caS$ (Theorem \ref{main}). An adaption of this construction appears in \cite[§ 3.1.]{calmès2024motivic} for Poincar\'e $\infty$-categories.

In section \ref{hjtgr54rz} we classify dualities on a symmetric monoidal $\infty$-category with duality $\caC$ that are $\caC$-linear and preserve the structure duality.
We define these dualities as homotopy $C_2$-fixed points for a canonical $C_2$-action on the delooping $B\caC$ of $\caC$, the $(\infty,2)$-category with connected space of objects and endomorphisms of the essentially unique object given by $\caC,$ and classify these dualities via the picard spectrum $\pic(\caC)$ of $\caC$ (Lemma \ref{ppp}) that inherits a $C_2$-actions from $\caC$ that we twist via multiplication with $-1$ indicated by $(-)_\dual:$
\begin{theorem}\label{ujj} (Theorem \ref{htewett})
%Let $\caC$ be a rigid symmetric monoidal $\infty$-category.Let $\caC$ be endowed with the duality coming from rigidity provided by Theorem \ref{thhh}.	There are canonical equivalences of grouplike $\bE_\infty$-spaces $$((B\caC)^{hC_2})^\simeq \simeq \Omega^\infty(\pic(\caC)^{BC_2} /\pic(\caC)), $$$$\{\caC\}\times_{B\caC}(B\caC)^{hC_2} \simeq \Pic(\caC)^{B C_2}, $$where the inverse of the latter sends $\L \in \Pic(\caC)^{h C_2}$ to a $\caC$-linear duality $ X \mapsto X^\dual \otimes L$ on $\caC.$
Let $(\caC,D)$ be a symmetric monoidal $\infty$-category with duality.
There is an exact sequence of spectra:
$$\pic(\caC) \to (\pic(\caC),D \circ (-)_\dual)^{h C_2} \to (\pic(\caC)[1],D)^{hC_2} \to \pic(\caC)[1]. $$
\end{theorem}
%We prove that such dualities are classified by invertible objects in $\caC$ that are homotopically fixed by the $C_2$-actionon $\caC$ twisted by the $C_2$-action inverting an object, in other words by $\Pic(\caC)^{hC_2}$, where $\Pic(\caC)$ carries the respective $C_2$-action  (Theorem \ref{htewett}).

We use the latter theorem to dedude the following corollary:
\begin{corollary}\label{ujjj} (Corollary \ref{hoy})
\label{hoy}Let $\caC$ be a rigid symmetric monoidal $\infty$-category.
There is an exact sequence of spectra:
$$\pic(\caC) \to \pic(\caC)^{B C_2} \to (\pic(\caC)[1],(-)^\dual)^{hC_2} \to \pic(\caC)[1].$$

\end{corollary}

We apply Corollary \ref{ujjj} to compute some homotopy groups of 
$\pic(\integers)[1]^{hC_2}$ and $\pic(S[\frac{1}{2}])[1]^{hC_2}$
(Proposition \ref{comp}).

%\begin{enumerate}	\item We have 	
	%$((B\Mod^\perf_{H\integers})^{hC_2})^\simeq $ and 
%	$$\pi_\ell(\pic(\integers)[1]^{hC_2})= \begin{cases} \integers/4\integers, \ell=0 \\
%		\integers/2\integers, \ell=1,2, \\
%		0, \ell> 2.\\	 
%	\end{cases}$$
	
%	\item We have $$\pi_\ell(\pic(S[\frac{1}{2}])[1]^{hC_2} )= \begin{cases} \integers/4\integers, \ell=0 \\
%		\integers/2\integers \oplus \integers, \ell=1,2.\\	 
%	\end{cases}$$
%end{enumerate}
%\end{proposition}

In section \ref{j5r5z5z} % we discuss applications e.g. to finitely generated projective modules over $\bE_\infty$-ring spectramaking use of our twisting techniques to specialize our hermitian $K$-theory to symplectic versions of hermitian $K$-theory.
we classify dualities on compact spectra. Using an argument from equivariant homotopy theory communicated to us by Niko Naumann we prove that there are uncountably many inequivalent dualities on the $\infty$-category of compact spectra (Corollary \ref{inco}). For that we prove the following proposition:
\begin{proposition}(Proposition \ref{hteetr5})
The canonical map of spectra $S \to \pic(S)$ that sends the
unit to $S[1]$ induces an injection:
\begin{equation}\label{oo}
\pi_0(S^{BC_2}) \to \pi_0(\pic(S)^{BC_2}).
\end{equation}

\end{proposition}
%In the following we prove that $\pi_0(\pic(S)^{BC_2})$ contains $ \integers \oplus \integers_2, $ where $\integers_2$ is the ring of 2-adic numbers.
By Theorem \ref{ujj} the right hand side of injection (\ref{oo}) is the set of equivalence classes of dualities on the $\infty$-category of compact spectra. By Lin's Theorem \cite[Theorem 1.1]{lin} there is an isomorphism $ \pi_0(S^{BC_2}) \cong \integers \oplus \integers_2$, where $\integers_2$ is the uncountable ring of 2-adic numbers.
This implies that there are uncountably many inequivalent dualities on the $\infty$-category of compact spectra.

In the paper \cite{spitzweck.gw} hermitian and real $K$-theory for stable $\infty$-categories
with duality is developed using an $\infty$-categorical version of the hermitian $S_\bullet$-construc\-tion
thereby obtaining a suitable theory which goes beyond direct sum hermitian $K$-theory,
and which was further developed in \cite{real.wald}. %real $K$-theory for Waldhausen $\infty$-categories with genuine duality is developed
Recently, \cite{calmès2021hermitian1}, \cite{calmès2021hermitian2}, \cite{calmès2021hermitian3} construct real $K$-theory of stable $\infty$-categories equipped with a quadratic functor so called Poincar\'e $\infty$-categories, which produces a genuine $C_2$-spectrum whose underlying spectrum is $K$-theory, whose genuine $C_2$-fixed points are hermitian $K$-theory and whose geometric fixed points are $L$-theory \cite[Corollary 4.4.14.]{calmès2021hermitian2}, and build a far reaching framework for algebraic surgery theory \cite[§ 1]{calmès2021hermitian3}. % in this context.build a far reaching framework for real $K$-theory of stable $\infty$-categories equipped with a quadratic functor so called Poincar\'e $\infty$-categories, which combines $K$-theory, hermitian $K$-theory and $L$-theory in real $K$-theory. % bridge between hermitian K-theory and L-theory and develop algebraic surgery in this context.

We give partial comparison statements Remarks \ref{rey} and Remark \ref{comp1}. 
In \cite{heine2024equivalencerealshermitian} we prove a comparison between the real K-theories of \cite{calmès2021hermitian1}, \cite{calmès2021hermitian2}, \cite{calmès2021hermitian3} and \cite{real.wald}.

Expanded versions of parts of this paper are a part of the second named author's phd thesis.

\vskip.7cm

{\bf Acknowledgements:}
We thank Niko Naumann for providing a proof of Lemma \ref{hteetr5}.
We would like to thank Hongyi Chu,
David Gepner, Jens Hornbostel, Kristian Moi, Thomas Nikolaus, Oliver R\"ondigs, Sean Tilson and Girja Tripathi for very helpful
discussions and suggestions on the subject.
Moreover we thank the referee for many helpful comments and suggestions.

\subsection*{Notation and terminology}

We fix a hierarchy of Grothendieck universes whose objects we call small, large, etc.
We call a space small, large, etc. if its set of path components and its homotopy groups are so for any choice of base point. We call an $\infty$-category small, large, etc if its maximal subspace and all its mapping spaces are small, large, respectively.

\subsubsection*{General notation}
We write 
\begin{itemize}
%\item $\Set$ for the category of small sets.
\item $\Delta$ for the category of finite, non-empty, totally ordered sets and order preserving maps, whose objects we denote by $[\mathrm{n}] = \{0 < ... < \mathrm{n}\}$ for $\mathrm{n} \geq 0$.
\item $\Fin$ for the category of finite sets and %, whose objects we denote by $\mathrm{n} = \{1, ... , \mathrm{n}\}$ for $\mathrm{n} \geq 0$.
 $\Fin_*$ for the category of finite pointed sets. %, whose objects we denote by $\langle\mathrm{n}\rangle = \{*, 1, ..., \mathrm{n}\}$ for $\mathrm{n} \geq 0$.
\item $\caS$ for the $\infty$-category of small spaces.
\item $\Sp$ for the $\infty$-category of small spectra.

\item $ \Cat_\infty$ for the $\infty$-category of small $\infty$-categories.

\item $C_2:=\mathbb{Z}/2\mathbb{Z}.$

\end{itemize}

We often indicate $\infty$-categories of large objects by $\widehat{(-)}$, for example we write $\widehat{\caS}, \widehat{\Cat}_\infty$ for the $\infty$-categories of large spaces, large $\infty$-categories, respectively.

%We denote by $\Cat_\infty^{\rc \rc} \subset \widehat{\Cat}_\infty$ the subcategory of $\infty$-categories with small colimits and small colimits preserving functors.
%By \cite{lurie.higheralgebra} Proposition 4.8.1.3. the $\infty$-category $\Cat_\infty^{\rc \rc}$ carries a closed symmetric monoidal structure such that the subcategory inclusion $\Cat_\infty^{\rc \rc} \subset \widehat{\Cat}_\infty$ is lax symmetric monoidal.

\vspace{1mm}

For any $\infty$-category $\caC$ containing objects $\A, \B$ we write
\begin{itemize}
\item $\caC(\A,\B)$ for the space of maps $\A \to \B$ in $\caC$,
\item $\caC_{/\A}$ %:= \{\A\} \times_{\Fun(\{1\},\caC)} \Fun([1],\caC)$ 
for the $\infty$-category of objects over $\A$,
\item $\Ho(\caC)$ for the homotopy category of $\caC$,
%\item $\caC^{\triangleleft}, \caC^{\triangleright}$ for the $\infty$-category arising from $\caC$ by adding an initial, final object, respectively,
\item $\caC^\simeq $ for the maximal subspace in $\caC$.
\item %Note that $\Ho(\Cat_\infty)$ is cartesian closed and for small $\infty$-categories $\caC,\caD$ we write 
$\Fun(\caC,\caD)$ for the $\infty$-category of functors $\caC \to \caD$ to any $\infty$-category $\caD.$
\end{itemize}

%\cite[Definition 2.1.1.10., Definition 4.1.3.2.]{lurie.higheralgebra}, where we drop the term symmetric, (symmetric) monoidal $\infty$-categories \cite[Definition 2.0.0.7, Definition 4.1.3.6.]{lurie.higheralgebra}, and the corresponding notions of maps of (non-symmetric) $\infty$-operads and (lax) (symmetric) monoidal functors.

%\subsubsection{Symmetric monoidal $\infty$-categories}

We call a functor $\phi:\caC \to \caD$ an embedding if it is fully faithful.
We call $\phi$ an inclusion if it induces an embedding on maximal subspaces and on mapping spaces. If $\phi:\caC \to \caD$ is an inclusion, we often write $\caC \subset \caD$
if $\phi$ is clear from the context.
For every $\infty$-category $\caC$ and subcategory $\iota: \caB' \subset \Ho(\caC)$ %of the homotopy category of $\caC$
there is a unique inclusion $\caB \subset \caC$ inducing $\iota$ on homotopy categories.
% the inclusion %of $\caB'$ into the homotopy category of $\caC.$

\subsubsection*{Symmetric monoidal $\infty$-categories}
We heavily use symmetric monoidal $\infty$-categories \cite[Definition 2.0.0.7.]{lurie.HA}.
We identify symmetric monoidal $\infty$-categories with commutative monoid objects in $\Cat_\infty$ via \cite[Example 2.4.2.4.]{lurie.HA}.
We denote symmetric monoidal $\infty$-categories in the form $(\caC,\otimes)$ and call $\caC$ the underlying $\infty$-category. We typically drop $\otimes$ from the notation unless there are several symmetric monoidal structures on $\caC$, which we need to distinguish.

\begin{itemize}
\item For every symmetric monoidal $\infty$-category $\caC$ we write
$\Calg(\caC)$ for the $\infty$-category of commutative algebras in $\caC$.

\item For every $\infty$-category $\caC$ that admits finite products
we write $\Cmon(\caC)$ for the $\infty$-category of commutative monoid objects in $\caC$
and $\Grp_{\bE_\infty}(\caC) \subset \Cmon(\caC)$ for the full subcategory of $\bE_\infty$-group objects.

\item For every symmetric monoidal $\infty$-category $\caC$ and commutative algebra $A$ in $\caC$ we write
$\Mod_A(\caC)$ for the $\infty$-category of $A$-modules in $\caC$.

\end{itemize}
There is a forgetful functor $\Calg(\caC)\to \caC$ that preserves small limits \cite[Corollary 3.2.2.5.]{lurie.HA}.

A symmetric monoidal $\infty$-category is cartesian if the tensor unit is final and $\otimes$ is the product.
By \cite[2.4.1.5., 2.4.1.7.]{lurie.HA} every $\infty$-category $\caC$ that admits finite products carries a unique cartesian symmetric monoidal structure denoted by $(\caC, \times)$ and there is a canonical equivalence
$\Calg(\caC, \times) \simeq \Cmon(\caC).$
For every symmetric monoidal $\infty$-category $\caC$ and commutative algebra $A$ in $\caC$ the $\infty$-category $\Mod_A(\caC)$ inherits a symmetric monoidal structure from $\caC$, the relative tensor product $\otimes_A$, such that the free $A$-module functor is symmetric monoidal (see \cite[Theorem 4.5.2.1.]{lurie.HA} or \cite[ Corollary 7.4.]{heine2023monadicity}).

\begin{itemize}
	
\subsubsection*{Group actions}
\item For any group $\G$ and $\infty$-category $\caC$ let $$\caC[\G]:= \Fun(B\G, \caC)$$ be the $\infty$-category of objects of $\caC$ with $\G$-action, where $B\G $ is the category with one object and automorphisms of the unique object given by $\G.$

%whose space of objects is connected and whose mapping spaces are $\G.$

Pulling back along the map $B\G \to *$ induces a functor $\triv: \caC \to \caC[\G]$ that endows an object of $\caC$ with trivial $\G$-action.

\item For $\X \in \caC[\G]$ let $\X^{h\G}:= \lim \X \in \caC$ the homotopy fixed points of the $\G$-action on $\X$ if this limit exists. 
If all $B\G$-shaped limits exist in $\caC$, the
functor $\tr: \caC \to \caC[\G]$ admits a right adjoint $(-)^{h\G}$ that sends $\X$ to $\X^{h\G}.$

\item For any $\infty$-category $\caC$ containing an object $\X \in \caC$ and any space $\K$ let $\X^\K \in \caC$ the cotensor if it exists.
For $\X \in \caC[\G]$ with trivial $\G$-action there is a canonical equivalence $\X^{h\G} \simeq \X^{B\G}$.

\item We refer to objects of $\caS[C_2]$ as spaces with $C_2$-action and to objects of $\Cat_\infty[C_2]$ as $\infty$-categories with $C_2$-action.
\end{itemize}

For any $C_2$-equivariant functor $ \caC \to \caD$ between $\infty$-categories with $C_2$-action whose underlying functor (after forgetting the $C_2$-actions) admits a right adjoint, the right adjoint refines to a $C_2$-equivariant functor
such that the unit and counit are $C_2$-equivariant natural transformations
\cite[Corollary 5.15.]{heine2023monadicity}.

\subsubsection*{Straightening}\label{noil}
We often use the $\infty$-categorical Grothendieck construction ((un)straightening)
\cite[Theorem 3.2.0.1.]{lurie.HTT}, which classifies functors $\alpha: \rS \to \Cat_\infty$ for any small $\infty$-category $\rS$ by cocartesian fibrations $\X \to \rS$ resulting in an equivalence $$ \Fun(\rS,\Cat_\infty)\simeq \Cat^\cocart_{\infty/\rS},$$
where the right hand side is the subcategory of cocartesian fibrations to $\rS$ and functors over $\rS$ preserving cocartesian lifts.
A functor $\rS \to \Cat_\infty$ takes values in spaces if and only if the corresponding cocartesian fibration is a left fibration \cite[Proposition 2.4.2.4.]{lurie.HTT}.
If $\rS$ is a space, the cocartesian fibration $\X \to \rS$ is the functor $\colim(\alpha)\to \colim(*) \simeq \rS$ induced by the unique map $\alpha \to *$ in $\Fun(\rS,\Cat_\infty)$
by \cite[Corollary 3.3.4.3.]{lurie.HTT}.
Straightening over $BC_2$ gives an equivalence $\Cat_\infty[C_2]\simeq \Cat_{\infty/BC_2}$
that sends an $\infty$-category with $C_2$-action $\caD$ to $\caD_{hC_2}$. % and sends the trivial $C_2$-action on some 

Let $\rS$ be an $\infty$-category with $C_2$-action.
A $C_2$-equivariant functor $\beta: \rS \to \Cat_\infty$, where $\Cat_\infty$ carries the trivial action, is classified by a $C_2$-equivariant cocartesian fibration $\X \to \rS$,
refining the non-equivariant (un)straightening construction:
the $C_2$-equivariant functor $\beta$ corresponds to a functor $\rS_{hC_2} \to \Cat_\infty$
classified by a cocartesian fibration $\Y \to \rS_{hC_2}$ (over $BC_2$),
which by straigtening over $BC_2$ is precisely classified by a $C_2$-equivariant functor $\X \to \rS$
whose underlying functor (after forgetting the $C_2$-action) is a cocartesian fibration \cite[Proposition 2.4.2.11.]{lurie.HTT}, which classifies the functor $\beta: \rS \to \rS_{hC_2} \to \Cat_\infty$.

%By straightening a $C_2$-equivariant functor $\rS \to \Cat_\infty$ is classified by a functor $\rS_{hC_2} \to \Cat_\infty\times BC_2 $ over $BC_2$ uniquely determined by a functor$\rS_{hC_2} \to \Cat_\infty.$ The latter is classified by a cocartesian fibration $\X \to \rS_{hC_2}$ (over $BC_2$), 

\section{Dualities} 
In this section we define $\infty$-categories with duality
and symmetric monoidal and bimonoidal $\infty$-categories with duality,
to which we apply hermitian $K$-theory in the following sections.

\subsection{$\infty$-categories with duality} 

\begin{notation}\label{hjo}
The category $\Delta$ carries a $C_2$-action
that sends $[n]$ to $[n]$ and a map $\alpha: [n] \to [m]$ to $i \mapsto m-\alpha(n-i)$,
which we denote by $(-)^\op$.

\end{notation}

\begin{construction}
The $C_2$-action $(-)^\op$ on $\Delta$ of Notation \ref{hjo} induces a $C_2$-action on presheaves $\caP(\Delta).$
There is a canonical localization $ \caP(\Delta) \rightleftarrows \Cat_{\infty} $
whose local objects are the complete Segal spaces \cite[Theorem 14.6.]{Barwick2021OnTU}.
The $C_2$-action on $\caP(\Delta)$ preserves complete Segal spaces and so restricts to a $C_2$-action on $\Cat_{\infty},$ which sends an $\infty$-category $\caC$ to its opposite $\infty$-category $\caC^\op.$
We refer to this $C_2$-action on $\Cat_\infty$ as the non-trivial $C_2$-action to distinguish it from the trivial $C_2$-action, and we write $(\Cat_\infty,(-)^\op)$ for it.
\end{construction}

\begin{definition}\label{dua}
	
An $\infty$-category with duality is a homotopy $C_2$-fixed point of $(\Cat_\infty,(-)^\op)$. %the non-trivial $C_2$-action on $ \Cat_\infty.$	
%where we use the unique non-trivial $C_2$-action on $\Cat_\infty$ of Proposition \ref{unicity}.

\end{definition}
The $\infty$-category of small $\infty$-categories with duality is
$(\Cat_\infty,(-)^\op)^{hC_2}.$
%To shorten notation we typically write $(\Cat_\infty,(-)^\op)^{hC_2}$ for $(\Cat_\infty,(-)^\op)^{hC_2}.$

\vspace{1mm}

\begin{notation}
	
We often write an $\infty$-category with duality as $(\caC, D)$, where $\caC$ is an $\infty$-category
and $D$ is a duality on $\caC$, i.e. an object in the fiber $\{\caC\}\times_{\Cat_\infty} (\Cat_\infty,(-)^\op)^{hC_2}$ of the forgetful functor $(\Cat_\infty,(-)^\op)^{hC_2} \to \Cat_\infty.$
	
\end{notation}

\begin{remark}
	
The $C_2$-action on $\Cat_\infty$ restricts to the full subcategory $\caS$ of spaces.
The restricted $C_2$-action is trivial and is the only $C_2$-action on $\caS$
since the space of automorphisms of $\caS$ is contractible. 
The $C_2$-equivariant embedding $(\caS,\triv) \subset (\Cat_\infty,(-)^\op)$ induces an %embedding %So there is a canonical equivalence $$\caS^{hC_2} \simeq \caS[C_2]$$and so a canonical %left and right adjoint 
embedding $$\caS[C_2] \subset (\Cat_\infty,(-)^\op)^{hC_2}$$ 
identifying a duality on a space with a $C_2$-action on a space.

\end{remark}

\begin{notation}
%Let $$\Cmon(\Cat_\infty)$$ the $\infty$-category of commutative monoids in $\Cat_\infty$,which we identify with small symmetric monoidal $\infty$-categories.
% and symmetric monoidal functors.
	
Let $\Cat_\infty^{\prod} \subset \Cat_\infty$ be the subcategory of small $\infty$-categories having finite products and functors preserving finite products.
	
\end{notation}

The functor $ \Cmon: \widehat{\Cat}_\infty^{\prod} \to \widehat{\Cat}_\infty, \caC \mapsto \Cmon(\caC)$ sends the non-trivial $C_2$-action on $\Cat_\infty$
to a $C_2$-action on $ \Cmon(\Cat_\infty)$.

\begin{definition}\label{dua2}
A symmetric monoidal $\infty$-categories with duality is homotopy $C_2$-fixed point of the induced $C_2$-action on $ \Cmon(\Cat_\infty),$
which we also denote by $(\Cmon(\Cat_\infty),(-)^\op).$
%where we use the unique non-trivial $C_2$-action on $\Cat_\infty$ of Proposition \ref{unicity}.
	
\end{definition}
The $\infty$-category of small symmetric monoidal $\infty$-categories with duality is
$ (\Cmon(\Cat_\infty),(-)^\op)^{hC_2}.$ 

\begin{notation}
	
We sometimes write a symmetric monoidal $\infty$-category with duality as $(\caC,\otimes, D)$, where $(\caC, \otimes)$ is a symmetric monoidal $\infty$-category
and $D$ is a symmetric monoidal duality on $\caC$, i.e. an object in the fiber $\{\caC\}\times_{\Cmon(\Cat_\infty)} (\Cmon(\Cat_{\infty}),(-)^\op)^{hC_2}$ of the forgetful functor $(\Cmon(\Cat_{\infty}),(-)^\op)^{hC_2} \to \Cmon(\Cat_\infty).$
	
\end{notation}

\begin{remark}
The functor $ \Cmon: \widehat{\Cat}_\infty^{\prod} \to \widehat{\Cat}_\infty, \caC \mapsto \Cmon(\caC)$ preserves limits as a consequence of \cite[Corollary 2.5.]{gepner-groth-nikolaus} so that there is a canonical equivalence $ (\Cmon(\Cat_\infty),(-)^\op)^{hC_2} \simeq \Cmon((\Cat_\infty,(-)^\op)^{hC_2}).$
\end{remark}

\begin{lemma}
	
Let $(\caC,\tau)$ be an $\infty$-category with $C_2$-action and $X \in \caC$
such that $\caC$ admits %If $\caC$ admits for every object $X$ 
the product $X \times \tau(X)$.
There is an object $\widetilde{X \times \tau(X)} \in (\caC,\tau)^{hC_2}$ lying over 
$X \times \tau(X)$ such that for every $\Y \in (\caC,\tau)^{hC_2}$ the induced map $$(\caC,\tau)^{hC_2}(Y, \widetilde{X \times \tau(X)}) \to \caC(Y, X \times \tau(X))
\to \caC(Y,X)$$ induced by projection to the first factor is an equivalence.

%the forgetful functor $\caC^{hC_2}\to \caC$ admits a right adjoint that sends $X$ to$X \times \tau(X)$.

%\item If $\caC$ admits for every object $X$ the product $X \prod \tau(X)$,the forgetful functor $\caC^{hC_2}\to \caC$ admits a left adjoint that sends $X$ to$X \prod \tau(X)$.
	
\end{lemma}

\begin{proof}The $C_2$-action on $\caC$ is classified via straightening by a functor $\phi: \caC_{hC_2} \to BC_2$ by \cite[Corollary 3.3.4.3.]{lurie.HTT}. By \cite[Corollary 3.3.3.2.]{lurie.HTT} there is a canonical equivalence $ \caC^{hC_2} \simeq \Fun_{BC_2}(BC_2,\caC_{hC_2})$ 
and the forgetful functor $\caC^{hC_2}\to \caC$ identifies with the functor
$\kappa: \Fun_{BC_2}(BC_2,\caC_{hC_2}) \to \caC$ evaluating at the unique object of $BC_2$.
Hence the object $\widetilde{X \times \tau(X)} \in (\caC,\tau)^{hC_2}$ with the claimed properties is the right Kan extension relative to $BC_2$ of $X:*\to \caC$ along
the functor $* \to BC_2$. By \cite[Lemma 4.3.2.13.]{lurie.HTT} this right Kan extension exists
and sends the unique object of $BC_2$ to $X \times \tau(X)$ if 
%there is an object $\widetilde{X \times \tau(X)} \in (\caC,\tau)^{hC_2}$ with if the functor $\kappa$ admits a right adjoint if %for every $X \in \caC$
the constant functor $* \times_{BC_2} * \simeq C_2 \to \caC_{hC_2}$ with value $X$
admits a $\phi$-limit. %, and the right adjoint sends $X$ to this $\phi$-limit.
By \cite[Proposition 4.3.1.9.]{lurie.HTT} this $\phi$-limit %of the constant functor $* \times_{BC_2} * \simeq C_2 \to  \caC_{hC_2}$ with value $X$ 
is the limit of the functor $C_2 \to \caC$ corresponding to the pair $(X,\tau(X))$, which is $X \times \tau(X).$
	
\end{proof}
\begin{corollary}
The forgetful functor $(\Cat_\infty,(-)^\op)^{hC_2}\to \Cat_\infty$ admits a right adjoint that sends $\caC$ to $\caC \times \caC^\op$.
	
\end{corollary}

\begin{lemma}
Let $(\caC,\tau)$ be an $\infty$-category with $C_2$-action and $X\in (\caC,\tau)^{hC_2}$ such that $\caC$ admits the product $X \times \tau(X)$.
There is a canonical $C_2$-equivariant equivalence $$\widetilde{X \times \tau(X)} \simeq \widetilde{X\times X}.$$

\end{lemma}

\begin{notation}
Let $\caC$ be an $\infty$-category and $X\in \caC$ such that $\caC$ admits the product $X \times X$.
We write $\widetilde{X \times X} $ for $\widetilde{X\times \triv(X)}$.
	
\end{notation}

\begin{proof}
Projection to the first factor $X \times \tau(X) \to X$ is adjoint to a map $\rho:\widetilde{X \times \tau(X)} \to \widetilde{X\times X}$ in $\caC^{hC_2}$ lying over the equivalence $\id\times \kappa: X\times X \simeq X \times \tau(X)$, where $\kappa: X \simeq \tau(X)$ is part of the fixed point datum on $X$. So $\rho$ is an equivalence.
\end{proof}

\subsection{Bimonoidal $\infty$-cate\-gories with duality}

%Next we define bimonoidal $\infty$-categories with duality, which are $\infty$-categories endowed with two symmetric monoidal structures and a duality compatible with each other.
To define bimonoidal $\infty$-categories with duality
we define a presentably symmetric monoidal structure on $\Cmon(\Cat_\infty)$
compatible with the $C_2$-action.

%the following symmetric monoidal localizations $\caL$ on $\Pr^L$ from \cite{gepner-groth-nikolaus}:$$ \caC \mapsto \Cmon(\caC), \; \caC \mapsto \Grp(\caC), \ \caC \mapsto \Sp(\caC)$$with local objects the preadditive, additive, stable presentable $\infty$-categories, respectively.
%Being symmetric monoidal the localization $\caL$gives rise to a localization 
%\begin{equation}\label{loc}\Calg(\Pr^L)[C_2] \overset{\caL}{\longrightarrow} \Calg(\Pr^L)[C_2]\end{equation}on commutative algebras with $C_2$-action.So we get $\caL(\Cat_\infty) \in \Calg(\Pr^L)[C_2]$and especially $$\Cmon(\Cat_\infty) \in \Calg(\Pr^L)[C_2].$$
%Now we are able to define bimonoidal $\infty$-categories with duality.
%There is a functor $$ \Calg(\Pr^L) \to \Pr^L, \ \caC \mapsto \Calg(\caC)$$that yiels a functor $\Calg(\Pr^L)[C_2] \to \Pr^L[C_2].$
%We define a functor $$\Rig: \Calg(\Pr^L)[C_2] \xrightarrow{\Cmon} \Calg(\Pr^L)[C_2] \xrightarrow{\Calg} \Pr^L[C_2].$$

%We call objects of $\Rig(\Cat_\infty)$ bimonoidal $\infty$-categories and objects of $\Rig(\caS)$ bimonoidal spaces.The embedding $\caS \subset \Cat_\infty$ in$\Calg(\Pr^L)[C_2]$ yields an embedding$\Rig(\caS) \subset \Rig(\Cat_\infty)$in $\Pr^L[C_2].$

%The $\infty$-category $\Calg(\Pr^L)$ has small limits so that we can form $$\Rig(\Cat_\infty)^{hC_2} \in \Calg(\Pr^L).$$

\begin{definition}
An $\infty$-category $\caC$ is preadditive if it has a zero object and for any objects $\X, \mathrm{Y} \in \caC$ the canonical morphism $\X \coprod \mathrm{Y} \to \X\times \mathrm{Y}$ is an equivalence.
\end{definition}

\begin{notation}
Given two objects $\X,\Y$ of a preadditive $\infty$-category we write $\X \oplus \Y$ for
$\X \coprod \Y \simeq \X \times \Y.$
%We identify the finite coproduct and finite product in any preadditive $\infty$-category and refer to it as the sum, denoted by $\oplus.$
\end{notation}

\begin{definition}
	
A symmetric monoidal $\infty$-category is preadditive if the underlying $\infty$-category is preadditive and the tensor product preserves finite products component-wise.
\end{definition}
\begin{notation}
Let $\Cat_\infty^\preadd \subset \Cat_\infty^{\prod}$ be the full subcategory
of preadditive $\infty$-categories and $\Cmon(\Cat_\infty)_\preadd \subset \Cmon(\Cat_\infty)$ the subcategory of preadditive symmetric monoidal $\infty$-categories and symmetric monoidal functors preserving finite products.
	
\end{notation}

The next proposition is \cite[Proposition 2.3., Corollary 2.4., Corollary 2.5.,  Corollary 4.9.]{gepner-groth-nikolaus}: % imply the following lemmas:

%We have the following important lemmas:

\begin{proposition}\label{leomor}
Let $\caC$ be a preadditive $\infty$-category and $\caD$ an $\infty$-category that has finite products.
\begin{enumerate}
\item The $\infty$-category $\Cmon(\caD)$ is preadditive. 

\item The forgetful functor $\Cmon(\caD) \to \caD$ is an equivalence if and only if $\caD$ is preadditive.

\item The $\infty$-categories $\Fun^{\prod}(\caC,\caD), \Fun^{\prod}(\caD,\caC)$ are preadditive.

\item %For every preadditive $\infty$-category $\caC$ 
The forgetful functor $\Cmon(\caD) \to \caD$ induces an equivalence $$\Fun^{\prod}(\caC,\Cmon(\caD)) \to \Fun^{\prod}(\caC,\caD),$$
where ${\prod}$ indicates functors preserving finite products.

\item If $\caC, \caD$ are also presentable, %Let $\caD$ be a presentable $\infty$-category and $\caC$ a presentable preadditive $\infty$-category. 
the free functor $\caD \to \Cmon(\caD)$ induces an equivalence 
$$\Fun^{\L}(\Cmon(\caD), \caC) \to \Fun^{\L}(\caD,\caC),$$
where $\L, \R$ indicate left, right adjoint functors, respectively.
\end{enumerate}

\end{proposition}

%\begin{proof}This follows from  \cite[Proposition 2.4.1.7.]{lurie.HA}, \cite[Proposition 3.2.4.7.]{lurie.HA} and \cite[Proposition 2.4.3.9.]{lurie.HA}.\end{proof}

%\begin{lemma}\label{lemm}Let $\caD$ be an $\infty$-category having finite products and $\caC$ a preadditive$\infty$-category. The forgetful functor $\Cmon(\caD) \to \caD$ induces an equivalence $$\rho: \Fun^{\prod}(\caC,\Cmon(\caD)) \to \Fun^{\prod}(\caC,\caD),$$where ${\prod}$ indicates functors preserving finite products.\end{lemma}

%\begin{proof}By Lemma \ref{leomor} the forgetful functor $\Cmon(\caC) \to \caC$ is an equivalence.#The functor $\rho$ is inverse to the functor$$\Fun^{\prod}(\caC,\caD) \to \Fun^{\prod}(\Cmon(\caC),\Cmon(\caD)) \simeq \Fun^{\prod}(\caC,\Cmon(\caD)).$$	\end{proof}

%\begin{corollary}\label{cortp}\end{corollary}

%\begin{proof}By Lemma \ref{lemm} the forgetful functor $\Cmon(\caD) \to \caD$ induces an equivalence $$\Fun^{\L}(\Cmon(\caD), \caC)\simeq \Fun^{\R}(\caC,\Cmon(\caD))^\op \to \Fun^{\R}(\caC,\caD)^\op \simeq\Fun^{\L}(\caD,\caC).$$	\end{proof}
\begin{notation}
Let $\Pr^L$ be the $\infty$-category of presentable $\infty$-categories and left adjoint functors.\end{notation}
By \cite[Corollary 4.8.1.4, Proposition 4.8.1.17.]{lurie.HA} there is a closed symmetric monoidal structure on $\Pr^L$,
where the internal hom of $\caC,\caD \in \Pr^L$ is the $\infty$-category $\Fun^L(\caC,\caD)$ of left adjoint functors $\caC \to \caD$.

By Proposition \ref{leomor} %gives the following corollary:
(1), (2), (4) there is a localization on $\Pr^L$ sending $\caC$ to $\Cmon(\caC)$
whose local objects are the preadditive $\infty$-categories
and whose unit $\caC \to \Cmon(\caC)$ for $\caC \in \Pr^L$ is the free functor.
By \cite[Theorem 4.6.]{gepner-groth-nikolaus} this localization is symmetric monoidal.
%\begin{proof}Let $\caC$ be an $\infty$-category having finite products and $\caD$ an arbitrary $\infty$-category.The equivalence $\Fun(\caD^\op, \Fun(\Fin_*,\caC)) \simeq \Fun(\Fin_*, \Fun(\caD^\op,\caC))$restricts to an equivalence $\Fun(\caD^\op, \Cmon(\caC)) \simeq \Cmon(\Fun(\caD^\op,\caC)),$which restricts to an equivalence $$\Fun^{\lim} (\caD^\op, \Cmon(\caC)) \simeq \Cmon(\Fun^{\lim} (\caD^\op,\caC)),$$where $\lim$ indicates the full subcategories of small limits preserving functors. By \cite[Proposition 4.8.1.17.]{lurie.HA} there is a canonical equivalence$\caD \otimes \caC \simeq \Fun^{\lim}(\caD^\op,\caC)$and so a canonical equivalence 	$\caD\otimes \Cmon(\caC) \simeq \Cmon(\caD \otimes \caC).$\end{proof}
This implies that the localization functor $\Cmon:\Pr^L \to \Pr^L$ is lax symmetric monoidal and so sends any presentably symmetric monoidal $\infty$-category $\caC$ (with $C_2$-action) to a presentably symmetric monoidal structure on $\Cmon(\caC)$ (with $C_2$-action),
which we denote by $$(\Cmon(\caC), \otimes).$$
In particular, for every presentably symmetric monoidal $\infty$-category $\caC$ with $C_2$-action the $\infty$-category $\Cmon(\caC)^{hC_2}$
carries a presentably symmetric monoidal structure.

\begin{notation}
	
For every presentably symmetric monoidal $\infty$-category $\caC$ let $$\Rig(\caC):= \Calg(\Cmon(\caC),\otimes).$$
\end{notation}

\begin{remark}
	
An object $\X \in \Rig(\caC)$ roughly consists of an object $\X \in \caC$
equipped  with two coherently unital, associative and commutative operations $+,\cdot: \X \times \X \to \X$ such that $\cdot$ coherently distributes over $+.$	
\end{remark}

\begin{remark}
By functoriality for every presentably symmetric monoidal $\infty$-category with $C_2$-action $(\caC,\tau)$ there is an $\infty$-category with $C_2$-action $(\Rig(\caC),\tau)$
and canonical equivalences $$(\Rig(\caC),\tau)^{hC_2} \simeq \Calg((\Cmon(\caC),\tau)^{hC_2},\otimes) \simeq \Rig((\caC,\tau)^{hC_2}).$$
		
\end{remark}

%\begin{remark}
%Let $\caC$ be a presentably symmetric monoidal $\infty$-category.
%The symmetric monoidal structure on $\Cmon(\caC)$ is closed, wherethe internal hom of two commutative algebras $A,B$ of $\caC$
%The unit $\caC \to \Cmon(\caC)$ of the localization of Corollary \ref{locc} is
%symmetric monoidal and is the free functor.
%Thus the tensor unit of $\Cmon(\caC)$ is free on the tensor unit of $\caC$.A commutative algebra $A$ in $\Cmon(\caC)$ has a multiplication map$\mu: A \otimes A \to A$ in $\Cmon(\caC) $ corresponding to a morphism $\xi: A \to \mathrm{Mor}_{\Cmon(\caC)}(A,A) $ in $\Cmon(\caC)$ to the internal hom of $\Cmon(\caC)$.By the freeness of the tensor unit the morphism $\xi$ induces a map of spaces $\caC(1_\caC,A) \to {\Cmon(\caC)}(A,A) $ that sends a morphism $Z: 1_\caC \to A$ to the morphism $ \mu(-,Z) \simeq \mu(Z,-): A \to A.$Thus the multiplication map $\mu:A \otimes A \to A$ is a map of commutative algebras in both variables.\end{remark}

\begin{definition}
	
\begin{itemize}
\item A bimonoidal $\infty$-category is an object of $\Rig(\Cat_\infty) $.% commutative algebra with respect to $(\Cmon(\Cat_\infty), \otimes)$.
\item A bimonoidal $\infty$-category with duality is an object of $ (\Rig(\Cat_\infty),(-)^\op)^{hC_2}.$ % commutative algebra with respect to $((\Cmon(\Cat_\infty),(-)^\op)^{hC_2}, \otimes).$
\end{itemize}
\end{definition}
We often denote a bimonoidal $\infty$-category by a triple $(\caC, \boxtimes,\otimes),$
where $\otimes $ distributes over $\boxtimes.$

%\begin{notation}Let $$\Rig(\Cat_\infty):=\Calg(\Cmon(\Cat_\infty),\otimes).$$	\end{notation}

\begin{definition}
A bimonoidal $\infty$-category $(\caC, \boxtimes,\otimes)$ (with duality) is preadditive if $\caC$ is preadditive and $(\caC, \boxtimes)$ is cartesian.
%Similarly, we define preadditive bimonoidal $\infty$-categories with duality.
\end{definition}

\begin{lemma}\label{leem}
The $\infty$-category $\Cat_\infty^{\prod}$
is preadditive.
\end{lemma}

\begin{proof}
	
The $\infty$-category $\Cat_\infty^{\prod}$ admits finite products (in fact all small limits) that are preserved by the inclusion $\Cat_\infty^{\prod} \subset \Cat_\infty$. Since every finite products preserving functor preserves the final object, the $\infty$-category $\Cat_\infty^{\prod}$ admits a zero object.
For every $\caC \in \Cat_\infty^{\prod}$ let $\mu_\caC:\caC \times \caC \to \caC$ the functor assigning the two-fold product.
By \cite[Proposition 2.4.3.19.]{lurie.HA} the statement follows from the fact that the functor 
$\caC \times 0 \to \caC \times \caC \xrightarrow{\mu_\caC}\caC $ is the identity, the functor
$(\caC \times \caC) \times (\caD \times \caD) \simeq (\caC \times \caD)\times(\caC \times \caD) \xrightarrow{ \mu_{\caC \times \caD}} \caC \times \caD$ is the functor $\mu_\caC \times \mu_\caD$ and for every map
$\f: \caC \to \caD $ in $\Cat_\infty^{\prod}$ we have $\f \circ \mu_\caC = \mu_\caD \times (\f \times \f).$

%This follows immediately from ...Let $\caB, \caC, \caD \in \Cat_\infty^{\prod}$ and $\Fun^{\prod}(\caC, \caD) \subset \Fun(\caC, \caD)$ the full subcategory of  functors preserving finite products.
% whose maximal subspace is the mapping space in  $\Cat_\infty^{\prod}$ from $\caC$ to $\caD$.Let $\ast$ the contractible $\infty$-category, which is the final object of $\Cat_\infty^{\prod}$.The $\infty$-category $\Fun^{\prod}(\ast, \caD) $ is contractible so that $\ast$ is a zero object of $ \Cat_\infty^{\prod}$.
%Using the zero object we obtain canonical functors $\caB \to \caB \times \caC, \caC \to \caB \times \caC$.
%that are the identity respectively the zero morphism on each factor.
%We need to see that the canonical functor $$\psi_\caD: \Fun^{\prod}(\caB \times \caC, \caD) \to \Fun^{\prod}(\caB, \caD) \times \Fun^{\prod}(\caC, \caD) $$ is an equivalence.For any $\infty$-category $\K$ the functor $\Fun(\K, \psi_\caD)$ identifies with the functor $\psi_{\Fun(\K,\caD)}$.So by Yoneda we may reduce to check that $\psi$ induces a bijection on equivalence classes.An inverse of $\psi_\caD$ is given by the functor$$ \Fun^{\prod}(\caB, \caD) \times \Fun^{\prod}(\caC, \caD) \to\Fun^{\prod}(\caB \times \caC, \caD \times \caD) \to \Fun^{\prod}(\caB \times \caC, \caD).$$
	
\end{proof}

\begin{corollary}\label{proo2}
The full subcategory
$\Cat_\infty^\preadd \subset \Cat_\infty^{\prod}$ is preadditive.	
	
%The forgetful functor $$(\Cat_\infty^{\prod}) \to \Cat_\infty^{\prod} $$ is an equivalence.	
\end{corollary}

\begin{proof}
This follows since $\Cat_\infty^\preadd \subset \Cat_\infty^{\prod}$ is closed under finite products.	
	
\end{proof}

By the latter corollary the forgetful functors $$  \Cmon(\Cat_\infty^{\prod}) \to \Cat_\infty^{\prod},  \Cmon(\Cat_\infty^\preadd) \to \Cat_\infty^\preadd $$ are equivalences.

\begin{remark}
The inclusion $\Cat_\infty^{\prod} \subset \Cat_\infty$ preserves
finite products and so induces an inclusion 
\begin{equation}\label{klo}
\xi: \Cat_\infty^{\prod}\simeq \Cmon(\Cat_\infty^{\prod}) \subset \Cmon(\Cat_\infty).
\end{equation}

The inclusion $\xi$ is fully faithful and by an Eckmann-Hilton argument the essential image precisely consists of the cartesian symmetric monoidal $\infty$-categories.

\end{remark}
\begin{remark}\label{zulm}
	
By \cite[Corollary 4.8.1.4]{lurie.HA} the $\infty$-category $\Cat_\infty^{\prod}$ carries a presentably symmetric monoidal structure %with internal hom $\Fun^{\prod}$ 
such that the inclusion $\Cat_\infty^{\prod} \subset \Cat_\infty$ admits a symmetric monoidal left adjoint. Applying the lax symmetric monoidal localization functor
$\Cmon:\Pr^L \to \Pr^L$ we find that the embedding (\ref{klo})
%$\xi: \Cat_\infty^{\prod} \simeq \Cmon(\Cat_\infty^{\prod}) \subset \Cmon(\Cat_\infty)$
admits a symmetric monoidal left adjoint and so refines to a lax symmetric monoidal embedding
$$ (\Cat_\infty^{\prod},\otimes) \subset (\Cmon(\Cat_\infty),\otimes).$$
\end{remark}
\begin{remark}\label{hhhp}
The lax symmetric monoidal inclusion $(\Cat_\infty^{\prod}, \otimes) \subset (\Cat_\infty, \times)$
gives rise to an inclusion $\Calg(\Cat_\infty^{\prod},\otimes) \subset \Calg(\Cat_\infty,\times)\simeq \Cmon(\Cat_\infty)$
that exhibits $\Calg(\Cat_\infty^{\prod},\otimes)$ as the subcategory of $\Cmon(\Cat_\infty)$
of symmetric monoidal $\infty$-categories that admit finite products and whose tensor product preserves finite products component-wise and symmetric monoidal functors preserving finite products.	
	
\end{remark}

We have the following proposition:
\begin{proposition}\label{thhay}
	
There is a presentably symmetric monoidal structure on the $\infty$-category $\Cat_\infty^\preadd$
such that the embedding $$\Cat_\infty^\preadd \subset \Cat_\infty^{\prod}$$
%\Cmon(\Cat_\infty^\preadd) \subset \Cmon(\Cat_\infty) $$
admits a symmetric monoidal left adjoint.
	
\end{proposition}

\begin{remark}
Because the embedding (\ref{klo}) admits a symmetric monoidal left adjoint by Remark \ref{zulm}, by Proposition \ref{thhay} the embedding $$\Cat_\infty^\preadd \simeq \Cmon(\Cat_\infty^\preadd) \subset \Cmon(\Cat_\infty) $$ admits a symmetric monoidal left adjoint.
As $\Cat_\infty^\preadd$ is closed in $ \Cmon(\Cat_\infty)$ under the $C_2$-action,
we find that the latter embedding admits a $C_2$-equivariant symmetric monoidal left adjoint.
\end{remark}

%The $\infty$-categories $\Cat_\infty^\preadd, \Cat_\infty^{\prod}$ are presentably symmetric monoidal $\infty$-categories and the embedding $$\Cat_\infty^\preadd \subset \Cat_\infty^{\prod}$$ admits a symmetric monoidal left adjoint.

\begin{proof}[Proof of Proposition \ref{thhay}]

Let $\Cat_\infty^{\coprod} \subset \Cat_\infty$ be the subcategory of small $\infty$-categories having finite coproducts and functors preserving finite coproducts. Taking the opposite $\infty$-category restricts to an equivalence
$\Cat_\infty^{\prod} \simeq \Cat_\infty^{\coprod}$, which further restricts to an
autoequivalence of $ \Cat_\infty^\preadd.$
Consequently, it is enough to prove that the embedding $ \Cat_\infty^\preadd \subset \Cat_\infty^{\coprod}$ admits a symmetric monoidal left adjoint, which is notationally more convenient.
For any small $\infty$-category $\caC$ that admits finite coproducts let $\caC'$ be the essential image of the composition $\caC \to \Fun^{\prod}(\caC^\op, \caS) \to \Cmon(\Fun^{\prod}(\caC^\op, \caS)) \simeq \Fun^{\prod}(\caC^\op, \Cmon(\caS)) $ of the Yoneda-embedding and the free functor.
Since the  Yoneda-embedding $\caC \to \Fun^{\prod}(\caC^\op, \caS)$ preserves finite coproducts \cite[Proposition 5.5.8.10.]{lurie.HTT}, the full subcategory $\caC'$ is closed in the preadditive $\infty$-category $\Cmon(\Fun^{\prod}(\caC^\op, \caS))$ under finite coproducts and therefore is itself preadditive.
We will prove that the induced functor $\caC \to \caC'$ induces an equivalence $\kappa_\caD: \Fun^{\coprod}(\caC', \caD) \to \Fun^{\coprod}(\caC, \caD).$
The functor $\kappa_\caD$ is the pullback of the functor $\kappa_{\Fun^{\prod}(\caD^\op,\caS)}$
along the Yoneda-embedding $\caD \subset \Fun^{\prod}(\caD^\op,\caS) \simeq \Fun(\caD^\op,\Cmon(\caS))$.
Since $\Fun^{\prod}(\caC^\op, \caS)$ is preadditive and presentable \cite[Proposition 5.5.8.10.]{lurie.HTT},
we can therefore assume that $\caD$ is a presentable preadditive $\infty$-category
when proving that $\kappa_\caD$ is an equivalence.

The $\infty$-category $\Fun^{\prod}(\caC^\op, \caS)$ is generated by $\caC$ under small sifted colimits \cite[Lemma 5.5.8.14.]{lurie.HTT}. Since $\Cmon(\Fun^{\prod}(\caC^\op, \caS)) $ is generated by
$\Fun^{\prod}(\caC^\op, \caS)$ under small sifted colimits, $\Cmon(\Fun^{\prod}(\caC^\op, \caS)) $ is generated by $\caC'$ under small sifted colimits.
Every object of $\caC$ is compact and projective in $\Fun^{\prod}(\caC^\op, \caS)$ by \cite[Proposition 5.5.8.10.4.]{lurie.HTT} so that every object of $\caC'$ is compact and projective in $\Fun^{\prod}(\caC^\op, \Cmon(\caS))$. By \cite[Proposition 5.5.8.22.]{lurie.HTT} this guarantees that the embedding $\caC' \subset \Fun^{\prod}(\caC^\op, \Cmon(\caS))$ induces an equivalence $\Fun^{\prod}(\caC'^\op,\caS) \simeq \Fun^{\prod}(\caC^\op, \Cmon(\caS)). $

For every small $\infty$-category $\caB$ that admits finite coproducts the induced functor $\Fun^\L(\Fun^{\prod}(\caB^\op,\caS),\caD) \to \Fun^{\coprod}(\caB,\caD)$ is an equivalence \cite[Proposition 5.5.8.15.]{lurie.HTT}.
%For any presentable $\infty$-category $\caE$ 
%The full subcategory $\Fun^{\prod}(\caC^\op,\caS) \subset \Fun(\caC^\op,\caS)$ is an accessible localization with respect to the set of morphisms $\colim(\y \circ F) \to \y(\colim(F))$ for any finite set $\K$ and functor $F:\K \to \caC,$ where $\y$ is the Yoneda-embedding of $\caC.$Hence $\Fun^{\prod}(\caC^\op,\caS)$ is presentable and 
%A left adjoint functor $\caE \to \caE'$ yields a left adjoint functor$\Fun(\caC^\op,\caE) \to \Fun(\caC^\op,\caE')$ that preserves local equivalences and so induces a left adjoint functor $\Fun^{\prod}(\caC^\op,\caE) \to \Fun^{\prod}(\caC^\op,\caE')$. The free functor $\caS \to \Cmon(\caS)$ induces a left adjoint functor $\Fun^{\prod}(\caC^\op, \caS) \to \Fun^{\prod}(\caC^\op, \Cmon(\caS))$ that restricts to a functor $\caC \to \caC'$. 
%By Lemma \ref{lemm} the induced functor $\Fun^{\prod}(\caD^\op,\Cmon(\caS)) \to \Fun^{\prod}(\caD^\op,\caS)$ is an equivalence.Since $\Cmon(\caS)$ is preadditive (Lemma \ref{leomor}), also$\Fun(\caD^\op,\Cmon(\caS))$ and so $\Fun^{\prod}(\caD^\op,\Cmon(\caS))$ are preadditive. 
So in the following commutative square the vertical functors are equivalences:
$$\xymatrix{\Fun^{\L}(\Fun^{\prod}(\caC^\op,\Cmon(\caS)), \caD) \ar[r]^\simeq \ar[d]^\simeq &\Fun^{\L}(\Fun^{\prod}(\caC^\op,\caS), \caD)  \ar[d]^\simeq \\
\Fun^{\coprod}(\caC', \caD)\ar[r] & \Fun^{\coprod}(\caC, \caD).}$$

\end{proof}

\begin{remark}
By Proposition \ref{thhay} the $\infty$-category $\Cat_{\infty}^\preadd$ carries a presentably symmetric monoidal structure and there is a lax symmetric monoidal embedding $(\Cat_{\infty}^\preadd, \otimes) \subset (\Cat_\infty^{\prod}, \otimes)$ right adjoint to the symmetric monoidal localization functor.
The latter gives rise to an inclusion $\Calg(\Cat_\infty^{\preadd},\otimes) \subset \Calg(\Cat_\infty^{\prod},\otimes) \subset \Cmon(\Cat_\infty)$
that exhibits $\Calg(\Cat_\infty^{\preadd},\otimes)$ as the full subcategory %of $\Calg(\Cat_\infty^{\prod},\otimes)$ 
spanned by the preadditive symmetric monoidal $\infty$-categories.
\end{remark}
Proposition \ref{thhay} gives the following corollary:

\begin{corollary}\label{thh}
	
There is a canonical localization $$((\Cmon(\Cat_\infty),(-)^\op)^{hC_2}, \otimes) \to ((\Cat_\infty^\preadd,(-)^\op)^{hC_2}, \otimes).$$
\end{corollary}

\begin{proof}
	
The $C_2$-equivariant inclusion $ \Cat_\infty^\preadd \subset \Cat_\infty$ yields an inclusion $$(\Cat_\infty^\preadd,(-)^\op)^{hC_2} \simeq (\Cmon(\Cat_\infty^\preadd),(-)^\op)^{hC_2} \subset (\Cmon(\Cat_\infty),(-)^\op)^{hC_2}$$ whose image are the preadditive symmetric monoidal $\infty$-categories with duality and additive symmetric monoidal functors preserving the duality.
\end{proof}
Passing to commutative algebras we obtain the following theorem:

%\Calg((\Cmon(\Cat_\infty),(-)^\op)^{hC_2}, \otimes)  \Calg((\Cat_\infty^\preadd)^{hC_2}, \otimes)

\begin{theorem}\label{thhH}
	
There is a canonical localization $$(\Rig(\Cat_\infty),(-)^\op)^{hC_2} \to (\Cmon(\Cat_\infty)_\preadd,(-)^\op)^{hC_2}$$
whose fully faithful right adjoint identifies preadditive symmetric monoidal $\infty$-categories with duality $(\caC,\otimes, D) $ with preadditive bimonoidal $\infty$-categories with duality $(\caC,\otimes, \boxtimes, D) $, where $\boxtimes=\oplus$ is the sum.
	
\end{theorem}

\section{Hermitian objects}

Next we define hermitian objects to construct hermitian $K$-theory. 

\label{hermit}
\begin{notation}
Let $e: \Delta \to \Delta$ be the functor $  [n] \mapsto [n] * [n]^\op \simeq [2n+1] $, which we call edgewise subdivision.
\end{notation}

\begin{remark}
The functor $e$ is a $C_2$-equivariant functor $(\Delta,\triv) \to (\Delta, (-)^\op)$. %if the source carries the trivial action and the target carries the unique non-trivial action.
\end{remark}
\begin{notation}
By functoriality of taking presheaves the $C_2$-equivariant functor $(\Delta,\triv) \to (\Delta, (-)^\op)$ induces a $C_2$-equivariant functor
$$\Tw:= e^\ast : (\caP(\Delta),(-)^\op) \to (\caP(\Delta),\triv),$$ %is $C_2$-equivariant.
%if the target carries the trivial action and the source carries the induced non-trivial action.
which gives on homotopy $C_2$-fixed points a functor $ (\caP(\Delta),(-)^\op)^{hC_2} \to \caP(\Delta)[C_2],$ also denoted by $\Tw$. 
\end{notation}

%\begin{notation}Let $\H$ be the composition $$ \caP(\Delta)^{hC_2} \xrightarrow{\Tw } \caP(\Delta)[C_2]\xrightarrow{(-)^{hC_2} } \caP(\Delta). $$

By the next lemma the functor $\Tw: (\caP(\Delta),(-)^\op)^{hC_2} \to \caP(\Delta)[C_2]$ restricts to a functor 
$ (\Cat_\infty,(-)^\op)^{hC_2} \to \Cat_\infty[C_2]$ denoted by the same name, which sends an $\infty$-category with duality $(\caC,D)$ to an $\infty$-category with $C_2$-action $(\Tw(\caC),D).$
This $C_2$-action sends a morphism $f:A \to B$ in $\caC$ to the morphism
$D(f): D(B) \to D(A)$ in $\caC.$

\begin{lemma}
For every (complete) Segal space $\caC$ the simplicial space $\Tw(\caC)$ is a (complete) Segal space.
	
\end{lemma}

\begin{proof}
To verify the Segal condition it is enough to observe that the edgewise subdivision
$ \Delta \to \Delta, [n] \mapsto [n] * [n]^\op $ is a co-Segal object:
for every $[n] \in \Delta$ the canonical maps
$[1] \simeq \{i-1,i\} \subset [n],$ $1 \leq i \leq n$,
induce an isomorphism
$$ [1] * [1]^\op \coprod_{[0]*[0]^\op} ...\coprod_{[0]*[0]^\op} [1] * [1]^\op \cong [3] \coprod_{[1]} ... \coprod_{[1]} [3] \to [n]* [n]^\op \cong [2n+1].$$
Note that this colimit is preserved by the embedding $\Delta \subset \SegSpc$
into the full subcategory of $\caP(\Delta)$ spanned by the Segal spaces
by design of the latter.

Next we prove completeness. The unique map $[1] \to [0]$ induces a map $$\rho: \caC_{[0]*[0]^\op} \simeq \caC_{[1]} \simeq \caC_{[0]} \times_{\caC_{[0]}} \caC_{[1]} \times_{\caC_{[0]}} \caC_{[0]} \to \caC_{[1]*[1]^\op} $$$$\simeq \caC_{[3]} \simeq \caC_{[1]} \times_{\caC_{[0]}} \caC_{[1]} \times_{\caC_{[0]}} \caC_{[1]},$$
which is an embedding if the map ${\caC_{[0]}} \to \caC_{[1]}$ induced by $[1] \to [0]$ 
is an embedding. 
So if $\caC$ is a complete Segal space, the map $\rho$ is an embedding of spaces.
The essential image of $\rho$ precisely consists of the objects of $\caC_{[1]} \times_{\caC_{[0]}} \caC_{[1]} \times_{\caC_{[0]}} \caC_{[1]}$ whose image under projection to the first and third factor is an equivalence with respect to the composition of the Segal space $\caC$, which are precisely the 
equivalences of the Segal space $\Tw(\caC).$
	 
\end{proof}

%\end{notation}

\begin{definition}
Let $(\caC,D)$ be an $\infty$-category with duality.
A hermitian object of $(\caC,D)$ is a homotopy $C_2$-fixed point $(\X,\tau) $ of $(\Tw(\caC),D).$
A hermitian object $(\X,\tau) $  of $(\caC,D)$ is non-degenerate if the morphism $X: A \to B$ is an equivalence. % in $\caC$. %its image under the functor $\caH(\caC,D) \to \Tw(\caC)$ corresponds to an equivalence in $\caC.$
\end{definition}
\begin{notation}
Let $(\caC,D)$ be an $\infty$-category with duality.
Let $$ \caH(\caC,D):=(\Tw(\caC),D)^{hC_2}$$ be the $\infty$-category of hermitian objects and 
$$\H^{\mathrm{nd}}(\caC,D) \subset \caH(\caC,D)$$ the full subcategory spanned by the non-degenerate hermitian objects of $\caC.$
	
\end{notation}
\begin{remark}
	
Let $(\caC,D)$ be an $\infty$-category with duality.
By definition a hermitian object of $(\caC,D)$ is a homotopy $C_2$-fixed point $(\X,\tau) $ of $(\Tw(\caC),D).$
Since the $C_2$-action on $\Tw(\caC)$ sends $f:A \to B$ to $D(f): D(B) \to D(A)$,
a hermitian object of $(\caC,D)$ consists of a morphism $X: A \to B$,
equivalences $A \simeq D(B), B \simeq D(A)$ and a homotopy between $X:A \to B$ and the composition $A \simeq D(B) \xrightarrow{D(X)} D(A) \simeq B$ and coherence data.

%The fiber of the right fibration $\H(\caC,D) \to \caC $ over some object $X$ of $\caC$is the space of homotopy $C_2$-fixed points of the fiber of the $C_2$-equivariant functor $\Tw(\caC,D) \to \widetilde{\caC^\op \times \caC}$ over the fixed point $(X,D(X)), $ which is $\caC(X, D(X)) $ with the following $C_2$-action:a morphism $f: X \to D(X) $ is sent to $D(f): X \simeq D(D(X)) \to D(X).$ So a hermitian object in $(\caC,D)$ is an object $X$ of $\caC$ equipped with ahomotopy $C_2$-fixed point of the $C_2$-space $\caC(X, D(X)) $,which is a morphism $ f: X \to D(X)$ in $\caC$ equipped with a homotopy to$ X \simeq D(D(X)) \xrightarrow{D(f)} D(X)$ and coherence data.

\end{remark}

\begin{remark}
	
For every $[n] \in \Delta$ there are natural maps $$[n] \to [n] * [n]^\op, \ [n]^\op \to [n] * [n]^\op,$$ i.e. natural transformations $\id \to e, (-)^\op \to e.$
So for any simplicial space $\caC$ there is a map of simplicial spaces
$$ \Tw(\caC):= \caC \circ e^\op \to (\caC \circ \id) \times (\caC \circ (-)^\op) = \caC \times \caC^\op. $$ 
\end{remark}

\begin{lemma}\label{gghjjnv}

The canonical natural transformation $\Tw \to \id \times (-)^\op$ of $C_2$-equivariant endofunctors of $\caP(\Delta)$ is $C_2$-equivariant.

\end{lemma}

\begin{proof}

The natural transformation $\id \to e$ is adjoint to a $C_2$-equivariant natural transformation $\alpha: \id \coprod (-)^\op \to e $ of functors $\Delta \to \Delta$, whose source carries the trivial action and whose target carries the induced non-trivial action.
We obtain a $C_2$-equivariant natural transformation $\Tw=e^* \to \id^* \times ((-)^\op)^*= \id \times (-)^\op.$

\end{proof}
\begin{construction}
For $(\caC,D) \in \caP(\Delta)^{hC_2}$ the map $\Tw(\caC) \to \caC \times  \caC^\op $ of simplicial spaces refines to a $C_2$-equivariant map $(\Tw(\caC),D) \to \widetilde{\caC \times  \caC^\op} $ by Lemma \ref{gghjjnv}, and so yields a map of simplicial spaces $$\H(\caC,D)= (\Tw(\caC),D)^{hC_2} \to (\widetilde{\caC \times  \caC^\op})^{hC_2} \simeq \caC. $$

\end{construction}

\begin{definition}
A right fibration of simplicial spaces is a map of simplicial spaces $\caC \to \caD$ such that
for any $n \geq 0$ the map $\caC_{[n]} \to \caC_{[0]} \times_{\caD_{[0]}} \caD_{[n]} $ induced by the map $[0] \simeq \{n\} \subset [n]$ is an equivalence.
\end{definition}

This generalizes the notion of right fibration of $\infty$-categories.

\begin{lemma}
Let $\caC$ be a simplicial space. The map $\Tw(\caC) \to \caC \times \caC^\op$ is a right fibration of simplicias spaces.
\end{lemma}

\begin{proof}
	
For any $n \geq 0$ the map $\Tw(\caC)_{[n]} \to \Tw(\caC)_{[0]} \times_{(\caC_{[0]}\times \caC_{[0]^\op})} (\caC_{[n]} \times \caC_{[n]^\op}) $ induced by the map $[0] \simeq \{n\} \subset [n]$ identifies with the canonical map 
$$\caC_{[n]*[n]^\op} \to \caC_{[0]*[0]^\op} \times_{(\caC_{[0]}\times \caC_{[0]^\op})} (\caC_{[n]} \times \caC_{[n]^\op}) $$ induced by the map $[0] \simeq \{n\} \subset [n]$.
The latter map is induced by the map
$$ [0]*[0]^\op \coprod_{[0] \coprod [0]^\op} ([n]\coprod [n]^\op) \to [n]*[n]^\op,$$
which is an equivalence by definition of the join construction. 
	
\end{proof}

\begin{corollary}Let $(\caC,D)$ be an $\infty$-category with duality.
The functor $\H(\caC)\to \caC$ is a right fibration.
		
\end{corollary}

%\begin{remark}Let $(\caC,D)$ be an $\infty$-category with duality.The fiber of the right fibration $\H(\caC,D) \to \caC $ over some object $X$ of $\caC$is the space of homotopy $C_2$-fixed points of the fiber of the $C_2$-equivariant functor $\Tw(\caC,D) \to \widetilde{\caC^\op \times \caC}$ over the fixed point $(X,D(X)), $ which is $\caC(X, D(X)) $ with the following $C_2$-action:a morphism $f: X \to D(X) $ is sent to $D(f): X \simeq D(D(X)) \to D(X).$ So a hermitian object in $(\caC,D)$ is an object $X$ of $\caC$ equipped with ahomotopy $C_2$-fixed point of the $C_2$-space $\caC(X, D(X)) $,which is a morphism $ f: X \to D(X)$ in $\caC$ equipped with a homotopy to$ X \simeq D(D(X)) \xrightarrow{D(f)} D(X)$ and coherence data.\end{remark}

\begin{lemma}\label{hermii}
Let $(\X,\tau)$ be space with $C_2$-action. There is a canonical equivalence over $\X:$
$$\caH^{\mathrm{nd}}(\X,\tau)=\caH(\X,\tau)\simeq (\X,\tau)^{hC_2}.$$	
	
\end{lemma}

\begin{proof}
For every $n \geq 0$ the diagonal map $\X \to \Cat_\infty([n]*[n]^\op,\X)$ is an equivalence.
Hence the $C_2$-equivariant forgetful functor $(\Tw(\X),\tau) \to \widetilde{\X \times \X}$ factors as a canonical $C_2$-equivariant equivalence
$(\Tw(\X),\tau) \simeq (\X, \tau)$ followed by the diagonal map $(\X,\tau) \to \widetilde{\X \times \X}.$
Thus $\caH^{\mathrm{nd}}(\X,\tau)=\caH(\X,\tau) =(\Tw(\X),\tau)^{hC_2} \simeq (\X,\tau)^{hC_2}$ over $\X.$
	
\end{proof}

\begin{corollary}
Let $(\caC,D)$ be an $\infty$-category with duality.
There is a canonical equivalence $$\caH^{\mathrm{nd}}(\caC,D)^\simeq \simeq (\caC^\simeq,D)^{hC_2}.$$
	
\end{corollary}

\begin{proof}
	
%For every $n \geq 0$ the $C_2$-equivariant 	
The inclusion $(\caC^\simeq,D)\subset (\caC,D)$ of symmetric monoidal $\infty$-categoried with duality induces a $C_2$-equivariant embedding $(\Tw(\caC^\simeq)^\simeq, D) \subset (\Tw(\caC)^\simeq,D)$ whose essential image are the objects of $\Tw(\caC)$ that correspond to equivalences in $\caC.$
Hence the latter embedding induces on homotopy $C_2$-fixed points an equivalence $\caH^{\mathrm{nd}}(\caC^\simeq,D) \simeq \caH^{\mathrm{nd}}(\caC,D)^\simeq.$ By Lemma \ref{hermii} there is a canonical equivalence
$\caH^{\mathrm{nd}}(\caC^\simeq,D)\simeq (\caC^\simeq,D)^{hC_2}.$	

\end{proof}

\section{Direct sum real $K$-theory}
\label{ht235wq}

Direct sum $K$-theory \cite[Definition 8.3]{gepner-groth-nikolaus} is the composition
$$\K : \Cmon(\Cat_\infty) \overset{(-)^\simeq}{\longrightarrow} \Cmon(\caS) \xrightarrow{(-)^\grp} \Grp(\caS) \to \Sp,$$
where the first functor in the composition is the right adjoint of the embedding $\Cmon(\caS) \subset \Cmon(\Cat_\infty)$, which takes the maximal subspace, the second functor takes group completion and the third functor takes the infinite delooping.
By \cite[Theorem 8.6]{gepner-groth-nikolaus} the first functor is lax symmetric monoidal and the second and third functor are symmetric monoidal. Passing to commutative algebras we obtain a functor 
$$\Calg(\Cmon(\Cat_\infty)) \to \Calg(\Sp).$$
%with $\Ring(\caS):=\Calg(\Grp(\caS))$.

Mimicking direct sum K-theory we make the following definition:

\begin{definition}
Direct sum hermitian K-theory $\K_\h $ is the K-theory of non-degenerate hermitian objects:
$$ (\Cmon(\Cat_\infty),(-)^\op)^{hC_2} \overset{\caH^{\mathrm{nd}}}{\longrightarrow}
\Cmon(\Cat_\infty) \xrightarrow{\K} \Sp.$$
\end{definition}

\begin{remark}\label{compi}
Direct sum hermitian K-theory $\K_\h $ factors as
$$(\Cmon(\Cat_\infty),(-)^\op)^{hC_2} \overset{(-)^\simeq}{\longrightarrow}
\Cmon(\caS)[C_2] \xrightarrow{(-)^{hC_2}} \Cmon(\caS)$$$$ \xrightarrow{(-)^\grp} \Grp_{\bE_\infty}(\caS) \to \Sp,$$
where the last two functors take group completion, take the infinite delooping, respectively.
\end{remark}
\begin{proposition}
\label{h4rgrf}
Direct sum hermitian $K$-theory $\K_\h$ is lax symmetric monoidal.
\end{proposition}

\begin{proof}
The first functor in the composition of Remark \ref{compi} is right adjoint to the symmetric monoidal embedding $ \Cmon(\caS)[C_2] \subset (\Cmon(\Cat_\infty),(-)^\op)^{hC_2}$ and so lax symmetric monoidal. The second functor is right adjoint to the symmetric monoidal diagonal functor $$\Cmon(\caS) \to \Cmon(\caS)[C_2].$$
The third and fourth functors in the definition of $\K_\h$ are symmetric monoidal
by \cite[Theorem 8.6]{gepner-groth-nikolaus}.
\end{proof}

Passing to commutative algebras the functor $\K_\h$ induces a functor
$$(\Calg(\Cmon(\Cat_\infty)),(-)^\op)^{hC_2} \to \Calg(\Sp).$$

Next we lift direct sum hermitian K-theory to genuine
$C_2$-spectra, which we define in the following.
%For that we use the twisted arrow category

\begin{notation}
For every $\infty$-category $\caC$ and $[n] \in \Delta$ let $\Tw([n])$ be the twisted arrow category of $[n]$ and $$\Span(\caC)_n \subset \Cat_\infty(\Tw([n]),\caC)$$ be the full subspace of functors
$F: \Tw([n]) \to \caC$ such that for every 
$0 \leq k \leq k' \leq \ell' \leq \ell \leq n
$ the following square is a pullback square:
$$\xymatrix{F(k \leq \ell) \ar[rr] \ar[d] && F(k \leq \ell') \ar[d] \\
F(k' \leq \ell)  \ar[rr] && F(k' \leq \ell').}$$

\end{notation}

The simplicial space $[\mathrm{n}] \mapsto \Cat_\infty(\Tw([\mathrm{n}]),\caC)$
induces a simplicial space $\Span(\caC): [\mathrm{n}] \mapsto \Span(\caC)_{\mathrm{n}}.$
Moreover the functor $\Cat_\infty \to \Fun(\Delta^\op, \caS), \caC \mapsto \Cat_\infty(\Tw(-),\caC)$ induces a functor $\Span: \Cat_\infty \to \Fun(\Delta^\op, \caS), \caC \mapsto \Span(\caC).$ 

\begin{notation}
Let $\Cat_{\infty}^\mathrm{lex} \subset \Cat_{\infty}$ be the subategory of small $\infty$-categories having finite limits and functors preserving finite limits.
\end{notation}
If $\caC$ is an $\infty$-category that admits pullbacks, the simplicial space $\Span(\caC)$ is a complete Segal space \cite[Proposition 3.4.]{9eb49b0be23f4b049aeeb44c6189ff21}.
So the functor $\Span: \Cat_\infty \to \Fun(\Delta^\op, \caS)$
restricts to a functor $\Span: \Cat_\infty^\mathrm{lex} \to \Cat_\infty$.

\begin{remark}\label{produ}

Let $\caC,\caD$ be $\infty$-categories. % having pullbacks.
The canonical equivalence of simplicial spaces 
$$\Cat_\infty(\Tw(-),\caC\times\caD) \simeq \Cat_\infty(\Tw(-),\caC) \times \Cat_\infty(\Tw(-),\caD)$$
restricts to an equivalence $$\Span(\caC  \times \caD) \simeq \Span(\caC) \times \Span(\caD).$$
So the functor $\Span: \Cat_\infty \to \Fun(\Delta^\op, \caS)$ preserves finite products. 
Since the inclusion $\Cat^\mathrm{lex}_\infty \subset \Cat_\infty$ preserves finite products, also the functor $\Span: \Cat_\infty^\mathrm{lex} \to \Cat_\infty$ preserves finite products.
\end{remark}

\begin{remark}

By Lemma \ref{leem} the $\infty$-category $\Cat_\infty^{\prod}$	is preadditive.
Since the inclusion $\Cat^\mathrm{lex}_\infty \subset \Cat_\infty^{\prod}$ preserves finite products, the subcategory $\Cat_\infty^\mathrm{lex} $ is preadditive, too.
So by Proposition \ref{leomor} (4) and Remark \ref{produ} the functor $\Span: \Cat_\infty^\mathrm{lex} \to \Cat_\infty$ uniquely lifts to a finite products preserving functor $\Cat_\infty^\mathrm{lex} \to \Cmon(\Cat_\infty).$
So every $\infty$-category having finite limits induces a symmetric monoidal
structure on $\Span(\caC)$ that is induced by the product of
$\caC$ but is not the product of $\Span(\caC).$
	
\end{remark}

\begin{remark}\label{symot}
Let $\caC$ be an $\infty$-category and $\caD$ an $\infty$-category having finite limits.
The natural transformation $\Tw(-)\to \id$ gives rise to a map $\caC \to \Span(\caC)$ of simplicial spaces.
The functor $$\Span(\Fun(\caC,\caD))\times \caC \to \Span(\Fun(\caC,\caD))\times \Span(\caC) \simeq$$$$ \Span(\Fun(\caC,\caD)\times \caC)\to \Span(\caD)$$ induced by evaluation $\Fun(\caC,\caD)\times \caC\to \caD$ corresponds to a functor
\begin{equation*}\label{eee}
\theta: \Span(\Fun(\caC,\caD)) \to \Fun(\caC,\Span(\caD)).\end{equation*}
The functor $\theta$ is fully faithful.
If $\caC$ is a space, $\theta$ is an equivalence. The latter holds since $\theta_{(-,\caD)}:\Span(\Fun(-,\caD)) \to \Fun(-,\Span(\caD)) $ is a natural transformation of finite products and small colimits preserving functors $\Cat_\infty^{\mathrm{lex}} \to \Cat_\infty$ and 
$\theta_{*,\caD}$ is an equivalence.
%Let $X$ be a space and $\caC$ and $\infty$-category.Since $X$ is the colimit of the constant functor $X \to \caS$ with value the contractible space, i.e. $X \simeq \colim_X(*)$, there is a natural map of simplicial spaces $$ \theta: \Span(\Fun(X,\caC)) \simeq \Span(\lim_X(\caC)) \to \lim_X(\Span(\caC)) \simeq \Fun(X, \Span(\caC)).$$
%The functor (\ref{eee}) is a map of complete Segal spaces and so represents a functor of$\infty$-categories.
Since $\Cat_{\infty}^\mathrm{lex}$ is preadditive and $\theta$ is natural in $\caC$, by Proposition \ref{leomor} (4) the map $\theta$ uniquely refines to a natural map of symmetric monoidal $\infty$-categories.
\end{remark}

\begin{notation}
Let $\Fin$ be the category of finite sets.	
\end{notation}
\begin{definition}
A genuine $C_2$-spectrum is a finite products preserving functor $$\Span(\Fin[C_2]) \to \Sp.$$
Let $$\Sp^{C_2}:= \Fun^{\prod}(\Span(\Fin[C_2]), \Sp)$$
be the $\infty$-category of genuine $C_2$-spectra.
\end{definition}

\begin{notation}
	
\begin{itemize}
\item Evaluation at $* \in \Fin[C_2] \to \Span(\Fin[C_2])$ 
induces a functor $(-)^{C_2}: \Sp^{C_2} \to \Sp$.
\item Restriction along the Yoneda-embedding 
$ B C_2 \subset \Fin[C_2] \to \Span(\Fin[C_2])$
induces a functor $(-)^{u}: \Sp^{C_2} \to \Sp[C_2]$.

\end{itemize}
\end{notation}

\begin{remark}
The $\infty$-category $\Span(\Fin)$ is the free preadditive $\infty$-category on the contractible $\infty$-category, i.e. for any preadditive $\infty$-category $\caC$
evaluation at the image of the singleton under the functor
$\Fin \to \Span(\Fin)$ defines an equivalence
$$ \Fun^{\prod}(\Span(\Fin),\caC) \simeq \caC.$$
\end{remark}
\begin{notation}

Let $t$ be the functor $$(\Cmon(\Cat_\infty), (-)^\op)^{hC_2}[C_2] \simeq (\Cmon(\Cat_\infty), (-)^\op \times \triv)^{h(C_2 \times C_2)} $$$$\to (\Cmon(\Cat_\infty),(-)^\op)^{hC_2},$$
where the last functor restricts the action via the diagonal.
\end{notation}
\begin{definition}
Direct sum real K-theory $\KR$ is the composition
$$ (\Cmon(\Cat_\infty),(-)^\op)^{hC_2} \simeq \Fun^{\prod}(\Span(\Fin), (\Cmon(\Cat_\infty),(-)^\op)^{hC_2}) \xrightarrow{(-)[C_2]}$$$$ \Fun^{\prod}(\Span(\Fin)[C_2], (\Cmon(\Cat_\infty),(-)^\op)^{hC_2}[C_2])
\xrightarrow{\Fun^{\prod}(\theta,t)} $$$$ \Fun^{\prod}(\Span(\Fin[C_2]), (\Cmon(\Cat_\infty),(-)^\op)^{hC_2}) \xrightarrow{\Fun^{\prod}(\Span(\Fin[C_2]),\K_\h)} \Sp^{C_2}=$$$$ \Fun^{\prod}(\Span(\Fin[C_2]), \Sp).$$
\end{definition}

\begin{remark}\label{expli}
Let $(\caC,D) $ be a symmetric monoidal $\infty$-category with duality.
The genuine $C_2$-spectrum $\KR(\caC,D):\Span(\Fin[C_2]) \to \Sp $
sends a finite $C_2$-set $(X,\tau)$ to the group completion of the $\bE_\infty$-space $\caH^{\mathrm{nd}}(\caC^\X,D)^\simeq$ of hermitian objects in the symmetric monoidal $\infty$-category with duality $ (\caC^\X,D^\tau)$, where $\caC^\X$ is the object-wise symmetric monoidal structure on the product %$\infty$-category $\Fun(\X,\caC)$ 
equipped with the duality
$$ (\caC^\X)^\op \simeq (\caC^\op)^\X \xrightarrow{D^\tau}\caC^\X.$$
%$$ \Fun(\X,\caC)^\op \simeq \Fun(\X,\caC^\op) \xrightarrow{\Fun(\tau,D)} \Fun(\X,\caC).$$
	
\end{remark}

\begin{theorem}\label{zum}
Let $(\caC,D) $ be a symmetric monoidal $\infty$-category with duality. There are canonical equivalences 
\begin{enumerate}
\item $\KR(\caC,D)^{C_2} \simeq \K_\h(\caC,D),$	
\item $\KR(\caC,D)^\u \simeq \K(\caC). $
\end{enumerate}

\end{theorem}

\begin{proof}
1. Let $\bar{\caC} \in \Fun^{\prod}(\Span(\Fin), (\Cmon(\Cat_\infty),(-)^\op)^{hC_2})$
correspond to $(\caC,D).$
Let $$\kappa: \Fin[C_2] \to \Span(\Fin[C_2]), \kappa': \Fin \to \Span(\Fin) $$ be the canonical functors and $$\triv: (\Cmon(\Cat_\infty),(-)^\op)^{hC_2} \to (\Cmon(\Cat_\infty),(-)^\op)^{hC_2}[C_2]$$ the functor assigning the trivial action.
Then there is a canonical equivalence $$ \KR(\caC,D)^{C_2} = (\K_\h \circ t \circ \bar{\caC}[C_2]\circ \theta)(\kappa(*)) \simeq (\K_\h \circ t \circ \bar{\caC}[C_2] \circ \kappa'[C_2])(*) $$	
$$\simeq (\K_\h \circ t \circ \triv) (\caC) \simeq \K_\h(\caC). $$

2. Let $\widetilde{\caC \times \caC} \in (\Cmon(\Cat_\infty),(-)^\op)^{hC_2}[C_2]$ be the cofree $C_2$-action on $$\caC\in (\Cmon(\Cat_\infty),(-)^\op)^{hC_2},$$
which also is the free $C_2$-action on $\caC$ since $ \Cmon(\Cat_\infty)$ is preadditive.

Let $\widetilde{\caC\times \caC^\op} \in (\Cmon(\Cat_\infty),(-)^\op)^{hC_2}$ be the cofree $\infty$-category with duality on $\caC\in \Cmon(\Cat_\infty)$.
Projection to the first factor $\caC \times \caC \to \caC$ is adjoint to a
map $\rho: t(\widetilde{\caC \times \caC}) \simeq \widetilde{\caC \times \caC^\op}$ in $(\Cmon(\Cat_\infty),(-)^\op)^{hC_2}$ lying over the equivalence $\caC \times \caC \simeq \caC \times \caC^\op$. So $\rho$ is an equivalence.
Hence there is a canonical equivalence $$\caH^\mathrm{nd}(t(\widetilde{\caC \times \caC})) \simeq \caH^\mathrm{nd}(\widetilde{\caC \times \caC^\op}) \simeq \caC.$$
Note that $C_2 \in \Fin[C_2]$ is the free $C_2$-action on the point.
Thus there is a canonical equivalence $$ \KR(\caC)^{u} = (\K_\h \circ t \circ \bar{\caC}[C_2]\circ \theta)(\kappa(C_2)) \simeq (\K_\h \circ t \circ \bar{\caC}[C_2] \circ \kappa'[C_2])(C_2) $$	
$$\simeq \K_\h(t(\widetilde{\caC \times \caC})) \simeq \K(\caH^\mathrm{nd}(t(\widetilde{\caC \times \caC}))) \simeq \K(\caC).$$

\end{proof}

By \cite[Theorem 2.2.6.2., Proposition 2.2.6.16.]{lurie.HA} for every small $\infty$-category $\caC$ and presentably symmetric monoidal
$\infty$-category $\caD$ there is a presentably symmetric monoidal structure on
the $\infty$-category $\Fun(\caC,\caD)$ given by Day-convolution characterized by the following universal property:
the evaluation functor $  \Fun(\caC,\caD) \times \caC \to \caD$ is symmetric monoidal and for every symmetric monoidal $\infty$-category $\caB$
the canonical functor $$ \Fun^{\otimes,\lax}(\caB, \Fun(\caC,\caD)) \to
\Fun^{\otimes,\lax}(\caB \times \caC,\caD) $$ is an equivalence,
where $\Fun^{\otimes,\lax} (-,-)$ refers to the $\infty$-category of lax symmetric monoidal functors. 

By the next Lemma \ref{locam} the full subcategory $\Fun^{\prod}(\caC,\caD) \subset \Fun(\caC,\caD)$ is a symmetric monoidal localization and so inherits a symmetric monoidal structure if $\caC$ is compatible with finite products.

\begin{lemma}\label{locam}
For every small $\infty$-category $\caC$ having finite products and
presentable $\infty$-category $\caD$ the full subcategory $\Fun^{\prod}(\caC,\caD) \subset \Fun(\caC,\caD)$ is a localization.
If $\caC$ is a small symmetric monoidal $\infty$-category compatible with finite products and $\caD$ is a presentably symmetric monoidal $\infty$-category,
the localization $\Fun^{\prod}(\caC,\caD) \subset \Fun(\caC,\caD)$ is
symmetric monoidal. 	
\end{lemma}

\begin{proof}
Choose a regular cardinal $\kappa$ such that $\caD$ is $\kappa$-accessible.
Let $\caD^\kappa \subset \caD $ be the full subcategory of $\kappa$-compact objects.
For every $Z \in \caC$ let $\mathrm{lan}_Z: \caD \to \Fun(\caC,\caD)$ be the left adjoint to evaluation at $Z.$ 
The full subcategory $\Fun^{\prod}(\caC,\caD) \subset \Fun(\caC,\caD)$ is the localization with respect to the set of morphisms
$$S:= \{ \coprod_{j\in K} \mathrm{lan}_{H(j)}(Y) \to \mathrm{lan}_{(\prod_{j\in K}H(j))}(Y) \mid Y \in \caD^\kappa, H: K \to \caC, K \in \Fin \}.$$

To see that the localization is symmetric monoidal, we need to see that
$S$ is closed under the tensor product.
This follows from the fact that
$\caC$ is compatible with finite products and that there is a canonical equivalence
$$ \mathrm{lan}_Z(Y) \otimes \mathrm{lan}_{Z'}(Y') \simeq \mathrm{lan}_{Z \otimes Z'}(Y \otimes Y')$$
for $Y,Y' \in \caD$, $Z,Z' \in \caC.$
	
\end{proof}

\begin{proposition}
\label{hrf4t63}
Real K-theory $\KR$ is lax symmetric monoidal.
\end{proposition}

\begin{proof}
The functor $\theta: \Span(\Fin[C_2]) \to \Span(\Fin)[C_2]$ is symmetric monoidal (Remark \ref{symot}), the functor $$t: (\Cmon(\Cat_\infty),(-)^\op)^{hC_2}[C_2] \to (\Cmon(\Cat_\infty),(-)^\op)^{hC_2}$$
is symmeric mono\-idal. So all functors in the composition of $\KR$ are symmetric monoidal
except $\Fun^{\prod}(\Span(\Fin[C_2]),\K_\h)$, which is lax symmetric monoidal as $\K_\h$ is lax symmetric monoidal (Proposition \ref{h4rgrf}).
\end{proof}

\begin{remark}\label{remop}
We can compose $\KR$ with the lax symmetric monoidal embedding
$$(\Cat_\infty^\preadd, (-)^\op)^{hC_2} \hookrightarrow (\Cmon(\Cat_\infty),(-)^\op)^{hC_2}$$
of Corollary \ref{thh} to obtain a lax symmetric monoidal functor
$$\KR: (\Cat_\infty^\preadd,(-)^\op)^{hC_2} \to \Sp^{C_2}.$$
Passing to commutative algebras we obtain a functor
$$\KR: (\Calg(\Cat_\infty)_\preadd,(-)^\op)^{hC_2} \to \Calg(\Sp^{C_2})$$
assigning to any preadditive symmetric monoidal $\infty$-category with duality a real K-theory spectrum.

\end{remark}

\section{Symmetric monoidal $\infty$-categories with duality}\label{uhus}

In this section we prove that every symmetric monoidal $\infty$-category, in which every object has a dual, admits a canonical symmetric monoidal duality (Theorem \ref{ujkp}).

\begin{definition}
Let $\caC$ be a monoidal $\infty$-category and $\X,\Y \in \caC.$
We say that a morphism $1 \to \Y \otimes \X$ exhibits $\Y$ as right dual of
$\X$ if the induced natural transformation
$\id \to (\Y \otimes (-)) \circ (\X \otimes (-)) $ exhibits
$\Y \otimes (-): \caC \to \caC$ as right adjoint to $\X \otimes (-): \caC \to \caC$.
Similarly, we define left duals. 	
\end{definition}

\begin{remark}
Let $\caC$ be a symmetric monoidal $\infty$-category and $\X,\Y \in \caC.$
Then $\Y$ is a left dual of $\X$ if and only if $\Y$ is a right dual of $\X,$
and we call $Y$ the dual of $X$ denoted by $X^\dual.$
In particular, $X$ is the dual of $X^\dual,$ in symbols $ X \simeq (X^\dual)^\dual.$ %$
%So there is no need to distinguish between left and right duals, which we call duals.	
\end{remark}

\begin{definition}

A symmetric monoidal $\infty$-category is rigid if all objects admit a dual.
\end{definition}
\begin{notation}Let $$\Cmon(\Cat_\infty)^\rig \subset \Cmon(\Cat_\infty)$$ be the full subcategory of rigid symmetric monoidal $\infty$-categories.
Let $$\Cmon(\Cat_\infty)^\rig_\preadd \subset \Cmon(\Cat_\infty)_\preadd \simeq \Calg(\Cat_\infty^{\preadd})$$ be the full subcategory of preadditive rigid symmetric monoidal $\infty$-categories.
	
\end{notation}

\begin{remark} Since the opposite of a rigid symmetric monoidal $\infty$-category is rigid, the non-trivial $C_2$-action on $\Cmon(\Cat_\infty)$ restricts to $\Cmon(\Cat_\infty)^\rig$ and $\Cmon(\Cat_\infty)^\rig_\preadd.$
\end{remark}

In the following we consider the standard examples of preadditive rigid symmetric monoidal $\infty$-categories:
\begin{notation}
Let $R \in \Calg(\Sp)$ be an $\bE_\infty$-ring spectrum.	
	
\begin{itemize}
\item Let $\Perf(\R) \subset \Mod_\R(\Sp)$ be the full subcategory of dualizable $\R$-modules.
		
\item Let $\caP(\R) \subset \Mod_\R(\Sp)$ be the full subcategory of projective $\R$-modules,
i.e. retracts of finite direct sums $\R^{\oplus \ell}$ for some $\ell \geq 0.$	
		
\end{itemize}
	
\end{notation}

\begin{example}Let $R \in \Calg(\Sp)$ be an $\bE_\infty$-ring spectrum.	We have $\caP(\R)\subset \Perf(\R)$ and the symmetric monoidal structure on $\Mod_\R(\Sp)$ restricts to $\Perf(\R)$ and $ \caP(\R)$
and makes $\Perf(\R), \caP(\R)$ to preadditive rigid symmetric monoidal $\infty$-categories.
\end{example}

%For any rigid symmetric monoidal $\infty$-category $\caC$ and an object $\L \in \caC$ we will construct a lax duality (see below)on $\caC$ sending $\mathrm{Y}$ to $\mathrm{Y}^\vee \otimes \L:$\vspace{2mm}

We will prove the following theorem:

\begin{theorem}\label{ujkp}
There is a canonical section of the forgetful functor $$(\Cmon(\Cat_{\infty})^\rig,(-)^\op)^{hC_2} \to \Cmon(\Cat_\infty)^\rig$$	
equpping a symmetric monoidal $\infty$-category $\caC$ with a symmetric monoidal duality $\caC^\op \to \caC, X \mapsto X^\dual$.
Moreover for any tensor-invertible object $L$ of $\caC$ there is a duality $\caC^\op \to \caC, X \mapsto X^\dual\otimes L$.
	
\end{theorem}

To prove Theorem \ref{ujkp} we classify dualities by certain symmetric bilinear functors in the following sense:

\begin{definition}
Let $\caC $ be an $\infty$-category.
A symmetric bilinear functor $\caC \times \caC \to \caS$ is a $C_2$-equivariant functor
$\widetilde{\caC \times \caC} \to \caS$, where $\caS$ carries the trivial action.	
	
\end{definition}

The next remark follows from Notation and terminology (\ref{noil}):

\begin{remark}\label{iwo}
	
A functor $\caC^\op \times \caC^\op \to \caS$ is classified via the straightening equivalence 
\cite[Theorem 3.2.0.1.]{lurie.HTT} by a right fibration $\caD \to \caC \times \caC$.
A $C_2$-equivariant functor $\widetilde{\caC^\op \times \caC^\op} \to \caS$ is classified by a $C_2$-equivariant right fibration $\caD \to \widetilde{\caC \times \caC}$.
A $C_2$-equivariant functor $\widetilde{\caC^\op \times \caC^\op} \to \caS$ corresponds to a
functor $(\widetilde{\caC^\op \times \caC^\op})_{hC_2} \to \caS$ that is classified by a right fibration $\caD \to (\widetilde{\caC \times \caC})_{hC_2}$.

\end{remark}

\begin{definition}
Let $\caC $ be an $\infty$-category.
A symmetric bilinear functor $\caC \times \caC \to \caS$ is non-degenerate if 
the corresponding functor $\caC \to \caP(\caC^\op)$
factors as an equivalence $\caC \to \caC^\op$.

\end{definition}

\vspace{1mm}
In the following we will show that for every $\infty$-category $\caC$ there is an equivalence
between dualities on $\caC$ and non-degenerate symmetric bilinear functors $\caC^\op \times \caC^\op \to \caS.$
%right fibrations $\caD \to (\caC \times \caC)_{hC_2}$.
To prove this, it is convenient to work with the more general notion of so-called
pro-dualities and show that there is an equivalence between pro-dualities on $\caC$ and 
symmetric bilinear functors $\caC^\op \times \caC^\op \to \caS.$

%right fibrations $\caD \to (\caC \times \caC)_{hC_2}$.

\vspace{1mm}

The non-trivial $C_2$-action on $\Cat_{\infty}$ induces a $C_2$-action
on $ \Cat_{\infty/ [1]}.$

\begin{definition}

An $\infty$-category with pro-duality is a duality preserving functor $\caM \to [1]$. % of $((\Cat_\infty,(-)^\op)^{hC_2})_{/[1]}.$

\end{definition}

The $\infty$-category of small $\infty$-categories with pro-duality is $(\Cat_{\infty/ [1]},(-)^\op)^{hC_2}.$

\begin{notation}

Let $\caM \to [1]$ be an $\infty$-category with pro-duality and set $\caC:=\caM_0, \caD:=\caM_1$. The structure of a pro-duality provides equivalences
$\alpha: \caC^\op \simeq \caD, \beta: \caD^\op \simeq \caC$ such that $\alpha^\op=\beta^{-1}.$
%whose composition (after applying $(-)^\op$ is the identity). 
We call $\caC$ the underlying $\infty$-category and call
the functor $$D: \caC^\op \times \caC^\op \to \caS, (X,Y) \mapsto \caM(X,\alpha(Y))$$
the pro-duality functor.
We say that a duality preserving functor $\caM \to [1]$ equips $\caC$ with a pro-duality
and write $(\caC,D)$ for an $\infty$-category with pro-duality leaving the remaining structure notationally implicite.
A pro-duality provides in particular an equivalence $\gamma_{X,Y}: D(X,Y) \simeq D(Y,X)$
natural in $X,Y \in \caC$ such that $\gamma_{X,Y} \circ \gamma_{Y,X}=\id.$
%The latter equivalence corresponds to a factorization of the functor\caD \to \caP(\caC)$ as $\caD \simeq \caC^\op \to \caP(\caD^\op) \simeq \caP(\caC).$

%between the functor $\caM_0^\op \times \caM_1 \simeq \caM_1 \times \caM_0^\op \xrightarrow{\caM}\caS$ and the functor $\caM_0^\op \times \caM_1 \simeq \caM_1 \times \caM_0^\op \xrightarrow{\caM}\caS$,where the equivalence $\caM_0^\op \times \caM_1 \simeq \caM_1 \times \caM_0^\op$ in the first functor switches the factors and in the second functor applies the equivalences$\caM_{0}^\op \simeq \caM_1, \caM_1 \simeq \caM_0^\op$.

A map of $\infty$-categories with pro-duality $(\caC,D) \to (\caD,E)$, i.e. a morphism in $(\Cat_{\infty/ [1]},(-)^\op)^{hC_2}$, gives a functor $\phi:\caC \to \caD$ and a map
$\kappa: D \to E \circ (\phi^\op \circ \phi^\op) $ of functors $\caC^\op \times \caC^\op \to \caS.$

\end{notation}

\begin{definition}
	
A map of $\infty$-categories with pro-duality $(\caC,D) \to (\caD,E)$ preserves the pro-duality
if for every $X \in \caC$ the map $\kappa(-,X): D(-,X)\to E(-,\phi(X)) \circ \phi^\op$
corresponds to an equivalence $(\phi^\op)_!(D(-,X)) \to E(-,\phi(X)).$
	
\end{definition}

%\begin{remark}A map of $\infty$-categories with pro-duality $(\caC,D) \to (\caD,E)$is given by a duality preserving functor $\caM \to \caN$ over $[1]$.\end{remark}

%There is a forgetful functor$$((\Cat_\infty,(-)^\op)^{hC_2})_{/[1]} \to \Cat_{\infty}, (\caM \to [1]) \mapsto \caM_{0}\simeq (\caM_1)^\op.$$

%\begin{remark}In ... pro-dualities on a stable $\infty$-category $\caC$ were identified with certain quadratic functors on $\caC$\end{remark}

\begin{remark}

The $C_2$-equivariant inclusion $ \Cat_{\infty} \to \Cat_{\infty/ [1]}, \caC \mapsto \caC \times [1]$ induces an inclusion $ (\Cat_\infty,(-)^\op)^{hC_2} \hookrightarrow (\Cat_{\infty/ [1]},(-)^\op)^{hC_2}$
of the $\infty$-category of small $\infty$-categories with duality into the $\infty$-category of small $\infty$-categories with pro-duality, via which we view $\infty$-categories  $(\caC,D)$ with duality as $\infty$-categories with pro-duality
$(\caC, \caC^\op \times \caC^\op \to \caS, (X,Y)\mapsto \caC(X,D(Y))).$
\end{remark}

\begin{lemma}\label{ito}
	
\begin{enumerate}
\item A small $\infty$-category with pro-duality $(\caC,D)$ belongs to the image of the inclusion $ (\Cat_\infty,(-)^\op)^{hC_2} \hookrightarrow (\Cat_{\infty/ [1]},(-)^\op)^{hC_2}$
if and only if the functor $\caC^\op \to \caP(\caC), X \mapsto D(-, X)$ lands in $\caC$ and the resulting functor $\caC^\op \to \caC$ is an equivalence.

\item Let $(\caC,D), (\caD,E)$ be $\infty$-categories with pro-duality in the image of the inclusion $ (\Cat_\infty,(-)^\op)^{hC_2} \hookrightarrow (\Cat_{\infty/ [1]},(-)^\op)^{hC_2}$.
A map of $\infty$-categories with pro-duality $(\caC,D) \to (\caD,E)$ belongs to the image of the inclusion $ (\Cat_\infty,(-)^\op)^{hC_2} \hookrightarrow (\Cat_{\infty/ [1]},(-)^\op)^{hC_2}$ if and only if it preserves the pro-duality.

\end{enumerate}	

\end{lemma}

\begin{proof}
(1): Note that a small $\infty$-category with pro-duality $\caC,D)$ represented by a duality preserving functor $\caM \to [1]$ belongs to the image of the inclusion $ (\Cat_\infty,(-)^\op)^{hC_2} \hookrightarrow (\Cat_{\infty/ [1]},(-)^\op)^{hC_2}$
if and only if the functor $\caM \to [1]$ lies in the image of the inclusion 
$(-)\times[1]: \Cat_{\infty} \hookrightarrow \Cat_{\infty/ [1]}$. This is equivalent to the condition that the functor $\caM \to [1]$ is a cartesian fibration classifying an equivalence $\kappa: \caM_1 \to \caM_0.$
By the definition of cartesian morphism this is equivalent to say that 
the functor $\caM_1 \to \caP(\caM_{0}), X \mapsto \caM(-,X)$ lands in the full subcategory of representables and the resulting functor $\caM_1 \to \caM_0$, which identifies with $\kappa$, is an equivalence.
Composing with the equivalence $\caM_0^\op \simeq \caM_1$ the latter condition is equivalent to say that the functor $\caC^\op \to \caP(\caC), X \mapsto D(-, X)$ lands in $\caC$ and the resulting functor $\caC^\op \to \caC$ is an equivalence.

(2): Let $(\Cat_{\infty/[1]})' \subset \Cat_{\infty/[1]}$ be the subcategory of functors
$\caB \to [1]$ that are cocartesian and cartesian fibrations and functors over $[1]$ that preserve cocartesian and cartesian morphisms. The $\infty$-category with $C_2$-action
$(\Cat_{\infty/[1]},(-)^\op)$ restricts to an $\infty$-category with $C_2$-action
$((\Cat_{\infty/[1]})',(-)^\op)$.
The $C_2$-equivariant inclusion $(-)\times [1]: (\Cat_\infty,(-)^\op) \to (\Cat_{\infty/[1]},(-)^\op)$ induces an embedding
$ (\Cat_\infty, (-)^\op) \to ((\Cat_{\infty/[1]})',(-)^\op),$ which induces an embedding
$ (\Cat_\infty, (-)^\op)^{hC_2} \to ((\Cat_{\infty/[1]})',(-)^\op)^{hC_2}.$

Let $(\caC,D), (\caD,E)$ be $\infty$-categories with pro-duality in the image of the inclusion $ (\Cat_\infty,(-)^\op)^{hC_2} \hookrightarrow (\Cat_{\infty/ [1]},(-)^\op)^{hC_2}$
represented by duality preserving functors $\caM \to [1], \caN \to [1],$ respectively,
and $\phi: (\caC,D) \to (\caD,E)$ a map of $\infty$-categories with pro-duality represented by a duality preserving functor $\rho: \caM \to \caN$ over $[1].$
Then $\phi$ belongs to the image of the inclusion $ (\Cat_\infty,(-)^\op)^{hC_2} \hookrightarrow (\Cat_{\infty/ [1]},(-)^\op)^{hC_2}$ if and only if $\rho$ is a map of cocartesian and cartesian fibrations over $[1].$ Since $\rho^\op$ is equivalent to $\rho$ by the fixed point datum, this is equivalent to say that $\rho$ is a map of cartesian fibrations over $[1].$

The map $\phi$ preserves the duality if and only if for every $X \in \caC$ the induced map $\lambda: \caM(-,X)\to \caN(-,\rho(X)) \circ \rho^\op$
corresponds to an equivalence $\lambda': (\phi^\op)_!(\caM(-,X)) \to \caN(-,\phi(X)).$
Since $\caM \to [1], \caN \to [1]$ are cartesian fibrations classifying functors $\alpha: \caM_1 \to \caM_0, \beta: \caM_1 \to \caN_0$, respectively, the map $\lambda$ identifies with the map $ \caC(-,\alpha(X))\to \caD(-,\beta(\rho(X))) \circ \phi^\op$
and the map $\lambda'$ identifies with the map 
$ (\phi^\op)_!(\caC(-,\alpha(X))) \simeq \caD(-,\phi(\alpha(X))) \to \caD(-,\beta(\rho(X)))$
representing the canonical map $\phi(\alpha(X) \to \beta(\rho(X))$.
The latter map is an equivalence if and only if $\rho$ is a map of cartesian fibrations over $[1].$
	
\end{proof}

Next we give a slightly different description of the $\infty$-category of small $\infty$-categories with pro-duality (Corollary \ref{jull}):

\begin{lemma}\label{jul}	
For every small $\infty$-category with duality $\caC$ there is a canonical 
$C_2$-action on $\Cat_{\infty/\caC}$, denoted by $(-)^\op$, such that the forgetful functor refines to a $C_2$-equivariant functor $(\Cat_{\infty/\caC},(-)^\op) \to (\Cat_{\infty},(-)^\op) $ and an equivalence
$$(\Cat_{\infty/ \caC},(-)^\op)^{hC_2}\simeq ((\Cat_\infty,(-)^\op)^{hC_2})_{/\caC}$$
over $(\Cat_\infty,(-)^\op)^{hC_2}.$

\end{lemma}

\begin{proof}
The $C_2$-action on $\Cat_\infty$ gives rise to a $C_2$-action on $\Fun([1],\Cat_\infty)$
such that the functor $\rho: \Fun([1],\Cat_\infty) \to \Cat_\infty$ evaluating at the target is $C_2$-equivariant.
Thus the pullback $\Cat_{\infty/ \caC}$ of the $C_2$-equivariant functor $\rho$ along the $C_2$-equivariant functor $* \to \Cat_\infty$ corresponding to the homotopy $C_2$-fixed point $\caC$ of $\Cat_\infty$ carries a $C_2$-action and $(\Cat_{\infty/ \caC})^{hC_2}$ is the fiber
of the induced functor $\rho^{hC_2}: \Fun([1],\Cat_\infty)^{hC_2} \to (\Cat_\infty,(-)^\op)^{hC_2}$
over $\caC.$ The functor $\rho^{hC_2}$ identifies with the functor
$\Fun([1],(\Cat_\infty,(-)^\op)^{hC_2}) \to (\Cat_\infty,(-)^\op)^{hC_2}$ evaluating at the target whose fiber over $\caC$ is $((\Cat_\infty,(-)^\op)^{hC_2})_{/\caC}.$

%By ... the functor $\rho$ is a cocartesian fibration. By ... this implies that $\rho$ classifies a $C_2$-equivariant functor $\kappa: \Cat_\infty \to \widehat{\Cat}_\infty$ whose source carries the non-trivial $C_2$-action and whose target carries the trivial $C_2$-action.Hence $\kappa$ gives rise to a functor $(\Cat_\infty,(-)^\op)^{hC_2} \to \widehat{\Cat}_\infty[C_2]$sending $\caC$ to $\Cat_{\infty/ \caC}.$

\end{proof}

\begin{corollary}\label{jull}
There is a canonical $C_2$-action on $\Cat_{\infty/[1]}$ and a canonical equivalence
$(\Cat_{\infty/ [1]},(-)^\op)^{hC_2}\simeq ((\Cat_\infty,(-)^\op)^{hC_2})_{/[1]}.$

\end{corollary}	

\begin{notation}
Let $\caR \subset \Fun([1],\Cat_\infty)$ be the full subcategory of right fibrations.
\end{notation}

%We consider the pullback $$ \Cat_\infty \times_{\Cat_\infty} \caR$$of evaluation at the target $\caR \to \Cat_\infty$along the endofunctor of $\Cat_\infty$ sending$\caC $ to $(\caC \times \caC)_{hC_2}$ and 

The following definition is motivated by Remark \ref{iwo}:

\begin{definition}
The $\infty$-category $\Sym$ of small $\infty$-categories equipped with a symmetric bilinear functor is the pullback:
$$\xymatrix{\Sym \ar[rr]^{} \ar[d] && \caR \ar[d]^{}\\
\Cat_{\infty}\ar[rr] && \Cat_{\infty},}$$
where the right vertical functor evaluates at the target and the bottom horizontal functor
sends $\caC $ to $(\widetilde{\caC \times \caC})_{hC_2}$.
	
%equivalence over $\Cat_{\infty}:$$$(\Cat_{\infty/ [1]},(-)^\op)^{hC_2} \simeq \Cat_\infty \times_{\Cat_\infty} \caR.$$
	
\end{definition}

We prove the following theorem, for which we use the twisted arrow right fibration $\Tw(\caC)\to \caC \times \caC^\op$ associated to any $\infty$-category $\caC,$
which by straightening classifies the mapping space functor $\caC^\op \times \caC \to \caS$
\cite[Proposition 5.2.1.3.]{lurie.HTT}.

\begin{theorem}\label{main}
There is a canonical equivalence \begin{equation}\label{ooop}
(\Cat_{\infty/ [1]},(-)^\op)^{hC_2} \simeq \Sym\end{equation}
over $\Cat_{\infty}$ that sends an $\infty$-category with pro-duality 
$(\caC,D)$ to the right fibration $\caD \to \widetilde{\caC \times \caC}$
classifying $D.$
Under this equivalence small $\infty$-categories with duality correspond to non-degenerate symmetric bilinear functors.
%$\caM \to [1]$ the pullback of the $C_2$-equivariant twisted arrow right fibration$\Tw(\caM)\to \widetilde{\caM\times \caM^\op}$ along the $C_2$-equivariant embedding $$\widetilde{\caM_0 \times \caM_0}\simeq \caM_0 \times \caM_1^\op \subset \widetilde{\caM\times \caM^\op}.$$

%\item Under equivalence (\ref{ooop}) small $\infty$-categories with duality correspond to non-degenerate symmetric bilinear functors.Equivalence (\ref{ooop}) sends a small $\infty$-category with duality $(\caC,D)$ to the right fibration $\caD \to \widetilde{\caC \times \caC}$classifying $\caC^\op \times \caC^\op \to \caS, (X,Y)\mapsto \caC(X,D(Y)).$

%the pullback of the $C_2$-equivariant twisted arrow right fibration$\Tw(\caC)\to \widetilde{\caC\times \caC^\op}$ along the $C_2$-equivariant equivalence $\id \times D^\op:\widetilde{\caC \times \caC}\simeq \widetilde{\caC\times \caC^\op}.$\end{enumerate}

%(\caM_0 \times \caM_1^\op) \times_{(\caM \times \caM^\op)} \Tw(\caM) \to \caM_0 \times \caM_0.$$

%equivalence over $\Cat_{\infty}:$$$(\Cat_{\infty/ [1]},(-)^\op)^{hC_2} \simeq \Cat_\infty \times_{\Cat_\infty} \caR.$$

\end{theorem}

Theorem \ref{main} follows from the next lemma and proposition:

\begin{lemma}\label{Lemus}

The outer square in the following diagram is a pullback square:
$$\xymatrix{\Cat_\infty \times_{\Cat_\infty[C_2]} \caR[C_2] \ar[r]^{} \ar[d] & \caR[C_2] \ar[d]^{}\ar[r]& \caR \ar[d]\\
\Cat_{\infty}\ar[r] & \Cat_{\infty}[C_2]\ar[r] & \Cat_\infty,}$$
where the middle and right vertical functor evaluate at the target, the left bottom horizontal functor sends $\caC $ to $\widetilde{\caC \times \caC}$ and the right bottom horizontal functor takes homotopy $C_2$-coinvariants.
% defines an equivalence$$ \Cat_\infty \times_{\Cat_\infty[C_2]} \caR[C_2] \simeq \Cat_\infty \times_{\Cat_\infty} \caR, $$where the pullback on the left hand side is formed along the functor $ \Cat_\infty \to \Cat_\infty[C_2]$ that sends $\caC$ to $\widetilde{\caC \times \caC}$, the cofree $C_2$-action switching the factors.

\end{lemma}

\begin{proof}
%Note first that for every space $X$ taking the colimit gives a functor $\Fun(X,\Cat_\infty)\to \Cat_{\infty}$ that sends the final object to $X$ and so lifts to a functor $\Fun(X,\Cat_\infty)\to \Cat_{\infty/X}.$	
	
Since $B C_2$ is a space, straightening \cite[Theorem 3.2.0.1.]{lurie.HTT} gives an equivalence
$ \Cat_\infty[C_2] \simeq \Cat_{\infty/ B C_2}.$ 
Hence we have a right fibration $ \Cat_\infty[C_2] \simeq \Cat_{\infty/ B C_2} \to  \Cat_{\infty} $ that sends a small $\infty$-category with $C_2$-action
$\caD$ classified by a functor $ X \to B C_2$ to $X$, where $X$ has the universal property of the coinvariants $\caD_{hC_2}$, the colimit of the functor $BC_2 \to \Cat_\infty$ encoding the $C_2$-action on $\caD$, by \cite[Corollary 3.3.4.3.]{lurie.HTT}.
The right fibration $ \Cat_\infty[C_2] \to \Cat_{\infty} $ yields an equivalence $$\Fun([1],  \Cat_\infty[C_2])
\simeq  \Cat_\infty[C_2] \times_{  \Cat_{\infty} } \Fun([1], \Cat_{\infty}) $$ over $  \Cat_\infty[C_2]$ that restricts to an equivalence
$\caR[C_2]
\simeq \Cat_\infty[C_2] \times_{  \Cat_{\infty} } \caR $ over $ \Cat_\infty[C_2]. $

\end{proof}

\begin{proposition}\label{theo}

There is a pullback square of $\infty$-categories with $C_2$-action:
%, where $\caR, \Cat_\infty$ carry the trivial action and $\Cat_{\infty/ [1]}$ carries the non-trivial action of Corollary \ref{jull}:
\begin{equation}\label{hih}
\xymatrix{(\Cat_{\infty/ [1]},(-)^\op)  \ar[rr]^{} \ar[d] && (\caR,\triv) \ar[d]^{}\\
\underbrace{(\Cat_{\infty/ \{0,1\}},(-)^\op)}_{\widetilde{\Cat_{\infty} \times \Cat_{\infty}}}\ar[rr] && (\Cat_{\infty},\triv),}\end{equation}
where the right vertical functor evaluates at the target, the bottom horizontal functor assigns the product, the left vertical functor takes the fiber over $0,1$ and 
the top horizontal functor assigns to any functor $\caM \to [1]$ the right fibration
$ \caD \to \caM_0 \times \caM_1^\op$ classifying the functor $\caM_0^\op \times \caM_1 \subset \caM^\op \times \caM \xrightarrow{\caM}\caS.$
%$$(\caM_0 \times \caM_1^\op) \times_{(\caM \times \caM^\op)} \Tw(\caM) \to \caM_0 \times \caM_1^\op.$$

%$C_2$-equivariant equivalence$$\Cat_{\infty/ [1]} \simeq \widetilde{\Cat_{\infty} \times \Cat_{\infty}} \times_{\Cat_{\infty}} \caR$$over $\Cat_{\infty/ \{0,1\}} \simeq \widetilde{\Cat_{\infty} \times \Cat_{\infty}}$sending a functor $\caM \to [1]$ to the right fibration$$(\caM_0 \times \caM_1^\op) \times_{(\caM \times \caM^\op)} \Tw(\caM) \to \caM_0 \times \caM_1^\op,$$where $\caR, \Cat_\infty$ carry the trivial action and $\Cat_{\infty/ [1]}$ carries the non-trivial action.
\end{proposition}

\begin{proof}

By Lemma \ref{gghjjnv} there is a 
%the natural transformation $\alpha: \Tw \to \id \times (-)^\op $ of functors $\Cat_\infty \to \Cat_\infty$ is $C_2$-equivariant if the source carries the unique non-trivial $C_2$-action and the target carries the trivial $C_2$-action.The map $\alpha$ defines a 
$C_2$-equivariant functor $$\gamma: \Cat_\infty \to \widetilde{\Cat_\infty \times \Cat_\infty} \times_{\Cat_\infty} \caR, \ \caC \mapsto (\caC,\caC^\op, \Tw(\caC) \to \caC \times \caC^\op).$$ 
Let $\rS$ be an $\infty$-category with duality. The functor $\gamma$ gives rise to a $C_2$-equivariant functor
$ \gamma_\rS: $$$ \Cat_{\infty/\rS} \to (\Cat_{\infty/\rS} \times \Cat_{\infty/\rS^\op}) \times_{\Cat_\infty} \caR,$$$$ (\caC \to \rS) \mapsto (\caC \to \rS, \caC^\op \to \rS^\op, \Tw(\caC) \to \caC \times \caC^\op), $$
where the $C_2$-action on $(\Cat_{\infty/\rS} \times \Cat_{\infty/\rS^\op})$
switches the factors and applies the duality of $\rS$ and the functor
$(\Cat_{\infty/\rS} \times \Cat_{\infty/\rS^\op}) \to \Cat_\infty$ in the pullback is the composition of the forgetful functor
$\nu': (\Cat_{\infty/\rS} \times \Cat_{\infty/\rS^\op}) \to \widetilde{\Cat_\infty \times \Cat_\infty}$ and the product functor $\mu: \widetilde{\Cat_\infty \times \Cat_\infty} \to \Cat_\infty.$

For every $\s \in \rS$ let $\omega$ be the functor $\Cat_{\infty/\rS} \to \Cat_{\infty}$ taking fiber over $\s$ and $\omega'$ the $C_2$-equivariant functor $ (\Cat_{\infty/\rS} \times \Cat_{\infty/\rS^\op}) \to \widetilde{\Cat_\infty \times \Cat_\infty}$
corresponding to the composition $ (\Cat_{\infty/\rS} \times \Cat_{\infty/\rS^\op}) \xrightarrow{q} \Cat_{\infty/\rS} \xrightarrow{\omega} \Cat_{\infty}$, where the first functor is projection to the first factor.

Let $\nu $ be the forgetful functor $\Cat_{\infty/\rS} \to \Cat_{\infty}$.
There is a canonical natural transformation $\rho: \omega \to \nu$, whose component at a functor $\caC \to \rS$ is the functor $\caC_\s \to \caC$.
The natural transformation $q \circ \rho: q \circ\omega \to q \circ\nu$
gives rise to a $C_2$-equivariant natural transformation $\rho' : \omega' \to \nu'.$
Composing with the $C_2$-equivariant product functor $\mu: \widetilde{\Cat_\infty \times \Cat_\infty} \to \Cat_\infty$ we obtain a $C_2$-equivariant natural transformation $\mu \circ \rho' : \mu \circ \omega' \to \mu \circ \nu': (\Cat_{\infty/\rS} \times \Cat_{\infty/\rS^\op}) \to \Cat_\infty.$
Since $\caR \to \Cat_\infty$ is a cartesian fibration, $\mu \circ \rho'$
gives rise to a $C_2$-equivariant functor
$$\kappa: (\Cat_{\infty/\rS} \times \Cat_{\infty/\rS^\op}) \times_{\Cat_\infty} \caR \to (\Cat_{\infty/\rS} \times \Cat_{\infty/\rS^\op}) \times_{\Cat_\infty} \caR, $$
$$ (\caC \to \rS, \caD \to \rS^\op, \caB \to \caC \times \caD) \mapsto (\caC \to \rS, \caD \to \rS^\op, (\caC_\s \times \caD_{\tau(s)}) \times_{(\caC \times \caD)} \caB \to \caC_\s \times \caD_{\tau(s)}.$$
where the $C_2$-equivariant functor in the pullback on the left hand side is $ \mu \circ \nu'$ and on the right hand side is $ \mu \circ \omega'$.
The $C_2$-equivariant composition $$ \Cat_{\infty/\rS} \xrightarrow{\gamma_\rS} (\Cat_{\infty/\rS} \times \Cat_{\infty/\rS^\op}) \times_{\Cat_\infty} \caR \xrightarrow{\kappa} (\Cat_{\infty/\rS} \times \Cat_{\infty/\rS^\op}) \times_{\Cat_\infty} \caR $$
$$ (\caC \to \rS)\mapsto (\caC \to \rS, \caC^\op \to \rS^\op, (\caC_\s \times \caC_{\tau(\s)}^\op) \times_{(\caC \times \caC^\op)} \Tw(\caC) \to \caC_\s \times \caC_{\tau(\s)}^\op$$ 
for $\rS=[1]$ and $\s=0$ gives a $C_2$-equivariant functor
$$\Cat_{\infty/[1]} \to (\Cat_{\infty/[1]} \times \Cat_{\infty/[1]}) \times_{\Cat_\infty} \caR $$
over $ \Cat_{\infty/[1]} \times \Cat_{\infty/[1]}$, which corresponds to the desired
$C_2$-equivariant functor
$$\zeta: \Cat_{\infty/[1]} \to \widetilde{\Cat_{\infty} \times \Cat_{\infty}} \times_{\Cat_\infty} \caR $$
over $\widetilde{\Cat_{\infty} \times \Cat_{\infty}}.$
By \cite[Theorem 5.2.]{heine2024localglobalprincipleparametrizedinftycategories} the functor $\zeta$ is an equivalence.

\end{proof}

%\begin{proof}[Proof of Theorem \ref{main}]The canonical $C_2$-equivariant equivalence$$\Cat_{\infty/ [1]} \simeq \widetilde{\Cat_{\infty} \times \Cat_{\infty}} \times_{\Cat_{\infty}} \caR$$over $\Cat_{\infty/ \{0,1\}} \simeq \widetilde{\Cat_{\infty} \times \Cat_{\infty}}$of Theorem \ref{theo} induces on homotopy $C_2$-fixed points an equivalence$$(\Cat_{\infty/ [1]},(-)^\op)^{hC_2} \simeq \Cat_\infty \times_{\Cat_\infty[C_2]} \caR[C_2].$$By Lemma \ref{Lemus} taking homotopy $C_2$-coinvariants defines an equivalence$$ \Cat_\infty \times_{\Cat_\infty[C_2]} \caR[C_2] \simeq \Cat_\infty \times_{\Cat_\infty} \caR.$$\end{proof}

\begin{corollary}\label{iglo}
The pullback square (\ref{hih}) of Proposition \ref{theo} induces on homotopy $C_2$-fixed points a pullback square
\begin{equation*}
\xymatrix{(\Cat_{\infty/ [1]},(-)^\op)^{hC_2}  \ar[rr]^{} \ar[d] && \caR[C_2] \ar[d]^{}\\
\Cat_{\infty} \ar[rr] && \Cat_{\infty}[C_2],}\end{equation*}
where the bottom horizontal functor sends $\caC $ to $\widetilde{\caC \times \caC}$.
	
\end{corollary}

\begin{remark}\label{hallos}

The composition $$ \caS[C_2] \subset (\Cat_\infty,(-)^\op)^{hC_2} \subset (\Cat_{\infty/ [1]},(-)^\op)^{hC_2} \to \caR[C_2] \to \Cat_{\infty}[C_2]$$
is the canonical embedding, where the middle functor is the top functor in the pullback square of Corollary \ref{iglo} and the last functor evaluates at the target.

\end{remark}

\begin{proof}[Proof of Theorem \ref{main}]
The first statement follows immediately from Corollary \ref{iglo} and Lemma \ref{Lemus}.
By Remark \ref{ito} small $\infty$-categories with duality correspond to non-degenerate symmetric bilinear functors.

\end{proof}

Now we are ready to construct for any symmetric monoidal $\infty$-category $\caC$
a pro-duality on $\caC.$ 

\begin{definition}
A symmetric monoidal $\infty$-category with pro-duality is a commutative monoid in $(\Cat_{\infty/[1]},(-)^\op)^{hC_2}.$
	
\end{definition}
The $\infty$-category of symmetric monoidal $\infty$-categories with pro-duality is $$
\Cmon((\Cat_{\infty/ [1]},(-)^\op)^{hC_2}) \simeq (\Cmon(\Cat_{\infty/ [1]}),(-)^\op)^{hC_2}.$$
\begin{corollary}\label{ihol}
There is a canonical pullback square %equivalence \begin{equation}\label{ooopu}
%\Cmon((\Cat_{\infty/ [1]},(-)^\op)^{hC_2}) \simeq \Cmon(\Sym)\end{equation}
%covering equivalence (\ref{ooop}).The right hand side is the pullback
$$\xymatrix{\Cmon((\Cat_{\infty/ [1]},(-)^\op)^{hC_2}) \ar[rr]^{} \ar[d] && \Cmon(\caR) \ar[d]^{}\\
\Cmon(\Cat_{\infty})\ar[rr] && \Cmon(\Cat_{\infty}),}$$
where the right vertical functor evaluates at the target and the bottom horizontal functor
sends $\caC $ to $(\widetilde{\caC \times \caC})_{hC_2}$.

\end{corollary}

\begin{notation}\label{waru}
	
Let $\caC$ be a symmetric monoidal $\infty$-category.
The tensor product functor $\otimes:\caC\times \caC \to \caC$ refines to a $C_2$-equivariant 
symmetric monoidal functor $\widetilde{\caC\times\caC} \to (\caC,\triv)$
corresponding to a symmetric monoidal functor $(\widetilde{\caC\times\caC})_{hC_2} \to \caC.$
For every object $A$ in $\caC$ the pullback $(\widetilde{\caC\times\caC})_{hC_2} \times_{\caC} \caC_{/A} \to (\widetilde{\caC\times\caC})_{hC_2}$ along $(\widetilde{\caC\times\caC})_{hC_2} \to \caC$ is a right fibration classifying a functor $(\widetilde{\caC^\op\times\caC^\op})_{hC_2} \to \caS$ corresponding to a $C_2$-equivariant functor
$$D_A: \widetilde{\caC^\op\times\caC^\op} \to (\caS, \triv), (X,Y) \mapsto \caC(X\otimes Y,A),$$ and so corresponds under the equivalence of Theorem \ref{main} 1. to an $\infty$-category with pro-duality denoted by $(\caC,D_A).$
For $A$ the tensor unit we write $D_A=(-)^\dual.$

%where $D_A$ is the functor $\caC^\op \times \caC^\op \to \caS, (X,Y) \mapsto \caC(X\otimes Y,A).$
\end{notation}

\begin{remark}
For every commutative algebra $A$ in $\caC$ the functor $\caC_{/A} \to \caC$ is symmetric monoidal. Thus the right fibration $(\widetilde{\caC\times\caC})_{hC_2} \times_{\caC} \caC_{/A} \to (\widetilde{\caC\times\caC})_{hC_2}$ is a symmetric monoidal functor and so corresponds under the equivalence of Corollary \ref{ihol} to a symmetric monoidal $\infty$-category with pro-duality refining $(\caC,D_A).$

%If $\caC$ is rigid and $A$ is tensor-invertible, the right fibration $(\widetilde{\caC\times\caC})_{hC_2} \times_{\caC} \caC_{/A} \to (\widetilde{\caC\times\caC})_{hC_2}$ classifies via straightening thefunctor $((\widetilde{\caC\times\caC})_{hC_2})^\op \xrightarrow{\otimes} \caC^\op \xrightarrow{[-,A]} \caC.$

%We write$(\caC,(-)^\dual)$ for the symmetric monoidal $\infty$-category with duality provided by Theorem \ref{ujkp}.	
	
\end{remark}

\begin{definition}
	
A right fibration $\caA \to \caB$ is representable if $\caA$ has a final object.
This is equivalent to say that $\caA \to \caB$ classifies a representable presheaf.

\end{definition}
\begin{notation}
Let $\caU \subset \caR$ be the full subcategory of representable right fibrations.
\end{notation}

\begin{theorem}\label{thhh}

\begin{enumerate}
\item There is a canonical finite products preserving functor $$\Phi: \Cmon(\Cat_\infty) \times_{\Cat_\infty} \caU \to \Cmon(\Cat_\infty) \times_{\Cat_\infty} (\Cat_{\infty/[1]},(-)^\op)^{hC_2}$$
over $ \Cmon(\Cat_\infty)$
sending a symmetric monoidal $\infty$-category $\caC$ and an object $A$ of $\caC$ to the  pro-duality $D_A$ %$\caC^\op \to \caP(\caC), X \mapsto \caC((-)\otimes X,\L)$
on $\caC.$ 
%The $C_2$-action on $\Cmon(\Cat_\infty)^\rig$ induced by the non-trivial $C_2$-action on $\Cat_\infty$ is trivial.The forgetful functor $$\Cmon(\Cat_{\infty/[1]})^{hC_2} \to \Cmon(\Cat_\infty)$$ admits a section.

\item If $\caC$ is rigid and $\L$ is tensor-invertible, the pro-duality $D_A$ on $\caC$ is a duality.

\item For any symmetric monoidal functor $\alpha: \caC \to \caD$ between rigid symmetric monoidal $\infty$-categories and every object $A$ of $\caC$ the map $\Phi(\alpha): (\caC, D_A) \to (\caD, D_{\alpha(A)})$ %of $\infty$-categories with pro-duality 
preserves the pro-duality.
\end{enumerate}
\end{theorem}

\begin{proof}
%We will construct a canonical section $\alpha$ of the forgetful functor $$(\Cmon(\Cat_\infty)^\rig)^{hC_2} \to \Cmon(\Cat_\infty)^\rig.$$ Such a section $\alpha$ is adjoint to a $C_2$-equivariant functorfrom the trivial $C_2$-action on $ \Cmon(\Cat_\infty)^\rig $ to the $C_2$-action on $ \Cmon(\Cat_\infty)^\rig $ induced by the non-trivial $C_2$-action on $\Cat_\infty$, whose underlying functor (after forgetting the equivariance) is the identity.We start with constructing the section $\alpha$:

(1): Since right fibrations are stable under pullback, evaluation at the target $ \caR \to \Cat_\infty$ is a cartesian fibration classifying a functor $$ \Cat_\infty^\op \to \widehat{\Cat}_\infty, \ \caC \mapsto \caR_\caC \simeq \caP(\caC):= \Fun(\caC^\op, \caS),$$
which is the canonical functor $\Fun((-)^\op,\caS)$ by \cite[Corollary A.32.] {Gepner2017LaxCA}.
Thus a natural transformation $\theta:\mathrm{F} \to \G$ of functors
$ \Cmon(\Cat_\infty) \to \Cat_\infty$ yields a map of cartesian fibrations
$\Cmon(\Cat_\infty) \times_{\Cat_\infty} \caR \to  \Cmon(\Cat_\infty) \times_{\Cat_\infty} \caR$ over $\Cmon(\Cat_\infty)$
that induces on the fiber over any $\caC \in \Cmon(\Cat_\infty)$
the functor $$\theta_\caC^\ast: \caP(\G(\caC)) \to \caP(\mathrm{F}(\caC)).$$
We apply this to the canonical map from the forgetful functor
$ \Cmon(\Cat_\infty) \to \Cat_\infty$ followed by the functor 
$\Cat_\infty \to  \Cat_\infty, \ \caC \mapsto (\caC \times \caC)_{h C_2}$ to the forgetful functor $ \Cmon(\Cat_\infty) \to \Cat_\infty$ (component-wise induced by the multiplication map) to obtain a map of cartesian fibrations
$$\Psi: \Cmon(\Cat_\infty) \times_{\Cat_\infty} \caR \to  \Cmon(\Cat_\infty) \times_{\Cat_\infty} \caR$$ over $\Cmon(\Cat_\infty)$
that induces on the fiber over any $\caC \in \Cmon(\Cat_\infty)$
the functor $$\caP(\caC) \to \caP((\caC \times \caC)_{h C_2})$$
that precomposes with the functor $\otimes: (\caC \times \caC)_{h C_2} \to \caC$.

The desired functor is the restriction
$$\Cmon(\Cat_\infty) \times_{\Cat_\infty} \caU 
\subset \Cmon(\Cat_\infty) \times_{\Cat_\infty} \caR \xrightarrow{\Psi}
\Cmon(\Cat_\infty) \times_{\Cat_\infty} \caR$$$$
\simeq \Cmon(\Cat_\infty) \times_{\Cat_\infty} (\Cat_{\infty/[1]},(-)^\op)^{hC_2},$$
which preserves finite products, where the last equivalence is by Theorem \ref{main}.

(2): If $\caC$ is rigid and $A$ is tensor-invertible, the associated pro-duality $D_A$ on $\caC$ is a duality: this follows from Lemma \ref{ito} (1) because the functor
$\caC^\op \to \caP(\caC), X \mapsto \caC((-)\otimes X,A)$ factors as 
$$ \caC^\op \xrightarrow{(-)^\dual} \caC \xrightarrow{ (-)\otimes A} \caC \hookrightarrow \caP(\caC),$$ where the latter is the Yoneda-embedding. The functor
$(-)\otimes A: \caC \to \caC$ is an equivalence because $A$ is tensor-invertible.
The functor $(-)^\dual: \caC^\op \to \caC$ is an equivalence by \cite[Lemma 3.10.]{heine2024topologicalmodelcellularmotivic}.

(3): Under the conditions of (3) the functor $D_A: \caC^\op \times \caC^\op \to \caS, (X,Y)\mapsto \caC(X\otimes Y,A)$ identifies with the functor $\caC(X, [Y,A]_\caC)$,
where $[-,-]$ is the internal hom of $\caC.$
For every $X \in \caC$ the canonical map %$D_A \to D_{\alpha(A)} \circ (\alpha^\op \times \alpha^\op)$
$$D_A(-,X) \to D_{\alpha(A)}(-,\alpha(X)) \circ \alpha^\op$$
identifies with the canonical map
$$\caC(-,\X^\dual \otimes A) \to \caD(-,\alpha(X)^\dual\otimes\alpha(A)) \circ \alpha^\op.$$
The latter corresponds to a map 
%corresponds to a map $\alpha^\op_!(D_A(-,X)) \to D_{\alpha(A)}(-,\alpha(X)) $.The latter map identifies with the canonical map
$\alpha^\op_!(\caC(-,\X^\dual \otimes A)) \to \caD(-,\alpha(X)^\dual\otimes\alpha(A)) $,
which is an equivalence by \cite[Proposition 5.2.6.3.]{lurie.HTT}.

%Hence the claim follows from Lemma \ref{ito} (2).

%The induced map $$ \caC(X\otimes Y,A) \to \caC(\alpha(X)\otimes \alpha(Y),\alpha(A))$$

\end{proof}

\begin{corollary}

There is a canonical functor $$\Cmon(\caU) \to (\Cmon(\Cat_{\infty/[1]}), (-)^\op)^{hC_2}$$
over $ \Cmon(\Cat_\infty)$
sending a symmetric monoidal $\infty$-category $\caC$ and a commutative algebra
$A $ in $\caC$ to the symmetric monoidal pro-duality $(\caC,D_A).$

\end{corollary}

\begin{proof}

The functor $\Phi$ over $\Cmon(\Cat_\infty)$ preserves finite products and so
%because the pullbacks $$\Cmon(\Cat_\infty)^\rig \times_{\Cat_\infty} (\Cat_\infty)_\ast, \Cmon(\Cat_\infty)^\rig \times_{\Cat_\infty} (\Cat_\infty,(-)^\op)^{hC_2}$$ are pullbacks of $\infty$-categories with small limits and functors preserving small limits, and $\Pic$ is closed under small limits in $\Cmon(\Cat_\infty)^\rig \times_{\Cat_\infty} (\Cat_\infty)_\ast$.Since $\Phi$ preserves finite products, $\Phi$ 
gives rise to the desired functor
$$\Cmon(\caU) \simeq \Cmon(\Cmon(\Cat_\infty) \times_{\Cat_\infty} \caU) \to$$$$ \Cmon(\Cmon(\Cat_\infty) \times_{\Cat_\infty} (\Cat_{\infty/[1]},(-)^\op)^{hC_2}) \simeq (\Cmon(\Cat_{\infty/[1]}), (-)^\op)^{hC_2}$$ over $\Cmon(\Cat_\infty)$.	

\end{proof}

The diagonal embedding $\Cat_{\infty} \to \Fun([1], \Cat_\infty), \caC \mapsto (\id: \caC \to \caC)$
induces an embedding $\Cat_{\infty} \to \caU$ that preserves finite products and so yields an embedding $\Cmon(\Cat_{\infty}) \to \Cmon(\caU).$
Restricting along this embedding we obtain the following corollary, which gives Theorem \ref{ujkp}, where for the second part of the statement we use Lemma \ref{ito}.

\begin{corollary}\label{coukl}

There is a canonical functor $$\Cmon(\Cat_\infty) \to (\Cmon(\Cat_{\infty/[1]}), (-)^\op)^{hC_2}$$
over $ \Cmon(\Cat_\infty)$
equipping a symmetric monoidal $\infty$-category $\caC$ with a pro-duality $(\caC,(-)^\dual).$
The latter functor restricts to a section of the forgetful functor $$(\Cmon(\Cat_{\infty})^\rig,(-)^\op)^{hC_2} \to \Cmon(\Cat_\infty)^\rig.$$	

\end{corollary}

\begin{corollary}
There is a $C_2$-equivariant equivalence $$(\Cmon(\Cat_{\infty})^\rig,(-)^\op) \simeq (\Cmon(\Cat_{\infty})^\rig,\triv).$$ 
%induced by the $C_2$-action on $\Cat_\infty$ taking the opposite $\infty$-category restricts to the trivial $C_2$-action on $\Cmon(\Cat_{\infty})^\rig$.

\end{corollary}

\begin{proof}

The forgetful functor $(\Cmon(\Cat_{\infty})^\rig,(-)^\op)^{hC_2} \to \Cmon(\Cat_\infty)^\rig$ admits a canonical section by Corollary \ref{coukl}.
This section $$\Cmon(\Cat_\infty)^\rig \to (\Cmon(\Cat_{\infty})^\rig,(-)^\op)^{hC_2}$$ corresponds to a $C_2$-equivariant functor $$(\Cmon(\Cat_\infty)^\rig,\triv) \to (\Cmon(\Cat_{\infty})^\rig,(-)^\op)$$ whose underlying functor after forgetting the $C_2$-actions is the identity.

\end{proof}

%\begin{corollary}There is a canonical functor $$\Cmon(\Pic) \to (\Cmon(\Cat_{\infty})^\rig,(-)^\op)^{hC_2}$$over $ \Cmon(\Cat_\infty)^\rig$sending a symmetric monoidal $\infty$-category $\caC$ and an associative algebra$\L $ of $\caC$ to a duality $\caC \to \caC^\op, X \mapsto \caC(X,\L).$	The forgetful functor $$(\Cmon(\Cat_\infty)^\rig)^{hC_2} \to \Cmon(\Cat_\infty)^\rig$$ admits a section.\end{corollary}

\begin{lemma}\label{aga}
Let $\X$ be a grouplike $\bE_\infty$-space (seen as rigid symmetric monoidal $\infty$-category). %Since tensor-inverses in $\X$ are duals, $\X$ is rigid.
The grouplike $\bE_\infty$-space with $C_2$-action $(X,(-)^\dual)$ of Notation \ref{waru} is the fiber of the $C_2$-equivariant codiagonal map $\mu: \widetilde{\X\oplus\X} \to (\X,\triv)$.
\end{lemma}

\begin{proof}
	
Let $\theta$ be the map of grouplike $\bE_\infty$-spaces $\theta: (\widetilde{\X\times\X})_{hC_2} \times_{\X} * \to (\widetilde{\X\times\X})_{hC_2}$, that is the fiber of the map $(\widetilde{\X\oplus\X})_{hC_2} \to \X$ corresponding to the $C_2$-equivariant codiagonal map $\mu: \widetilde{\X\oplus\X} \to (\X,\triv)$ of grouplike $\bE_\infty$-spaces.
Then $\theta$ seen as map of grouplike $\bE_\infty$-spaces over $BC_2$ corresponds via straightening $$ \Grp_{\bE_\infty}(\caS)[C_2] \simeq \Grp_{\bE_\infty}(\caS[C_2]) \simeq \Grp_{\bE_\infty}(\caS_{/BC_2}) $$ to the fiber of the $C_2$-equivariant map of grouplike $\bE_\infty$-spaces $\mu.$
Thus by Remark \ref{hallos} we find that the grouplike $\bE_\infty$-space with $C_2$-action $(X,(-)^\dual)$ 
is the fiber of $\mu.$
	
%so corresponds under the equivalence of Corollary \ref{ihol} to a symmetric monoidal $\infty$-category with pro-duality refining $(\caC,D_A).$By definition the $C_2$-equivariant symmetric monoidal functor $(-)^\dual: \widetilde{\X \oplus \X} \to (\caS,\triv)$ corresponds to a functor $(\widetilde{\X\oplus\X})_{hC_2} \to \caS$ that is classified by the fiber $F \to (\widetilde{\X\oplus\X})_{hC_2}$ of the map$(\widetilde{\X\oplus\X})_{hC_2} \to \X$ corresponding to the $C_2$-equivariant codiagonal map$ \widetilde{\X\oplus\X} \to (\X,\triv).$The map $F \to (\widetilde{\X\oplus\X})_{hC_2}$ is likewise the pullback of the map$\B C_2 \to BC_2 \times \X$ along the canonical map $(\widetilde{\X\oplus\X})_{hC_2} \to \X \times BC_2$ and so the fiber of the $C_2$-equivariant codiagonal map$ \widetilde{\X\oplus\X} \to (\X,\triv)$.	
	
\end{proof}
\begin{remark}
We can compose the functor $$\KR: (\Cmon(\Cat_\infty)_\preadd, (-)^\op)^{hC_2} \to \Calg(\Sp^{C_2})$$
of Remark \ref{remop} with the functor $$ \gamma: \Cmon(\Cat_\infty)^\rig_\preadd \to  (\Cmon(\Cat_\infty)_\preadd, (-)^\op)^{hC_2} , \caC \mapsto (\caC,(-)^\dual)$$
to obtain a functor $$ \Cmon(\Cat_\infty)^\rig_\preadd \to \Calg(\Sp^{C_2})$$
that assigns to a preadditive rigid symmetric monoidal $\infty$-category
$\caC$ a real $K$-theory $\bE_\infty$-$C_2$-spectrum $\KR(\caC,(-)^\dual).$

Moreover we can twist the duality with a $C_2$-action before applying real $K$-theory:
composing the functor $$\gamma[C_2]: \Cmon(\Cat_\infty)^\rig_\preadd[C_2] \to  (\Cmon(\Cat_\infty)_\preadd, (-)^\op)^{hC_2}[C_2]$$
with the twisting functor $$t: (\Cmon(\Cat_\infty)_\preadd, (-)^\op)^{hC_2}[C_2] \simeq (\Cmon(\Cat_\infty)_\preadd, (-)^\op \times \triv)^{h(C_2 \times C_2)} $$$$\to (\Cmon(\Cat_\infty)_\preadd,(-)^\op)^{hC_2}$$
and the functor $\KR$ gives a functor
$$ \Cmon(\Cat_\infty)^\rig_\preadd[C_2] \xrightarrow{t \circ \gamma[C_2]}(\Cmon(\Cat_\infty)_\preadd,(-)^\op)^{hC_2} \xrightarrow{\KR}\Calg(\Sp^{C_2}).$$

%We can compose the latter functor with the functor$\Calg(\Sp)\to \Calg(\Cat_\infty)^\rig_\preadd, R \mapsto \Mod_\R(\Sp)$ to obtain a functor $$ \Calg(\Cat_\infty)^\rig_\preadd \to \Calg(\Sp^{C_2}).$$

\end{remark}

\begin{remark}\label{rey}
Let $R \in \Calg(\Sp)$ be an $\bE_\infty$-ring spectrum.
Real $K$-theory $KR(\caP(R),(-)^\dual)$ only depends on the connective cover of $\R.$
Other constructions of the Real $K$-theory of $\caP(R),(-)^\dual$ can be found in % is equivalent to the real $K$-theory of $\caP(R), (-)^\dual$ of 
\cite{real.wald}, other constructions of the Real $K$-theory of $\Perf(R),(-)^\dual $
can be found in \cite{calmès2021hermitian1}, \cite{calmès2021hermitian2}, \cite{calmès2021hermitian3}, \cite{real.wald} and \cite{spitzweck.gw}.
%If $R$ is a commutative ring, hermitian $K$-theory $KR(\caP(R),(-)^\dual)$ agrees with the connective cover of the hermitian $K$-theory of $\Perf(R),(-)^\dual$ of \cite{calmès2021hermitian1}, \cite{calmès2021hermitian2}, \cite{calmès2021hermitian3} and \cite{real.wald} if one chooses the genuine symmetric refinement \cite[Theorem B]{hebestreit2023stable}. The latter result also implies that the $K$-theory of $\Perf(R)$ defined via $Q$-constructionor $S$-construction agrees with group completion of $\caP(\R)^\simeq.$
The hermitian K-theory of $\Perf(R),(-)^\dual$ of $\cite{spitzweck.gw}$ agrees with the hermitian K-theory of $\Perf(R),(-)^\dual$ of \cite{real.wald} and with the connective cover of the hermitian K-theory of $\Perf(R),(-)^\dual$ of \cite{calmès2021hermitian1}, \cite{calmès2021hermitian2}, \cite{calmès2021hermitian3} if one chooses the homotopy symmetric refinement by \cite[B.1.1. Proposition]{calmès2021hermitian2}.

%and equivalent to the genuine symmetric $K$-theory of $\Perf(R).$The Real $K$-theory $KR(\Perf(R)(-)^\dual)$ is equivalent to the real $K$-theory of $\caP(R)(-)^\dual$ of \cite{real.wald} and equivalent to the genuine symmetric $K$-theory of $\Perf(R).$

%If we use instead of $P(R)$ all perfect $R$-modules we obtain directsum (hermitian) K-theory of this category, where in the 0-th homotopy grouptriangles do not split into sums. Better constructions of hermitian and realK-theory in this case can be found in ...
	
\end{remark}

\begin{remark}\label{comp1}
	
In \cite{hebestreit2023stable}[Theorem B] the authors prove that for every commutative ring $R$
hermitian $K$-theory of $(\Perf(R), (-)^\dual)$ for the genuine symmetric Poincar\'e structure agrees with the group completion of the $\bE_\infty$-space of non-degenerate hermitian objects in $\caP(R), (-)^\dual.$
Applying the result to the hyperbolic $\infty$-category this also implies that the $K$-theory of $\Perf(R)$ defined via $Q$-construction or $S$-construction agrees with group completion of $\caP(R)^\simeq.$

%... associate to any additive $\infty$-category with duality $(\caC,D)$ (endowed with the cartesian structure) ... a genuine $C_2$-spectrum of real $K$-theory, which 
It would be interesting to know if this result extends to real $K$-theory, i.e. that the genuine $C_2$-spectrum of real $K$-theory of $(\Perf(R), (-)^\dual)$ of \cite[Definition 4.5.1.]{calmès2021hermitian2} agrees with our construction of the genuine $C_2$-spectrum of real $K$-theory of $\caP(R), (-)^\dual.$
By the fundamental fibre square \cite[Corollary 4.4.14.]{calmès2021hermitian2}
this would imply that $L$-theory for the genuine symmetric Poincar\'e structure of \cite[Definition 4.4.4.]{calmès2021hermitian2} agrees with the geometric fixed points of our
genuine $C_2$-spectrum $KR(\caP(R), (-)^\dual).$
	
%quadratic extension of the functor $$\caC^\op \to \Sp_{\geq0}, \Z \mapsto \caC(\Z,D(\Z))^{hC_2},$$which is a  genuine $C_2$-spectrum called by the authors genuine symmetric Poincar\'e structure on the stable envelope of $\caC$ whose geometric fixed points are the $\L$-space of what these authors call the genuine symmetric Poincar\'e structure on the stable envelope of $\caC$.
	
\end{remark}

\section{Linear dualities}
\label{hjtgr54rz}

In this section we "classify" dualities on a fixed symmetric monoidal $\infty$-category with duality $\caC$ that are compatible with the $\caC$-action and the duality.

\begin{notation}
For every symmetric monoidal $\infty$-category $\caC$ let $$B\caC \subset \Mod_{\caC}(\Cat_\infty)$$ be the full symmetric monoidal subcategory spanned by $\caC.$
\end{notation}

\begin{remark}
The $\infty$-category $\Mod_{\caC}(\Cat_\infty)$ carries a canonical closed action of
$\Cat_{\infty}$ and so is enriched in $\Cat_{\infty}$ \cite[Proposition 6.10.]{HEINE2023108941}
or  \cite[Example 3.23.]{heine2024bienriched}, what one calls an $(\infty,2)$-category. Hence the full subcategory $B\caC$ may be viewed as an
%(symmetric monoidal) 
$(\infty,2)$-category as well.

\end{remark}

\begin{remark}
The space $(B\caC)^\simeq$ is connected and there is a canonical symmetric monoidal equivalence between $\caC$ and the symmetric monoidal $\infty$-category of endomorphisms of the tensor unit in $B\caC $.
In particular, for any grouplike $\bE_\infty$-space $\X$ the symmetric monoidal $\infty$-category $B\X$ is a grouplike $\bE_\infty$-space and there is a canonical equivalence $\X \simeq \Omega(B\X) $ of grouplike $\bE_\infty$-spaces.
%The map $\X \simeq \Omega(B\X) $ of grouplike $\bE_\infty$-spacesis adjoint to an equivalence of spectra $\X[1] \simeq B\X$.

\end{remark}

\begin{remark}\label{Funkt}
	
Every symmetric monoidal functor $\caC \to \caD$ gives rise to a symmetric monoidal functor $\caD \otimes_{\caC} (-): \Mod_{\caC}(\Cat_{\infty}) \to \Mod_{\caD}(\Cat_{\infty})$	that restricts to a symmetric monoidal functor $B\caC \to B\caD.$	
The latter is the one-object case of $\infty\Cat$-enriched Yoneda-extension 
(see \cite{HINICH2020107129} or \cite{heine2024equivalencemodelsinftycategoriesenriched} for enriched presheaves and \cite[Corollary 5.27.]{heine2024higher} for enriched Day-convolution).

\end{remark}

Next we equip $B\caC$ with a $C_2$-action when $\caC$ is a symmetric monoidal $\infty$-category with duality.

%\begin{notation}Let $\Calg:=  \Fun^\otimes(\Fin,-): \Cmon(\Cat_\infty) \to \Cmon(\Cat_\infty), $where $\Fin$ carries the coproduct and $\Fun^\otimes(-,-)$ is the internal hom of $\Cmon(\Cat_\infty)$ with respect to the closed symmetric monoidal structure.\end{notation}Note that the functor $\Calg: \Cmon(\Cat_\infty) \to \Cmon(\Cat_\infty) $ preserves small limits .

\begin{lemma}
\label{htr5te4}
Let $G$ be a group and $\caD \in \Cmon(\widehat{\Cat}_\infty)[\G]$ 
such that $\caD$ admits geometric realizations that are preserved by the tensor product. Let $A \in \Calg(\caD)^{h\G} \simeq \Calg(\caD^{h\G})$
(This equivalence holds since $\Calg$ preserves small limits).

\begin{enumerate}
\item There is a $\G$-action on the symmetric monoidal $\infty$-category $\Mod_A(\caD)$ and a symmetric monoidal equivalence: $$\Mod_A(\caD^{h\G}) \simeq \Mod_A(\caD)^{h\G}.$$
 
\item If $\caD \in \Cmon(\widehat{\Cat}_\infty)[\G]$ carries the trivial $\G$-action, 
the $\G$-action on $\Mod_A(\caD)$ induced by $A \in \Calg(\caD)^{h\G} \simeq \Calg(\caD)[\G]$ arises by applying the functor
$$\Calg(\caD) \to \Cmon(\widehat{\Cat}_\infty), A \mapsto \Mod_A(\caD).$$
\end{enumerate}
\end{lemma}

\begin{proof}In the following we use that a symmetric monoidal $\infty$-category with $\G$-action $\caB$, i.e. a functor $B \G \to \Cmon(\Cat_\infty)$ is classified by a $B \G$-family of $\infty$-operads $\caB^\otimes \to \Fin_*  \times B \G$ in the sense of \cite[Definition 2.3.2.10.]{lurie.HA}, where the fiber over the object of $B \G$ is a symmetric monoidal $\infty$-category.
By functoriality of $\Calg, \Mod$ the $B \G$-family of $\infty$-operads $\caD^\otimes \to \Fin_*  \times B \G$ whose fiber over the object of $B \G$
is a symmetric monoidal $\infty$-category, induces $B \G$-families of $\infty$-operads $\Calg(\caD)^\otimes \to \Fin_*  \times B \G$,
$\Mod(\caD)^\otimes \to \Fin_*  \times B \G$ whose fiber over the object of $B \G$
is a symmetric monoidal $\infty$-category.

There is an adjunction $B \G \times (-): \mathrm{Op}_\infty \rightleftarrows \mathrm{Op}^{B \G}_\infty: ((-)^{h \G})^\otimes$ between $\infty$-operads and $B \G$-families of
$\infty$-operads, which restricts to the adjunction
$\triv: \Cmon(\Cat_{\infty}) \rightleftarrows \Cmon(\Cat_{\infty})[B \G]: (-)^{h \G}$.
The right adjoint sends $\Calg(\caD)^\otimes \to \Fin_*  \times \B G, \Mod(\caD)^\otimes \to \Fin_*  \times B \G $
to the symmetric monoidal $\infty$-categories $$\Calg(\caD^{h\G})^\otimes \to \Fin_*, \Mod(\caD^{h\G})^\otimes \to \Fin_*$$ that classify $\Calg(\caD)^{h\G} \simeq \Calg(\caD^{h\G}), \Mod(\caD)^{h\G} \simeq \Mod(\caD^{h\G})$ with the induced symmetric monoidal structures.

1.: The symmetric monoidal $\infty$-category $\Calg(\caD^{h\G})^\otimes \to \Fin_*$ is cocartesian \cite[Proposition 3.2.4.7.]{lurie.HA}. So by \cite[Proposition 2.4.3.9.]{lurie.HA} the 
object $A \in \Calg(\caD^{h\G})$ uniquely lifts to a commutative algebra in $\Calg(\caD^{h\G})$, i.e. a lax symmetric monoidal functor $\Fin_* \to \Calg(\caD^{h\G})^\otimes $. The latter corresponds by adjointness to a map of $B \G$-families of $\infty$-operads $\Fin_* \times B \G \to \Calg(\caD)^\otimes.$
We define a $B \G$-family of $\infty$-operads by the pullback of $B \G$-families of $\infty$-operads:
$$\xymatrix{\Mod_A(\caD)^\otimes \ar[r] \ar[d] & \Mod(\caD)^\otimes \ar[d] \\
\Fin_* \times B \G \ar[r]^A & \Calg(\caD)^\otimes}.$$

The right vertical functor is a cocartesian fibration by \cite[Theorem 4.5.3.1.]{lurie.HA}, where we use that $\caD$ admits geometric realizations that are preserved by the tensor product.
This implies that $\Mod_A(\caD)^\otimes \to \Fin_* \times B \G $ is a cocartesian fibration so that its fiber over the object of $B \G$ is a symmetric monoidal $\infty$-category. Applying the right adjoint $((-)^{h\G})^\otimes$ we obtain a pullback square
of $\infty$-operads $$\xymatrix{(\Mod_A(\caD)^{h\G})^\otimes \ar[r] \ar[d] & \Mod(\caD^{h\G})^\otimes \ar[d] \\
\Fin_* \ar[r]^A & \Calg(\caD^{h\G})^\otimes }$$
inducing an equivalence of symmetric monoidal $\infty$-categories:
$$(\Mod_A(\caD)^{h\G})^\otimes \simeq\Mod_A(\caD^{h\G})^\otimes.$$

2.: The functor
$\Calg(\caD) \to \Cmon(\widehat{\Cat}_\infty), A \mapsto \Mod_A(\caD)$
sends the object $A \in \Calg(\caD)[\G]$ to the symmetric monoidal $\infty$-category with $\G$-action classified by the pullback of $\Mod(\caD)^\otimes \to \Calg(\caD)^\otimes$ along $A: \Fin_* \times B \G \to \Calg(\caD)^\otimes$.

If $\caD^\otimes \to \Fin_*$ carries the trivial $\G$-action,
the classifying $B \G$-family of $\infty$-operads is $\caB^\otimes:= \caD^\otimes \times B \G \to \Fin_* \times B \G$ so that $\Mod(\caB)^\otimes \simeq \Mod(\caD)^\otimes \times B \G$ and similarly for $\Calg$.
Hence the $\G$-family of $\infty$-operads $\Mod_A(\caD)^\otimes \to \Fin_* \times B \G$ is the pullback of the cocartesian fibration $\Mod(\caD)^\otimes \to \Calg(\caD)^\otimes$ along $A: \Fin_* \times B \G \to \Calg(\caD)^\otimes.$ 

%The cocartesian fibration $\Mod(\caD)^\otimes \to \Calg(\caD)^\otimes$ classifies a lax symmetric monoidal functor $\Calg(\caD)^\otimes \to \widehat{\Cat}_\infty^\times $ so that the $G$-family of $\infty$-operads $\Mod_A(\caD)^\otimes \to \Fin_* \times \B \G$ classifies the lax symmetric monoidal functor $ \Fin_* \times \B G \xrightarrow{A}\Calg(\caD)^\otimes \to \widehat{\Cat}_\infty^\times $

\end{proof}
\begin{notation}
Let $(\caC,D)$ be a symmetric monoidal $\infty$-category with duality. 
By Proposition \ref{htr5te4} the symmetric monoidal $\infty$-category $\Mod_{\caC}(\Cat_\infty) $
carries an induced $C_2$-action, which we denote by $(\Mod_{\caC}(\Cat_\infty), D).$
\end{notation}
By Proposition \ref{htr5te4} there is a canonical equivalence $$(\Mod_{\caC}(\Cat_\infty),D)^{hC_2} \simeq \Mod_{\caC}((\Cat_\infty,(-)^\op)^{hC_2}).$$
The $C_2$-action $(\Mod_{\caC}(\Cat_\infty), D)$ restricts to a $C_2$-action $(B\caC,D).$

\begin{remark}\label{osol}
If $(\X,\tau)$ is a grouplike $\bE_\infty$-space with $C_2$-action, then $(B\X, \tau)$ is a grouplike $\bE_\infty$-space with $C_2$-action. By Proposition \ref{htr5te4} (2) the $C_2$-action $(B\X,\tau)$ arises by applying to $(\X,\tau)$ the functor
$B: \Grp_{\bE_\infty}(\caS)\to \Grp_{\bE_\infty}(\caS)$. % which lies in the full subcategory of connected grouplike $\bE_\infty$-spaces and is inverse to taking loops.
\end{remark}
\begin{definition}
Let $(\caC,D)$ be a symmetric monoidal $\infty$-category with duality. 	
A $\caC$-linear duality is a homotopy $C_2$-fixed point of the $C_2$-action on 
$\Mod_{\caC}(\Cat_\infty)$.
\end{definition}

\begin{remark}A $\caC$-linear duality $(\caM, E)$ consists of 
an $\infty$-category with duality $(\caM,E)$ and a $\caC$-action on $\caM$,
which is compatible with the duality in the following sense, and coherence data:
the following square commutes:
$$\xymatrix{\caC \times \caM \ar[r] \ar[d]^{D\times E} & \caM \ar[d]^E \\
\caC \times \caM  \ar[r] & \caM.}$$
	
\end{remark}

\begin{notation}
Let $(\caC,D)$ be a symmetric monoidal $\infty$-category with duality. 	
The symmetric monoidal $\infty$-category of 
$\caC$-linear dualities on $\caC$ is $(B\caC,D)^{hC_2}$.
%The grouplike $\bE_\infty$-space of $\caC$-linear dualities on $\caC$ is $(B\caC^{hC_2})^\simeq$.
\end{notation}

\begin{remark}The symmetric monoidal $\infty$-category $(B\caC,D)^{hC_2}$ is the full symmetric monoidal subcategory of $(\Mod_{\caC}(\Cat_\infty),D)^{hC_2} $ spanned by the $\caC$-linear dualities whose image in $\Mod_{\caC}(\Cat_\infty)$
is equivalent to $\caC.$
\end{remark}

\begin{notation}
Let $\caC$ be a symmetric monoidal $\infty$-category.

\begin{itemize}
\item Let $\Pic(\caC)\subset \caC^\simeq$ be the full grouplike $\bE_\infty$-space of tensor-invertible objects.

\item Let $\pic(\caC)$ be the corresponding spectrum to $\Pic(\caC)$, the infinite delooping.

\end{itemize}

\end{notation}

\begin{remark}
Let $\Grp_{\bE_\infty}(\caS) \subset \Cmon(\caS)$ be the full subcategory of grouplike $\bE_\infty$-spaces. The $C_2$-equivariant embedding $(\caS,\triv)\subset (\Cat_\infty,(-)^\op)$
induces a $C_2$-equivariant embedding
$(\Grp_{\bE_\infty}(\caS),\triv) \subset (\Cmon(\Cat_\infty),\op)$.
The functor $$\Pic: \Cmon(\Cat_\infty) \to \Grp_{\bE_\infty}(\caS)$$ is right adjoint to the embedding
$\Grp_{\bE_\infty}(\caS) \subset\Cmon(\Cat_\infty)$ and so refines to a $C_2$-equivariant functor $$ (\Cmon(\Cat_\infty),\op) \to (\Grp_{\bE_\infty}(\caS),\triv),$$
which we can prolong to a $C_2$-equivariant functor $$ (\Cmon(\Cat_\infty),\op) \to (\Grp_{\bE_\infty}(\caS),\triv)\hookrightarrow (\Sp,\triv) .$$
\end{remark}
%For any symmetric monooidal $\infty$-category with duality $(\caC,D)$ the

The next lemma tells us that the space of $\caC$-linear dualities on $\caC$ is the infinite loop space of the spectrum $(\pic(\caC)[1],D)^{hC_2}:$

\begin{lemma}\label{ppp}
Let $(\caC,D)$ be a symmetric monoidal $\infty$-category with duality.
There are canonical equivalences of grouplike $\bE_\infty$-spaces:
$$ ((B\caC,D)^{hC_2})^\simeq \simeq \Omega^\infty((\pic(\caC)[1],D)^{hC_2}).$$
$$\{\caC\}\times_{B\caC} (B\caC,D)^{hC_2} \simeq \Omega^\infty(0 \times_{\pic(\caC)[1]} (\pic(\caC)[1],D)^{hC_2}).$$
 
\end{lemma}

\begin{proof}
By Remark \ref{Funkt} and \ref{osol} the inclusion $(\Pic(\caC),D) \subset (\caC,D)$ of symmetric monoidal $\infty$-categories with duality induces an equivalence $(B\Pic(\caC),D) \simeq ((B\caC)^\simeq,D) $ of grouplike $\bE_\infty$-spaces with $C_2$-action and so an equivalence of grouplike $\bE_\infty$-spaces $$ ((B\caC,D)^{hC_2})^\simeq\simeq ((B\caC)^\simeq,D)^{hC_2} \simeq (B\Pic(\caC),D)^{hC_2} \simeq \Omega^\infty((\pic(\caC)[1],D)^{hC_2}).$$

We obtain an equivalence of grouplike $\bE_\infty$-spaces 
$$\{\caC\}\times_{B\caC} (B\caC,D)^{hC_2} \simeq \{\Pic(\caC)\}\times_{B\Pic(\caC)} (B\Pic(\caC),D)^{hC_2} \simeq $$$$ \Omega^\infty(\{0 \times_{\pic(\caC)[1]} (\pic(\caC)[1],D)^{hC_2}).$$

\end{proof}

%Let $\caD$ be an additive $\infty$-category and $\X \in \caD.$In the following we want to refine the inversion map$\X \to \X$ multiplying by $-1$ to a $C_2$-action.

%\begin{notation}Let $S$ be the sphere spectrum.The grouplike $\bE_\infty$-space of$C_2$-actions on $S$ is the fiber of the forgetful functor $\Pic(S)^{BC_2} \to \Pic(S)$.\end{notation}

%\begin{lemma}Let $\caD$ be a symmetric monoidal additive $\infty$-category that admits finite colimits and finite limitsLet $C$ be defined by the cofiber $$ \tu \to \widetilde{\tu \oplus \tu} \to C.$$Then $C \in \Pic(\caD)[C_2].$Let $ C^{-1} \to \widetilde{\tu \oplus \tu} \otimes C^{-1} \to \tu $$  C^{-1} \simeq \tu'.$There is a a canonical fiber sequence$$\X_\sim \to \widetilde{\X \oplus \X} \xrightarrow{\mu} \X$$in $\caD[C_2]$, where $\mu$ is the $C_2$-equivariant codiagonal morphism$\widetilde{\X \oplus \X} \to \X$ in $\caD.$	\end{lemma}

%We will prove the following theorem:

The next notation is motivated by Lemma \ref{aga}:

\begin{notation}

Let $\caD$ be an additive $\infty$-category and $(\X,\tau)\in \caD[C_2]$.

%Let $(\X, (-)^\dual)$ be the fiber of the $C_2$-equivariant codiagonal morphism$\widetilde{\X \oplus \X} \to \X.$

\begin{itemize}
\item Let $(\X, (-)^\dual \circ \tau)$ be the diagonal action of the fiber of the $C_2 \times C_2$-equivariant codiagonal morphism
$$\widetilde{\X \oplus \X} \simeq \widetilde{(\X,\tau) \oplus (\X, \tau)} \xrightarrow{\mu} (\X, \tau).$$
\item Let $(\X,(-)_\dual \circ \tau)$ be the diagonal action of the cofiber of the $C_2 \times C_2$-equivariant diagonal morphism
$$ (\X,\tau) \xrightarrow{\delta} \widetilde{(\X,\tau) \oplus(\X,\tau)} \simeq \widetilde{\X \oplus \X}.$$
\end{itemize}

%Let $\widetilde{\Omega}(\X)$ be the fiber of the $C_2$-equivariant diagonal morphism$$ \X \xrightarrow{\delta} \widetilde{\X \oplus \X}.$$

\end{notation}

\begin{remark}
	
Note that $$	(\X, (-)^\dual \circ \tau) \simeq (\X, \tau \circ (-)^\dual),
(\X,(-)_\dual \circ \tau) \simeq (\X,\tau \circ (-)_\dual)$$
by definition of a $C_2 \times C_2$-action.
	
\end{remark}

\begin{remark}\label{blablas}
%Let $\caD$ be an additive $\infty$-category and $\X\in \caD$. 
Every additive $\infty$-category $\caD$ is canonically enriched in the $\infty$-category $\Sp_{\geq0}$ of connective spectra.
In particular, for every $\ell \geq 0$ and $\X \in \caD$ the tensor $S^{\oplus \ell} \otimes \X$ of the $\ell$-fold sum of the shere spectrum $S$ and $\X$ is defined, and is given by $\X^{\oplus \ell}$.
Moreover there is a canonical equivalence $(S,(-)^\dual) \otimes (\X,\triv) \simeq (\X,(-)^\dual)$ in $\caD[C_2]$
lying over the identity and dually, an equivalence $(S,(-)_\dual) \otimes (\X,\triv) \simeq (\X,(-)_\dual)$ in $\caD[C_2]$ lying over the identity.

\end{remark}

\begin{notation}
We write $(S^1,\tau) \in \caS_*[C_2]$ for the image under the suspension functor
$\Sigma: \caS \to \caS_*$ of the $C_2$-action on $S^0$ permuting the two points.
	
\end{notation}
\begin{remark}\label{kinda}
Let $S^0$ carry the $C_2$-action permuting the two points.
Let $(-)_+: \caS \to \caS_*$ be the functor adding a disjoint base point and
$\Sigma: \caS \to \caS_*$ the suspension. % functor. %sending a space $\X$ to the cofiber of the map $\X \to *.$
There is a cofiber sequence of pointed spaces with $C_2$-action
$ (S^0)_+ \to (\ast)_+ \to (S^1,\tau)$.
Let $\caD$ be an $\infty$-category that admits loops and a zero object.
Then $\caD$ is cotensored over finite pointed spaces and the latter cofiber sequence yields for every $(\X,\tau) \in \caD[C_2]$ a fiber sequence in $\caD[C_2]$ after taking the diagonal action: $$(\X,\tau)^{(S^1,\tau)} \to (\X,\tau)^{(\ast)_+} \simeq (\X,\tau) \xrightarrow{\delta} (\X,\tau)^{(S^0)_+} \simeq \widetilde{(\X,\tau) \times (\X,\tau)} \simeq \widetilde{\X \times \X}.$$

Thus if $\caD$ is stable, there is a canonical equivalence in $\caD[C_2]:$
$$(\X,\tau \circ (-)_\dual) \simeq (\X,\tau)^{(S^1,\tau)}[1].$$	

\end{remark}

\begin{lemma}\label{oolp}
	
The following $C_2$-actions on $S^2$ in $\caS_*$ are equivalent:
	
\begin{enumerate}
\item The $C_2$-action $(S^1, \tau) \smash (S^1, \triv)$. % induced by the one on $S^0$ permuting the two elements.
		
\item %The suspension of the pointed $C_2$-space $(S^1,(-)^\dual), $ i.e. 
The cofiber of the $C_2$-equivariant codiagonal map $\widetilde{S^1 \vee S^1} \to (S^1, \triv)$.
		
\item The $C_2$-action $\widetilde{S^1 \smash S^1}$ by switching the factors. %and applying the trivial action on $S^1.$%induced by the one on $S^0$ permuting the two elements.
		
%where $S^1$ carries the trivial action. % by switching the factors and applying the trivial action on $S^1.$%induced by the one on $S^0$ permuting the two elements.
		
\end{enumerate}	
	
\end{lemma}

\begin{proof}
The pointed $C_2$-space $(S^1, \tau)$ can be modeled by the subset of $\mathbb{R}^2$ of distance 1 of the origin equipped with the $C_2$-action reflecting on the $x$-axis.	
%reflection on the $x$-axis when $S^1$ is represented by the subset of $\mathbb{R}^2$ of distance 1 of the origin.
The pointed $C_2$-space $(S^1, \tau) \smash (S^1, \triv)$ is then modeled by the subset of $\mathbb{R}^3$ of distance 1 of the origin equipped with the $C_2$-action reflecting on the $x$-axis.	

The cofiber of the $C_2$-equivariant codiagonal map $\widetilde{S^1 \vee S^1} \to (S^1, \triv)$ is the suspension of the cofiber of the $C_2$-equivariant codiagonal map $\widetilde{S^0 \vee S^0} \to (S^0, \triv)$.
The cofiber of the $C_2$-equivariant codiagonal map $\widetilde{S^0 \vee S^0} \to (S^0, \triv)$ is modeled by the pushout of the map $\{ 0,1,2 \} \to \{0,1\}, 0 \mapsto 0, 2 \mapsto 0, 1 \mapsto 1$ along the inclusion $\{ 0,1,2 \} \subset [0,2]$ into the interval,
which is the subset of $\mathbb{R}^2$ of distance 1 of the origin equipped with the $C_2$-action reflecting on the $x$-axis.

The pointed $C_2$-space $\widetilde{S^1 \smash S^1}$ is modeled by the one-point compactification of the $C_2$-space $\widetilde{\mathbb{R} \times \mathbb{R}}$.
The $C_2$-space $\widetilde{\mathbb{R} \times \mathbb{R}}$ is $C_2$-equivariantly isomorphic via rotation by 45 degree to $\mathbb{R}^2 $ equipped with reflection at the $x$-axis.
Hence the pointed $C_2$-space $\widetilde{S^1 \smash S^1}$ is modeled by the subset of $\mathbb{R}^2$ of distance 1 of the origin equipped with the $C_2$-action reflecting on the $x$-axis.	

\end{proof}

\begin{corollary}\label{sosol}
There is a canonical equivalence of spectra with $C_2$-action lying over the identity:
$$ (\widetilde{S[1] \smash S[1]})[-2] \simeq (S, (-)^\dual).$$	
	
\end{corollary}
\begin{corollary}
There is a canonical $C_2$-equivariant equivalence of spectra with $C_2$-action lying over the identity:
$$ (S,(-)^\dual\circ (-)_\dual)\simeq (S,(-)_\dual\circ (-)^\dual)\simeq(S, \triv).$$	
	
\end{corollary}

\begin{proof}
By Remark \ref{kinda} the $C_2$-spectrum $(S,(-)^\dual \circ (-)_\dual) \simeq (S,(-)_\dual \circ (-)^\dual)$ is the internal hom in $\Sp_{\geq0}[C_2]$ from $(\Sigma^\infty(S^1), \tau)$ to $(S, (-)^\dual)[1].$
Lemma \ref{oolp} implies that there is a canonical equivalence of spectra with $C_2$-action $(\Sigma^\infty(S^1), \tau) \simeq (S, (-)^\dual)[1]$ so that the latter internal hom in $\Sp_{\geq0}[C_2]$ is equivalent to the internal hom in $\Sp_{\geq0}[C_2]$ from $(S,\triv)$ to $(S,\triv),$ which is $(S,\triv).$

\end{proof}

Remark \ref{blablas} gives the following corollary:
\begin{corollary}\label{kui}
Let $\caD$ be an additive $\infty$-category and $\X\in \caD$ an object. % that admits loops.
There is a canonical equivalence $(\X,(-)^\dual \circ (-)_\dual) \simeq (\X,(-)_\dual \circ (-)^\dual) \simeq (\X,\triv)$ in $\caD[C_2].$

\end{corollary}

We  prove the following proposition:

\begin{proposition}\label{htewett}\label{htewegf}
Let $(\caC,D)$ be a symmetric monoidal $\infty$-category with duality.
There is an exact sequence of spectra:
$$\pic(\caC) \to (\pic(\caC),D \circ (-)_\dual)^{h C_2} \to (\pic(\caC)[1],D)^{hC_2} \to \pic(\caC)[1]. $$

%canonical equivalence of spectra %grouplike $\bE_\infty$-spaces $$(\pic(\caC),D)^{h C_2} \simeq B^\infty(\{\caC\}\times_{B\caC}(B\caC, D \circ (-)^\dual)^{hC_2} $$that sends $\L \in (\Pic(\caC),D)^{h C_2}$ to a $\caC$-linear duality $ \X \mapsto \X^\dual \otimes \L$ on $\caC.$Projection $$(\Pic(\caC),D)^{h C_2} \simeq \{\caC\}\times_{B\caC}(B\caC, D \circ (-)^\dual)^{hC_2} \to (B\caC, D \circ (-)^\dual)^{hC_2} $$induces an equivalence $$\Omega^\infty((\pic(\caC),D)^{hC_2} /\pic(\caC)) \to ((B\caC,D \circ (-)^\dual)^{hC_2})^\simeq. $$
%$$\{\caC\}\times_{B\caC}(B\caC, D \circ (-)^\dual)^{hC_2} \simeq  (\Pic(\caC),D)^{h C_2}, $$
%where the inverse of the latter sends $\L \in (\Pic(\caC),D)^{h C_2}$ to a $\caC$-linear duality $ \X \mapsto \X^\dual \otimes \L$ on $\caC.$
	
\end{proposition}
\begin{proof}[Proof of Proposition \ref{htewegf}]
%We use Remark \ref{ppp}. 
There is an exact sequence of spectra with $C_2$-action, where $\mu$ is the $C_2$-equivariant codiagonal map:
$$\widetilde{(\pic(\caC) \times (\pic(\caC)}  \xrightarrow{\mu} (\pic(\caC), D \circ (-)_\dual) \to (\pic(\caC)[1], D \circ (-)_\dual \circ (-)^\dual) $$$$ \simeq (\pic(\caC)[1], D),$$
where the last equivalence is by Corollary \ref{kui}.
%$$ \widetilde{(\pic(\caC), D)[1] \times (\pic(\caC), D)[1]}.$$ 
Taking homotopy $C_2$-fixed points gives an exact sequence of spectra
$$\pic(\caC) \to (\pic(\caC),D \circ (-)_\dual)^{h C_2}\to (\pic(\caC)[1],D)^{hC_2}.$$
%and so an exact sequence$$\pic(\caC) \to (\pic(\caC),D)^{h C_2} \to (\pic(\caC)[1],D \circ (-)^\dual)^{hC_2} \to \pic(\caC)[1]. $$
	
%(2): There is an exact sequence of spectra with $C_2$-action:$$ (\pic(\caC), D \circ (-)_\dual)\to (\pic(\caC)[1], D) \to\widetilde{(\pic(\caC)[1], D) \times (\pic(\caC)[1], D)}.$$
%Taking homotopy $C_2$-fixed points gives an exact sequence of spectra$$(\pic(\caC),D \circ (-)_\dual)^{hC_2} \to (\pic(\caC)[1],D)^{h C_2} \to \pic(\caC)[1].$$We conclude via Corollary \ref{kui}.
	
\end{proof}
We obtain the following corollary:
\begin{corollary}\label{hoy}Let $\caC$ be a rigid symmetric monoidal $\infty$-category.
There is an exact sequence of spectra:
$$\pic(\caC) \to \pic(\caC)^{B C_2} \to (\pic(\caC)[1],(-)^\dual)^{hC_2} \to \pic(\caC)[1]. $$

%Equip $\caC$ with the duality coming from rigidity provided by Corollary \ref{coukl}.	
%There is a canonical equivalence of grouplike $\bE_\infty$-spaces $$ \Pic(\caC)^{B C_2} \simeq \{\caC\}\times_{B\caC}(B\caC,(-)^\dual)^{hC_2} $$sending $\L \in \Pic(\caC)^{B C_2}$ to a $\caC$-linear duality $ \X \mapsto \X^\dual \otimes \L$ on $\caC.$Projection $$ \Pic(\caC)^{B C_2} \simeq \{\caC\}\times_{B\caC}(B\caC,(-)^\dual)^{hC_2} \to ((B\caC,(-)^\dual)^{hC_2})^\simeq$$induces an equivalence of grouplike $\bE_\infty$-spaces $$((B\caC,(-)^\dual)^{hC_2})^\simeq \simeq \Omega^\infty(\pic(\caC)^{BC_2} /\pic(\caC)).$$
%$$\{\caC\}\times_{B\caC}(B\caC,(-)^\dual)^{hC_2} \simeq \Pic(\caC)^{B C_2}, $$where the inverse of the latter sends $\L \in \Pic(\caC)^{B C_2}$ to a $\caC$-linear duality $ \X \mapsto \X^\dual \otimes \L$ on $\caC.$
	
\end{corollary}

\begin{proof}
We apply Proposition \ref{htewegf} to $D=(-)^\dual.$
By Corollary \ref{kui} there is a canonical equivalence $$(\pic(\caC),(-)^\dual \circ (-)_\dual)^{h C_2} \simeq \pic(\caC)^{B C_2}.$$
	
\end{proof}

\begin{construction}\label{ohas}
Let $(\caC,D)$ be a preadditive symmetric monoidal $\infty$-category with duality. By Proposition \ref{htewett} there is a canonical equivalence $$(\Pic(\caC),D \circ (-)_\dual)^{h C_2}  \simeq \{\caC\} \times_{B\Pic(\caC)} (B\Pic(\caC), D)^{hC_2}.$$
The $C_2$-equivariant symmetric monoidal embedding $(B\caC,D) \subset (\Mod_\caC(\Cat_\infty),D) $ induces a $C_2$-equivariant symmetric monoidal embedding $$(B\caC,D) \subset (\Mod_\caC(\Cat^\preadd_\infty),D).$$
%which gives rise to an equivalence$$\{\caC\} \times_{B\caC} (B\caC, D)^{hC_2} \simeq \{\caC\} \times_{\Mod_\caC(\Cat_\infty^\preadd)} (\Mod_\caC(\Cat_\infty^\preadd),D)^{hC_2}.$$

Let $\KR_\caC$ be the composition of lax symmetric monoidal functors
$$(\Pic(\caC),D \circ (-)_\dual)^{h C_2} \xrightarrow{\alpha} (B\Pic(\caC),D)^{hC_2} \subset (B\caC,D)^{hC_2}\subset (\Mod_\caC(\Cat^\preadd_\infty),D)^{hC_2}$$$$\xrightarrow{\text{forget}} (\Cat_\infty^\preadd,(-)^\op)^{hC_2} \xrightarrow{\KR}  \Sp^{C_2}.$$

\end{construction}

\begin{example}
The embedding $\caP(S) \subset \Perf(S)$ of symmetric monoidal $\infty$-categories
gives rise to an embedding $$\Pic(\caP(S))^{BC_2} \subset \Pic(\Perf(S))^{BC_2} = \Pic(S)^{BC_2}.$$
The object $(S, (-)^\dual) \in \Pic(S)[C_2]= \Pic(S)^{BC_2}$
lands in the full subspace $\Pic(\caP(S))^{BC_2}$ since $S$ lies in $\caP(S).$
The functor %$\pic(S)^{BC_2}$	Let $A = \integers, S, S[\frac{1}{2}]$, let $A \to R$ a map of $\bE_\infty$-ring spectra and $\caC=\Perf(R).$
$$\Pic(\caP(S))^{B C_2} \to \Pic(\caP(R))^{B C_2} \xrightarrow{\KR_{\caP(R)}} \Sp^{C_2}$$
sends $(S, (-)^\dual)$ to a version of symplectic K-theory.
	
\end{example}

\begin{notation}Let $(\caC,D)$ be a symmetric monoidal $\infty$-category with duality.
The $C_2$-equivariant codiagonal map
$$\mu: \widetilde{(\pic(\caC), D\circ (-)_\dual) \times (\pic(\caC), D\circ (-)_\dual)} \to (\pic(\caC), D\circ (-)_\dual)$$
induces on homotopy $C_2$-fixed points a map
$$\pic(\caC) \to (\pic(\caC),D\circ (-)_\dual)^{h C_2}.$$	
The latter induces on infinite loops a map of grouplike $\bE_\infty$-spaces:
$$\chi: \Pic(\caC) \to (\Pic(\caC),D\circ (-)_\dual)^{h C_2}$$
that sends $M$ to $\mu(M,D(M)).$

\end{notation}

The defining factorization of $\KR_\caC$ gives the following corollary:

\begin{corollary}
Let $(\caC,D)$ be a preadditive symmetric monoidal $\infty$-category with duality,
$\L \in (\Pic(\caC),D\circ (-)_\dual)^{h C_2}$ and  $\M \in \Pic(\caC)$. There is an equivalence
$$\KR_\caC(\L) \simeq \KR_\caC(\L \otimes \chi(M))$$
of $\KR(\caC)$-module spectra. %, where $\widetilde{(-)}$ indicates the canonical fixed point datum.
\end{corollary}

\begin{proof}Let $\alpha: (\Pic(\caC),D\circ (-)_\dual)^{h C_2} \to (B\Pic(\caC),D)^{hC_2} $
be the map of  grouplike $\bE_\infty$-spaces of Notation \ref{ohas}.
By the defining factorization of $\KR_\caC$ it is enough to see that
$\alpha(\L)$ is equivalent to $\alpha(\L\otimes\chi(M))= \alpha(L)\otimes_{\Pic(\caC)}\alpha(\chi(M)).$
This follows since $\alpha(\chi(M))\simeq \Pic(\caC)$ because the composition 
$$\Pic(\caC) \xrightarrow{\chi} (\Pic(\caC),D\circ (-)_\dual)^{h C_2} \xrightarrow{\alpha} (B\Pic(\caC),D)^{hC_2} $$
of maps of grouplike $\bE_\infty$-spaces is the zero map by Proposition \ref{htewegf}.

\end{proof}

In the following we apply Corollary \ref{hoy} to make the following computation:
%compute small homotopy groups of $B\Pic(H\integers)^{hC_2} $ and $B\Pic(S[\frac{1}{2}])^{hC_2}$.

\begin{proposition}\label{comp}

\begin{enumerate}
\item We have 	
%$((B\Mod^\perf_{H\integers})^{hC_2})^\simeq $ and 
$$\pi_\ell(\pic(\integers)[1]^{hC_2})= \begin{cases} \integers/4\integers, \ell=0 \\
\integers/2\integers, \ell=1,2, \\
0, \ell> 2.\\	 
\end{cases}$$
	
\item We have $$\pi_\ell(\pic(S[\frac{1}{2}])[1]^{hC_2} )= \begin{cases} \integers/4\integers, \ell=0 \\
\integers/2\integers \oplus \integers, \ell=1,2.\\	 
\end{cases}$$
\end{enumerate}
\end{proposition}

We need the following lemma:

\begin{lemma}\label{aros}
\begin{enumerate}
\item We have 	
%$((B\Mod^\perf_{H\integers})^{hC_2})^\simeq $ and 
$$\pi_\ell(\pic(\integers)^{BC_2})= \begin{cases} \integers/2\integers \oplus \integers, \ell=0 \\
\integers/2\integers, \ell=1, \\
0, \ell> 1.\\	 
\end{cases}$$
	
\item For $\ell=0,1$ we have $$\pi_\ell(\pic(S[\frac{1}{2}])^{BC_2})= \integers/2\integers \oplus \integers.$$
\end{enumerate}	
	
\end{lemma}

\begin{proof}In the following we use the well-known facts about group cohomology: $$\pi_\ell(H(\integers/2\integers)^{BC_2})=\begin{cases}\integers/2\integers, \ell \leq0 \\ 0, \ell > 0,
\end{cases} \pi_\ell(H(\integers)^{BC_2})=\begin{cases}0, \ell =-1,\\ \integers, \ell = 0 \\ 0, \ell > 0
\end{cases}.$$

(1): % By Corollary \ref{hoy} we find that 
%$\pi_0((B\Mod^\perf_{S[\frac{1}{2}]})^{hC_2})=\pi_0((B\Mod^\perf_{S[\frac{1}{2}]})^{hC_2})/ $ and $((B\Mod^\perf_{S[\frac{1}{2}]})^{hC_2})^\simeq $.
%We first compute $\pi_\ell(\pic(\integers)^{BC_2})$ for $\ell \geq 0.$
Since $\pic(\integers)$ is $1 $-truncated, also $\pic(\integers)^{BC_2}$ is $1 $-truncated.
So it remains to compute $\pi_\ell(\pic(\integers)^{BC_2})$ for $\ell=0,1.$
%We compute $\pi_0(\pic(\integers)^{BC_2})= \integers \oplus \integers/2 \integers.$
Since $\pi_0(\pic(\integers))=\integers$, $\pi_1(\pic(\integers))=\integers/2 \integers$
and $\pi_\ell(\pic(\integers))=0$ for every $\ell > 1$, the canonical map $\pic(\integers) \to \pic(\integers)_{\leq1} $ is an equivalence and $ \pic(\integers)_{\leq0}=H(\integers) $ and there is a fiber sequence $$ H(\integers/2\integers)[1] \to \pic(\integers) \to H(\integers).$$	
This fiber sequence induces a fiber sequence 
$$ H(\integers/2\integers)^{BC_2}[1] \to \pic(\integers)^{BC_2} \to H(\integers)^{BC_2},$$	
which yields on homotopy groups for $\ell \geq 0$ an exact sequence
$$\pi_{\ell+1}(H(\integers)^{BC_2})\to \pi_{\ell-1}(H(\integers/2\integers)^{BC_2}) \to \pi_\ell(\pic(\integers)^{BC_2}) $$$$\to \pi_\ell(H(\integers)^{BC_2}) \to \pi_{\ell-2}(H(\integers/2\integers)^{BC_2}).$$
%The latter identifies with an exact sequence$$\pi_{\ell+1}(H(\integers)^{BC_2})\to \integers/2\integers \to \pi_\ell(\pic(\integers)^{BC_2})\to \pi_\ell(H(\integers)^{BC_2}) \xrightarrow{\alpha} \integers/2\integers.$$		
The map $\integers = \pi_0(\pic(\integers)) \xrightarrow{\pi_0(\delta)} \pi_0(\pic(\integers)^{BC_2}) \xrightarrow{\alpha} \pi_0(H(\integers)^{BC_2})=\integers$ is the identity, where $\delta$ is the diagonal map. Thus $\alpha$ admits a canonical section.

For $\ell=0$ we obtain an exact sequence
$$0\xrightarrow{0} \integers/2\integers \to \pi_0(\pic(\integers)^{BC_2})\xrightarrow{\alpha} \integers \xrightarrow{0} \integers/2\integers$$	
and so the exact sequence $$0\xrightarrow{0} \integers/2\integers \to \pi_0(\pic(\integers)^{BC_2})\xrightarrow{\alpha} \integers\to 0 $$
is split via $\pi_0(\delta).$ So $\pi_0(\pic(\integers)^{BC_2})=\integers/2\integers\oplus \integers.$
%The resulting short exact sequence $ \integers/2\integers \to \pi_0(\pic(\integers)^{BC_2})\to \integers$ is split so that $\pi_0(\pic(\integers)^{BC_2})\cong \integers \oplus \integers/2\integers .$
For $\ell =1 $ we obtain an exact sequence
$$0 \to \integers/2\integers \to \pi_1(\pic(\integers)^{BC_2})\to 0 \to \integers/2\integers$$		
so that $\pi_1(\pic(\integers)^{BC_2}) \cong  \integers/2\integers.$
%For odd $\ell \geq 0 $ we obtain an exact sequence$$0\to \integers/2\integers \to \pi_\ell(\pic(\integers)^{BC_2})\to  \integers/2\integers \xrightarrow{\id} \integers/2\integers$$		so that $\pi_\ell(\pic(\integers)^{BC_2}) \cong  \integers/2\integers.$

(2): Recall that  $$ \pi_1(\pic(S[1/2]))= \pi_0(S[1/2])^\times=\integers[1/2]^\times\cong \integers/2	\integers \oplus \integers.$$
Moreover note that the map $\pic(S[1/2])_{\geq 2}^{BC_2}\to \pic(S[1/2])_{\geq 2}$ evaluating at the unique object of $BC_2$ is an equivalence since all homotopy groups of $ \pic(S[1/2])_{\geq 2}$ are uniquely divisible by 2.
%$\pic(S[1/2])$ has $\pi_1: Z+Z/2Z, \pi_0 is Z.$	
So the exact sequence	
$$\pic(S[1/2])_{\geq 2}	\to \pic(S[1/2])_{\geq 1} \to H(\pi_1(\pic(S[1/2]))[1]$$
induces an exact sequence 	
$$\pic(S[1/2])_{\geq 2}	\to (\pic(S[1/2])_{\geq 1})^{BC_2} \to H(\pi_1(\pic(S[1/2]))^{BC_2}[1]\simeq $$$$H(\integers/2	\integers \oplus \integers)^{BC_2}[1]$$
The latter gives rise to a long exact sequence
$$0=\pi_1(\pic(S[1/2])_{\geq 2}) \to \pi_1((\pic(S[1/2])_{\geq 1})^{BC_2}) \to \pi_1(H(\integers/2	\integers \oplus \integers)^{BC_2}[1]) \to $$
$$0=\pi_0(\pic(S[1/2])_{\geq 2}) \to \pi_0((\pic(S[1/2])_{\geq 1})^{BC_2}) \to \pi_0(H(\integers/2	\integers \oplus \integers)^{BC_2}[1]) \to 0.$$
%which identifies with the long exact sequence$$0 \to \pi_1(\pic(S[1/2])_{\geq 1}^{BC_2}) \to \pi_0(H(\integers/2	\integers \oplus \integers)^{BC_2}) \to $$$$0 \to \pi_0(\pic(S[1/2])_{\geq 1}^{BC_2}) \to \pi_{-1}(H(\integers/2\integers \oplus \integers)^{BC_2}) \to 0.$$
Thus $$\pi_1((\pic(S[1/2])_{\geq 1})^{BC_2}) \cong \pi_0(H(\integers/2	\integers \oplus \integers)^{BC_2})=\integers/2	\integers \oplus \integers,$$$$
\pi_0((\pic(S[1/2])_{\geq 1})^{BC_2}) \cong \pi_{-1}(H(\integers/2	\integers \oplus \integers)^{BC_2})=\integers/2	\integers.$$
The long exact sequence	
$$ \pic(S[1/2])_{\geq 1} \to \pic(S[1/2])\to H(\integers)$$ gives rise to an exact sequence	
$$ (\pic(S[1/2])_{\geq 1})^{BC_2}\to  \pic(S[1/2])^{BC_2} \to H(\integers)^{BC_2},$$ 
which gives rise to a long exact sequence
$$0\to \pi_1((\pic(S[1/2])_{\geq 1})^{BC_2}) \to \pi_1(\pic(S[1/2])^{BC_2}) \to \pi_1(H(\integers)^{BC_2}) \to $$
$$\pi_0((\pic(S[1/2])_{\geq 1})^{BC_2}) \to \pi_0(\pic(S[1/2])^{BC_2}) \to \pi_0(H( \integers)^{BC_2}) \to 0.$$
The latter identifies with the long exact sequence
$$0\to\pi_1((\pic(S[1/2])_{\geq 1})^{BC_2}) \to \pi_1(\pic(S[1/2])^{BC_2}) \to 0 \to $$
$$\pi_0((\pic(S[1/2])_{\geq 1})^{BC_2}) \to \pi_0(\pic(S[1/2])^{BC_2}) \to \integers \to 0.$$
Hence we obtain $$\pi_1(\pic(S[1/2])^{BC_2}) \cong \pi_1((\pic(S[1/2])_{\geq 1})^{BC_2}) \cong \integers/2	\integers \oplus \integers,$$
$$\pi_0(\pic(S[1/2])^{BC_2}) \cong\integers/2	\integers \oplus \integers.$$

\end{proof}

\begin{proof}[Proof of Proposition \ref{comp}]

(1): By Corollary \ref{hoy} there is an exact sequence $$\pic(\integers) \to \pic(\integers)^{BC_2}\to \pic(\integers)[1]^{hC_2},$$ which gives rise to a long exact sequence $$0 \to \pi_2(\pic(\integers)[1]^{hC_2}) \to  \pi_1(\pic(\integers)) \to  \pi_1(\pic(\integers)^{BC_2})\to \pi_1(\pic(\integers)[1]^{hC_2}) $$$$\to  \pi_0(\pic(\integers)) \to  \pi_0(\pic(\integers)^{BC_2}) \to  \pi_0(\pic(\integers)[1]^{hC_2}) \to 0.$$
By Lemma \ref{aros} the latter identifies with the long exact sequence
$$0 \to \pi_2(\pic(\integers)[1]^{hC_2}) \to \integers/2 \integers \xrightarrow{0}  \integers/2 \integers \to \pi_1(\pic(\integers)[1]^{hC_2}) $$$$\to \integers \xrightarrow{(2,1)} \integers \oplus \integers/2 \integers  \to \pi_0(\pic(\integers)[1]^{hC_2}) \to 0.$$
Since the map $(2,1): \integers \to \integers \oplus \integers/2\integers$ is injective,
we obtain that $$\pi_0(\pic(\integers)[1]^{hC_2})=\integers \oplus \integers/2 \integers/(2,1)\integers\cong \integers/4\integers, $$
$$\pi_1(\pic(\integers)[1]^{hC_2})= \pi_2(\pic(\integers)[1]^{hC_2})=\integers/2 \integers.$$

% We first compute $\pi_\ell(\pic(S[1/2])^{BC_2})$ for $\ell=0,1.$
(2): By Corollary \ref{hoy} there is an exact sequence $$\pic(S[1/2]) \to \pic(S[1/2])^{BC_2}\to \pic(S[1/2])[1]^{hC_2},$$ which gives rise to a long exact sequence $$0 \to \pi_2(\pic(S[1/2])[1]^{hC_2}) \to  \pi_1(\pic(S[1/2])) \to  \pi_1(\pic(S[1/2])^{BC_2})\to$$$$ \pi_1(\pic(S[1/2])[1]^{hC_2}) \to  \pi_0(\pic(S[1/2])) \to $$$$ \pi_0(\pic(S[1/2])^{BC_2}) \to  \pi_0(\pic(S[1/2])[1]^{hC_2}) \to 0.$$
By Lemma \ref{aros} the latter identifies with the long exact sequence
$$0 \to \pi_2(\pic(S[1/2])[1]^{hC_2}) \to \integers \oplus \integers/2 \integers \xrightarrow{0}  \integers \oplus \integers/2 \integers \to \pi_1(\pic(S[1/2])[1]^{hC_2}) $$$$\to \integers \xrightarrow{(2,1)} \integers \oplus \integers/2 \integers  \to \pi_0(\pic(S[1/2])[1]^{hC_2}) \to 0.$$
Since the map $(2,1): \integers \to \integers \oplus \integers/2\integers$ is injective,
we obtain that $$\pi_0(\pic(S[1/2])[1]^{hC_2})=\integers \oplus \integers/2 \integers/(2,1)\integers\cong \integers/4\integers, $$
$$\pi_1(\pic(S[1/2])[1]^{hC_2})= \integers \oplus \integers/2 \integers.$$
$$\pi_2(\pic(S[1/2])[1]^{hC_2})= \integers \oplus \integers/2 \integers.$$

\end{proof}

%In the following we use that for every symmetric monoidal additive $\infty$-category $\caD$ the full subspace of $ \caD[C_2]$ spanned by the $C_2$-actions on the tensor unit 1 of $\caD$ carries the structure of a grouplike $\bE_\infty$-space sincethis full subspace is canonically equivalent to the cotensor $ B(\caD(1,1))^{BC_2}$and $\caD(1,1)$ is a grouplike $\bE_\infty$-space.

%\begin{remark}\label{notabene}Let $\caD$ be a stable $\infty$-category % that admits finite limits and a zero objectand $\X\in \caD.$ %an object of $\caD$ with $C_2$-action.Then $\X'' \simeq \widetilde{\Omega}(\X[1]).$ %is the fiber of the $C_2$-equivariant diagonal morphism $\X[1] \to \widetilde{\X[1]\times\X[1]}$.
	
%We set $\Omega(\X)^{\tilde{h}C_2}:= \tilde{\Omega}(\X)^{hC_2}.$\end{remark}	

\section{Dualities on the $\infty$-category of spectra}\label{j5r5z5z}

In the following we prove that $\pi_0(\pic(S)^{BC_2})$ contains $ \integers \oplus \integers_2, $ where $\integers_2$ is the ring of 2-adic numbers.
By Lin's Theorem \cite[Theorem 1.1]{lin} there is an isomorphism 
$ \pi_0(S^{BC_2}) \cong \integers \oplus \integers_2$
and we prove that the canonical map of spectra $S \to \pic(S)$ that sends the
unit to $S[1]$ induces an injection (Proposition \ref{hteetr5}):
$$\pi_0(S^{BC_2}) \to \pi_0(\pic(S)^{BC_2}).$$

\begin{remark}\label{split}
Let $\caD$ be a stable $\infty$-category that admits small limits and $\X \in \caD.$	
Choosing a base point in the space $B C_2$,
the retraction $* \to B C_2 \to *$ of spaces induces a retraction
$\X \to \X^{B C_2} \to \X$ in $\caD.$
Thus the cofiber sequence of pointed spaces $(*)_+ \to (BC_2)_+ \to BC_2$ induces via the cotensoring of $\caS_*$ on $\caD$ a split exact sequence $\X^{BC_2}_* \to \X^{B C_2} \to \X$, where the first object is the pointed cotensor. %So the canonical map$\X \oplus \X^{BC_2}_* \to \X^{BC_2}$ is an equivalence.
In particular, the sequence $$(\X,(-)_\dual)[-1] \to (\X,\triv) \to \widetilde{\X\times\X}$$ in $\caD[C_2], $ where the latter map is the diagonal, induces a split exact sequence $$(\X,(-)_\dual)^{hC_2}[-1] \to \X^{B C_2} \to \X$$ in $\caD$,
which identifies $(\X,(-)_\dual)^{hC_2}[-1] \simeq\X^{BC_2}_*.$
\end{remark}

\begin{remark}\label{remu}
By \cite[Theorem 1.1]{lin} there is an isomorphism 
$  \pi_0(S^{BC_2}) \cong \integers \oplus \integers_2$.
The split fiber sequence $$ (S,(-)_\dual)^{hC_2}[-1] \to S^{BC_2} \to S$$
of Remark \ref{split} induces a short exact sequence $$0 \to \pi_0( (S,(-)_\dual)^{hC_2}[-1]) \to \pi_0(S^{BC_2})\cong \integers \oplus \integers_2 \to \pi_0(S) \simeq \integers \to 0,$$ where the last map identifies with projection.
Thus $ \pi_0( (S,(-)_\dual)^{hC_2}[-1]) \cong \integers_2.$
\end{remark}

\begin{notation}
%Let $\caD$ be an additive $\infty$-category that admits small limits and $X \in \caD$. 	
Let $X$ be a spectrum. The $C_2$-equivariant codiagonal map $\mu: \widetilde {X \oplus X} \to X$ induces on homotopy $C_2$-fixed points a map of spectra $$ \tau_X: X \to X^{B C_2} \simeq X \oplus (X,(-)_\dual)^{hC_2}[-1], \Y \mapsto \mu(\Y,\Y).$$
%For any point $\Z $ of $\Omega^\infty(X)$ 
We write $\theta_X $ for the map
$ X \to X^{B C_2} \to (X,(-)_\dual)^{hC_2}[-1]$.
	
\end{notation}

\begin{lemma}\label{lemaas}
	
The element $\theta_S(1)$ of $\pi_0((S,(-)_\dual)^{hC_2}[-1]) \cong \integers_2$
is a free generator of the $\integers_2$-module $\pi_0((S,(-)_\dual)^{hC_2}[-1])$, where $1 \in \pi_0(S) \cong \integers$ is the unit.	
	
\end{lemma}

\begin{proof}
	
For any spectrum $X$ there is a natural equivalence of grouplike $\bE_\infty$-spaces
$$ \Omega^\infty((X,(-)_\dual)^{hC_2}[-1]) \simeq \caS_*(BC_2,\Omega^\infty(X)) $$
covering the natural equivalence of grouplike $\bE_\infty$-spaces
$$ \Omega^\infty(X^{BC_2}) \simeq {\caS}(BC_2,\Omega^\infty(X)). $$
%\simeq
%{\Mon_{A_\infty}(\caS)}(C_2,\Omega(\Omega^\infty(X))),$$
under which for any $Z \in \Omega^\infty(X) $
the object $\theta_X(Z) \in \Omega^\infty((X,(-)_\dual)^{hC_2}[-1]) $ corresponds to a pointed map $\alpha^Z_X: $ %of $A_\infty$-spaces
$\ BC_2 \to \Omega^\infty(X)$.

We first show that $\pi_1(\alpha^1_S): \pi_1(B C_2) \cong C_2 \to \pi_1(S) \cong C_2$
is an isomorphism.
%remark that $\theta_S(1)$ is a free generator of the $\integers_2$-module $\pi_0((S, (-)_\dual)^{hC_2}[-1])$ if 
%The $C_2$-equivariant multiplication map $\widetilde {\pic(S) \times \pic(S)} \to \pic(S) $ induces on homotopy $C_2$-fixed points a map$$ \pic(S) \to \pic(S)^{B C_2} \simeq \pic(S) \oplus \tilde{\Omega}(\pic(S)^{hC_2}), \Y \mapsto \Y^{2}.$$We write $\theta_{\Pic(S)}(S[1])$ for the image of $S[1]$ under the map$$ \pi_0(\pic(S)) \xrightarrow{\delta} \pi_0(\pic(S)^{B C_2}) \cong \pi_0(\Pic(\Sp[C_2])) \xrightarrow{\gamma} \pi_0((\pic(S), (-)_\dual)^{hC_2}[-1]).$$
By naturality the map $\alpha^{S[1]}_{\pic(S)}$ factors as $B C_2 \xrightarrow{\alpha_S^1} \Omega^\infty(S) \to \Omega^\infty(\pic(S)).$ 
Thus $\pi_1(\alpha^{S[1]}_{\pic(S)})$ factors as $$B C_2 \xrightarrow{\pi_1(\alpha_S^1)} \pi_1(S) \cong C_2 \to \pi_1(\pic(S)) \cong 
\pi_0(S)^\times \cong \integers^\times \cong C_2, $$ where the last map is the identity.
Since $\tau_{\pic(S)}(S[1]) \simeq S^{1,1\sigma} \in \pi_0(\pic(S)^{B C_2}) \cong \pi_0(\Pic(S)[C_2]),$
the map $\pi_1(\alpha^{S[1]}_{\pic(S)})$ sends the generator of $C_2$
to the action map $S^2 \to S^2$ provided by $S^{1,1\sigma}.$
So $\alpha^1_S$ induces the zero map on fundamental groups.

By Remark \ref{remu} there is a canonical isomorphism  $ \integers_2 \cong \pi_0((S,(-)_\dual)^{hC_2}[-1]).$
Let $g$ the image of $1$ under this isomorphism.
There is a group homomorphism
$$\rho: \integers_2 \cong \pi_0((S,(-)_\dual)^{hC_2}[-1]) \cong \pi_0({\caS_*}(BC_2,\Omega^\infty(S))) \xrightarrow{\pi_1(-)} \hom(C_2,C_2).$$
If $\theta_S(1)$ is no free generator of the $\integers_2$-module $\pi_0((S,(-)_\dual)^{hC_2}[-1])$, then $\theta_S(1)$ lies in the maximal ideal of the local ring $\integers_2$ so that $\theta_S(1) = 2 \lambda g $ for some $\lambda \in \integers_2.$
Thus $\rho$ sends $\theta_S(1)= 2 \lambda g$ to a multiple of 2, which is zero.
In other words $\alpha^1_S$
%So it will be enough to see that $ 2 \ell \lambda \in \integers_2 \cong\pi_0((S, (-)_\dual)^{hC_2}[-1]) \cong \pi_0(\map_{\caS_*}(BC_2,\Omega^\infty(X)))$corresponds to a pointed map $BC_2 \to \Omega^\infty(X)$ that 
induces the zero map on fundamental groups.

\end{proof}

\begin{notation}
The map of spectra $S \to \pic(S)$ sending $1$ to the element $S[1]$ 
induces a homomorphism 
$$\xi: \integers \oplus \integers_2 \cong \pi_0(S^{BC_2}) \to \pi_0(\pic(S)^{BC_2}) \cong \pi_0(\Pic(\Sp[C_2])).$$
\end{notation}

The functor $\nu: \Sp^{C_2} \to \Sp[C_2]$
that forgets along the functor $BC_2 \subset \Fin[C_2] \to \Span(\Fin[C_2])$
admits a fully faithul right adjoint $ \iota: \Sp[C_2] \to \Sp^{C_2}$ \cite[Proposition 6.17.]{zbMATH06652659} such that $b:= \iota\circ \nu$, which we call Borel-completion, is $\Sp^{C_2}$-linear.

\begin{lemma}\label{leomort}
The restriction 
$$\xi': \integers \oplus \integers \subset \integers \oplus \integers_2 \cong \pi_0(S^{BC_2}) \to \pi_0(\pic(S)^{BC_2})\cong \pi_0(\Pic(\Sp[C_2]))$$
sends $(n,m)$ to $\nu(S^{(n-m)+m\sigma}),$ where $\sigma$ is the sign representation.
\end{lemma}

\begin{proof}
	
The restriction of $\xi'$ to the first summand factors as
$$\integers \cong \pi_0(S) \to \pi_0(S^{BC_2}) \to \pi_0(\pic(S)^{BC_2})\cong \pi_0(\Pic(\Sp[C_2])),$$
which is the map 
$$\integers \cong \pi_0(S) \cong \pi_0(\pic(S))\cong \pi_0(\Pic(\Sp)) \to \pi_0(\pic(S)^{BC_2})\cong \pi_0(\Pic(\Sp[C_2])) $$
sending $n \in \integers $ to $S^n \simeq \nu(S^{n+0\sigma}).$

The composition $$ \pi_0(\pic(S)^{B C_2}) \xrightarrow{\gamma}  \pi_0((\pic(S),(-)_\dual)^{hC_2}[-1]) \hookrightarrow \pi_0(\pic(S)^{B C_2})
$$ is the map $\id - \delta p, $ where $p$ is the forgetful functor
$\pi_0(\pic(S)^{B C_2}) \to \pi_0(\pic(S))$
and $\delta: \pi_0(\pic(S)) \to \pi_0(\pic(S)^{B C_2})$ is induced by the map
$B C_2 \to *.$

Since $\tau_{\pic(S)}(S[1]) \simeq \nu(S^{1,1\sigma}) \in \pi_0(\pic(S)^{B C_2}),$ the map $\gamma$ sends $\nu(S^{1,1\sigma})$ to $\theta_{\pic(S)}(S[1]).$
Consequently, the image of $\theta_{\pic(S)}(S[1])$ under the inclusion $$\pi_0((\pic(S), (-)_\dual)^{hC_2}[-1]) \hookrightarrow \pi_0(\pic(S)^{BC_2})\cong \pi_0(\Pic(\Sp[C_2]))$$ is $\nu(S^{1,1\sigma} \wedge S^{-2}) \simeq \nu(S^{-1,1\sigma}).$
The restriction of $\xi'$ to the second summand factors as
$$\integers \subset \integers_2 \cong \pi_0((S, (-)_\dual)^{hC_2}[-1]) \to \pi_0((\pic(S), (-)_\dual)^{hC_2}[-1]) $$$$\hookrightarrow \pi_0(\pic(S)^{BC_2})\cong \pi_0(\Pic(\Sp[C_2])).$$
By Lemma \ref{lemaas} the map $\integers \subset \integers_2 \cong \pi_0((S, (-)_\dual)^{hC_2}[-1])$ sends 1 to $\theta_S(1)$, which is sent by the map
$\pi_0((S, (-)_\dual)^{hC_2}[-1]) \to \pi_0((\pic(S), (-)_\dual)^{hC_2}[-1])$
to the element $\theta_{\pic(S)}(S[1])$. 
Hence the restriction of $\xi'$ to the second summand sends 1 to $\nu(S^{-1,1\sigma}).$

\end{proof}

We learned the proof of the following proposition from Niko Naumann:

\begin{proposition}\label{hteetr5}
The homomorphism $$\xi: \integers \oplus \integers_2 \cong \pi_0(S^{BC_2}) \to \pi_0(\pic(S)^{BC_2})$$ is injective.
\end{proposition}

\begin{proof}
By Remark \ref{split} the homomorphism $\xi$ is the sum $\alpha \oplus \beta$,
where $\alpha: \pi_0(S)\cong \integers \to \pi_0(\pic(S)) \cong \integers$ is
the canonical isomorphism and $$\beta: \integers_2 \cong \pi_0((S, (-)_\dual)^{hC_2}[-1]) \to \pi_0((\pic(S), (-)_\dual)^{hC_2}[-1])$$
is the canonical homomorphism.

The homomorphism $\beta$ is a map of pro-finite abelian groups and thus is injective if the restriction $\beta': \integers \subset \integers_2 \xrightarrow{\beta}\pi_0((\pic(S), (-)_\dual)^{hC_2}[-1])$
is injective: indeed, the kernel of $\beta$ is closed and therefore of the form
$2^\ell \integers_2$ for $\ell \geq 0$ if $\beta$ is not injective.
Hence in this case the kernel of $\beta'$ is of the form $2^\ell \integers$.
Consequently, it is enough to see that the restriction $\xi'=\alpha\oplus \beta'$
of $\xi$ to $\integers \oplus \integers$ is injective.

By Lemma \ref{leomort} the restriction $\xi'$ sends $(k,n)\in \integers \oplus \integers$ to $\nu(S^{(k-n)+n\sigma}).$ 
Let $k,n \in \integers$ such that $\nu(S^{(k-n)+n\sigma}) \simeq \nu(S^{0+0\sigma}).$
Set $m:= k-n.$ We like to see that $n=m=0$ so that also $k=0.$
Let $b: \Sp^{C_2} \to \Sp^{C_2}$ be the Borel completion, which factors as $
\Sp^{C_2} \xrightarrow{\nu} \Sp[C_2] \xrightarrow{\iota}\Sp^{C_2}$
and is $\Sp^{C_2}$-linear.
Let $\widetilde{E}_{C_2} := \colim_\ell S^{0+\ell\sigma}$ be the filtered colimit of genuine $C_2$-spectra.
For any genuine $C_2$-spectrum $X$ the Tate construction
$t(X)$ of $X$ is equivalent to the $C_2$-fixed points of the $C_2$-spectrum
$\widetilde{E}_{C_2} \wedge b(X).$
Since $\nu(S^{m+n\sigma}) \simeq\nu(S^{0+0\sigma}), $ we find that $b(S^{m+n\sigma}) \simeq b(S^{0+0\sigma}).$
We obtain a canonical equivalence
$$ \widetilde{E}_{C_2} \wedge b(S^{m+0\sigma}) \simeq \colim_\ell (S^{0+\ell\sigma} \wedge b(S^{m+0\sigma})) \simeq \colim_\ell b(S^{m+\ell\sigma}) \simeq \colim_\ell b(S^{m+(n+\ell)\sigma}) $$$$ \simeq \colim_\ell (S^{0+\ell\sigma} \wedge b(S^{m+n\sigma})) \simeq \widetilde{E}_{C_2} \wedge b(S^{m+n\sigma}) \simeq
\widetilde{E}_{C_2} \wedge b(S^{0+0\sigma}).$$
Applying $C_2$-fixed points we obtain an equivalence
of spectra $t(S^{0+0\sigma})[m] \simeq t(S^{m+0\sigma}) \simeq t(S^{0+0\sigma})$.
To deduce that $m=0$ and thus also $n=0$, it is enough to see that $ t(S^{0+0\sigma})$ is non-trivial and bounded below. 

Since the $ t(S^{0+0\sigma})$-module $t(\bF_2)$ is not trivial, $ t(S^{0+0\sigma})$ is not trivial.
To see that $ t(S^{0+0\sigma}) $ is bounded below,
we remark that $ t(S^{0+0\sigma}) $ is connective:
by Lin's Theorem \cite{lin} the spectrum $S^{BC_2}$ is connective.
Thus the spectrum $ t(S^{0+0\sigma}) $ is the cofiber of the following map of connective spectra:
$$(S^{0+0\sigma})_{hC_2} \to (S^{0+0\sigma})^{hC_2} \simeq S^{BC_2}.$$
\end{proof}

We obtain the following corollaries:

\begin{corollary}\label{inco}
The cardinality of the space $\pi_0((\pic(S)[1], (-)^\dual)^{hC_2})$ of dualities 
on the $\infty$-category of dualizable spectra (Lemma \ref{ppp}) is the continuum.
\end{corollary}

\begin{proof}
By Proposition \ref{htewegf} there is an exact sequence
$$ \pic(S) \to \pic(S^{BC_2}) \to (\pic(S)[1], (-)^\dual)^{hC_2}$$
that gives rise to an exact sequence
$$ \pi_0(\pic(S))=\integers \to \pi_0(\pic(S^{BC_2})) \to \pi_0((\pic(S)[1], (-)^\dual)^{hC_2}) \to \pi_{-1}(\pic(S)) = 0.$$
By Proposition \ref{hteetr5} the group $\pi_0(\pic(S^{BC_2})) $ contains the uncountable group $\integers \oplus \integers_2$ and so is uncountable.
Since $\integers$ is countable and $\pi_0(\pic(S^{BC_2}))$ is uncountable, we find that 
the cokernel $\pi_0((\pic(S)[1], (-)^\dual)^{hC_2})$ is uncountable.
\end{proof}

%\begin{notation}Let $(\caC,D)$ be stable $\infty$-category with duality.Let $K=1[1] \in \Pic(\caC)^{hC_2} $.Then  \end{notation}

\begin{corollary}
The element $(S,(-)^\dual) \in \pi_0(\pic(\Sp[C_2])) \cong \pi_0(\Pic(S)^{BC_2})$ has infinite order.
	
\end{corollary}

\begin{proof}
	
By Lemma \ref{leomort} the map 
$$ \integers \oplus \integers_2 \cong \pi_0(S^{BC_2}) \to \pi_0(\pic(S)^{BC_2})\cong \pi_0(\Pic(\Sp[C_2]))$$
sends $(0,1)$ to $\nu(S^{-1+\sigma}),$ where $\sigma$ is the sign representation.
By Proposition \ref{hteetr5} the latter map is injective. The element $(0,1) \in \integers \oplus \integers_2 $ has infinite order so that $\nu(S^{-1+\sigma}) \in \pi_0(\Pic(\Sp[C_2]))$ has infinite order.
Note that $\nu(S^{1+\sigma}) \simeq (\widetilde{S[1]\times S[1]})$ in $\Pic(\Sp[C_2])$.
By Corollary \ref{sosol} we have $$ \nu(S^{-1+\sigma}) \simeq\nu(S^{1+\sigma})[-2] \simeq  (\widetilde{S[1]\times S[1]})[-2]\simeq (S,(-)^\dual)$$ in $\Pic(\Sp[C_2])$.
\end{proof}

\begin{lemma}
Let $n \geq 0$ and $R$ a $n$-truncated $\bE_\infty$-ring spectrum. The order of the element $$(R,(-)^\dual) \in \pi_0(\Pic(\Mod_R[C_2])) \simeq \pi_0(\pic(R)^{BC_2})$$ is of the form $2^\ell$, where $0 \leq \ell \leq n+1.$
	
\end{lemma}

\begin{proof}
The canonical map of $\bE_\infty$-ring spectra $S_{\leq n} \to R$ gives rise to a map
$\pi_0(\Pic(S_{\leq n})^{BC_2}) \to \pi_0((\pic(R)^{BC_2})$ that sends
$(S_{\leq n},(-)^\dual)$ to $(R,(-)^\dual)$.
Consequently, it is enough to prove that the order of the element $$(S_{\leq n},(-)^\dual) \in \pi_0((\pic(S_{\leq n})^{BC_2})$$ is of the form $2^\ell$, where $0 \leq \ell \leq n+1$.
By Remark \ref{split} we have a split exact sequence $$\pi_0((\pic(S_{\leq n})_*^{BC_2}) \to \pi_0((\pic(S_{\leq n})^{BC_2}) \to \pi_0((\pic(S_{\leq n})).$$
%The retraction $* \to B C_2 \to *$ of spaces induces a retraction$(\pic(S_{\leq n}) \to (\pic(S_{\leq n})^{BC_2} \to(\pic(S_{\leq n})$ so that$(\pic(S_{\leq n})^{BC_2} \simeq (\pic(S_{\leq n}) \oplus F$, where $F$ is the fiber of the retraction map.
Since $(S_{\leq n},(-)^\dual) \in \pi_0((\pic(S_{\leq n})^{BC_2})$ lies over the neutral element of the group $\pi_0((\pic(S_{\leq n}))$, it is enough to see that
the image of $(S_{\leq n},(-)^\dual) \in \pi_0((\pic(S_{\leq n})^{BC_2})$ in $\pi_0((\pic(S_{\leq n})_*^{BC_2})$
has order of the form $2^\ell$, where $0 \leq \ell \leq n+1$.
To see this, we prove by induction on $n \geq 0$ that every element of $\pi_0((\pic(S_{\leq n})_*^{BC_2})$ has order
of the form $2^\ell$, where $0 \leq \ell \leq n+1$.
For $n =0$ this follows from the fact that by Lemma \ref{aros} we have that
$\pi_0((\pic(\integers)_*^{BC_2}) \cong \integers/2 \integers$. % and the image of $(\integers,(-)^\dual) \in \pi_0((\pic(\integers)^{BC_2})$ corresponds to $(0,1).$
%We prove this by induction on $n \geq 0.$ $pic(S)_{\geq1}\simeq BGL_1(S)$For $n > 0$ we prove by 

Note that the map $\integers \cong \pi_0(\pic(S)) \to \pi_0(\pic(S_{\leq n})) \to \pi_0(\pic(\integers))\cong \integers $
is the identity so that $\pi_0(\pic(S_{\leq n})) \cong \integers \oplus F$, where $F$ is the kernel of the second map in the latter composition.
Let $\pic'(S_{\leq n})$ be the fiber of the map $$\pic(S_{\leq n}) \to H(\pi_0(\pic(S_{\leq n}))) \to H(F).$$
The long exact sequence on homotopy groups applied to the latter exact sequence gives an an exact sequence $$H(\pi_{n+1}(S))[n+2] \to \pic'(S_{\leq n+1}) \to \pic'(S_{\leq n}).$$
The latter exact sequence yields an exact sequence $$H(\pi_{n+1}(S))[n+2]_*^{BC_2} \to \pic'(S_{\leq n+1})^{BC_2}_*  \to \pic'(S_{\leq n})^{BC_2}_*, $$ which gives rise to an exact sequence
$$\pi_0(H(\pi_{n+1}(S))[n+2]^{BC_2}_*) \to \pi_0(\pic'(S_{\leq n+1})^{BC_2}_*) \to \pi_0(\pic'(S_{\leq n})^{BC_2}_*).$$

The exact sequence $\pic'(S_{\leq n})\to \pic(S_{\leq n}) \to H(F) $
yields an exact sequence $$\pic'(S_{\leq n})_*^{BC_2}\to \pic(S_{\leq n})_*^{BC_2} \to H(F)_*^{BC_2}$$
and so an exact sequence
$$\pi_1(H(F)_*^{BC_2}) \to \pi_0(\pic'(S_{\leq n})_*^{BC_2})\to \pi_0(\pic(S_{\leq n})_*^{BC_2}) \to \pi_0(H(F)_*^{BC_2}).$$
Since $H(F)$ is discrete, also $H(F)_*^{BC_2}$ is discrete so that $\pi_1(H(F)_*^{BC_2})=0.$
The split exact sequence $H(F)_*^{BC_2} \to H(F)^{BC_2} \to H(F)$ gives a split exact sequence $\pi_0(H(F)_*^{BC_2}) \to \pi_0(H(F)^{BC_2}) \to F,$ where the last map is an isomorphism.
So $\pi_0(H(F)_*^{BC_2})=0$ and the map $\pi_0(\pic'(S_{\leq n})_*^{BC_2})\to \pi_0(\pic(S_{\leq n})_*^{BC_2})$ is an isomorphism.
We obtain an exact sequence $$\pi_0(H(\pi_{n+1}(S))[n+2]^{BC_2}_*) \to \pi_0(\pic(S_{\leq n+1})^{BC_2}_*) \to \pi_0(\pic(S_{\leq n})^{BC_2}_*).$$

The abelian group $$\pi_0(H(\pi_{n+1}(S))[n+2]^{BC_2}_*) \cong \pi_{-n-2}(H(\pi_{n+1}(S))_*^{BC_2}) \cong \pi_{-n-2}(H(\pi_{n+1}(S))^{BC_2})$$ is the $n+2$-th group cohomology of an abelian group, which is a $\bF_2$-vector space.
So the claim follows from the fact that for every exact sequence of abelian groups $A \to B \to C$, where every element of $A$ has order $\leq k$ and every element of $B$ has order $\leq r$ we have that every element of $C$ has order $\leq k$ and $\leq r.$
	
\end{proof}%abh-einreiseatlra-starnberg.de

\bibliographystyle{plain}
\bibliography{Herm} 

\end{document}